\documentclass[11pt,twoside]{article}%
\usepackage{amssymb}
\usepackage{amsmath}
\usepackage{amsfonts}
\usepackage{mathtools}
\usepackage[title,titletoc,header]{appendix}
\usepackage{esvect}
\usepackage{geometry}
\usepackage{graphicx}
\usepackage{indentfirst}
\usepackage{mathrsfs}
\usepackage{nopageno}
\usepackage{setspace}
\usepackage{remreset}
\usepackage[nottoc]{tocbibind}
\usepackage{hyperref}
\usepackage{mathtools, amssymb}
\newtheorem{theorem}{Theorem}[section]

\newtheorem{axiom}[theorem]{Axiom}

\newtheorem{conclusion}[theorem]{Conclusion}

\newtheorem{corollary}[theorem]{Corollary}

\newtheorem{definition}[theorem]{Definition}
\newtheorem{example}[theorem]{Example}

\newtheorem{lemma}[theorem]{Lemma}

\newtheorem{problem}[theorem]{Problem}
\newtheorem{proposition}[theorem]{Proposition}
\newtheorem{remark}[theorem]{Remark}

\newenvironment{proof}[1][Proof]{\noindent\textbf{#1.}
	\setlength{\parindent}{1ex} }{\hfill \rule{0.4em}{0.4em}}
\newcommand{\sbin}{\mathrel{\widetilde{\in}}}
\renewcommand{\labelenumi}{(\roman{enumi})}

\makeatletter
\@removefromreset{figure}{section}
\makeatother
\makeatletter
\def\blfootnote{\xdef\@thefnmark{}\@footnotetext}
\makeatother
\counterwithin{equation}{section}

\begin{document}
	
\title{Predicate and Set Bundles in Multi-valued Logic}
\author{Eugene Zhang}
\maketitle

\begin{abstract}\medskip
	In this paper, a new model for multi-valued logic is presented based on the notion of a bundle of predicates. The degree of truth and various logical operations for predicates in our model of multi-valued logic are rigorously defined and investigated. Furthermore, the bundle of sets as a special type of predicate bundles is thoroughly investigated to provide a rigorous model for fuzzy sets and certain adjectives/adverbs in linguistics. In addition, solutions to necessity/possibility and the sorites paradox in modal logic are given. \medskip

\end{abstract}

\tableofcontents
\listoffigures

\section{Introduction}\label{SectionIntroduction}

Logic is a branch of mathematics and philosophy that employs a (formal) language and deductive system to model human thoughts. Classical logic was initially developed by Boole and Frege in the 19th century as an expansion of Aristotelian logic (\cite{Boole} and \cite{Frege}) in which only two truth values are allowed\ for each proposition. However, sufficient evidence suggests that human thoughts can not be modeled purely by classical logic, notably reflected by the fact that linguistic expressions involving adjectives and adverbs generally can not be described by classical logic and sets. For example, \textquotedblleft young man\textquotedblright\ can not be represented by two truth values and a classical set because there is no single clear-cut value in age for \textit{young}. More specifically, despite the fact that a man with an age under 20 is certainly considered as young, age 22 or 23 or whatsoever, can also be regarded as young. The same can be applied to \textquotedblleft very young man\textquotedblright\ too because there are no values to define the adverb \textit{very} in age. Consequently, classical logic is inadequate to describe these linguistic objects and multi-valued logic is called upon.\smallskip

Multi-valued logic started with the work of Lukasiewicz who created a three-valued logic (\textit{true}, \textit{false} and \textit{possible}) in 1920 and later developed into an infinite valued variant with given rules for negation and implication \cite{Gottwald}. Later on, Post and Godel each proposed an infinite valued logics in which rules for conjunction and disjunction are given (\cite{Gottwald} and \cite{Post}). In 1958, Chang introduced an algebraic system with countably many truth values known as MV-algebras \cite{Chang}.\smallskip

Another field of multi-valued logic is the theory of fuzzy set and logic introduced by Zadeh from 1965 (\cite{Zadeh} and \cite{Hajek}). A fuzzy set extends a classical set by assigning each of its elements with a membership degree between 0 and 1. Fuzzy logic allows the truth value of variables to be any real number between 0 and 1. The problem with fuzzy logic is that fuzzy membership and truth lack rigorous definitions and are often under subjectivity and speculation.\smallskip

Furthermore, probability theory studies the uncertainty of random events and resembles multi-valued logic in some sense. Probability logic is proposed in which the truth values of sentences are replaced with probabilities (\cite{Gaifman} and \cite{Scott-Krauss}). Since the range of the truth values in a genuine multi-valued logic should include any real number between 0 and 1, it is reasonable to believe that such a theory should be related to probability (logic). However, a genuine theory of multi-valued logic must be beyond probability because thought is much more complex than random events.\smallskip

In this paper, we will present a new model of multi-valued logic based on the notion of a bundle of predicates. In general, a predicate in formal logic can be thought as a single standard that is satisfied (by a variable) if the predicate is true and not if it is false. As a result, it is evident intuitively that the degree of truth for a predicate in multi-valued logic can be described by the rate of (many) standards it passes (satisfies). Consequently, its degree of truth can be defined as the weight (percentage) of true predicates in a bundle of predicates. This idea is illustrated by the following example. \smallskip

Suppose $p(x)\Leftrightarrow\,$\textquotedblleft he is less than $20$ years old\textquotedblright\ ($x$ is the age of \textquotedblleft he\textquotedblright). Obviously, $p(x)$ is a first-order predicate in that $p(x)\Leftrightarrow (0<x<20)$. But \textquotedblleft he is young\textquotedblright\ can not be modeled by a first-order (or formal) predicate because no single age can define \textit{young}. However, we can employ a bundle of first-order predicates for \textit{young}, each of which carries a weight. So let a predicate bundle for \textit{young} be
\[
\mathcal{C}\,=\,\{p_{n}\colon p_{n}(x)\Leftrightarrow (0<x\leqslant n+19) \,\wedge\,1\leqslant n\leqslant20\,\wedge \,r(p_{n})=1/20\,\}
\]
where $r(p_{n})$ is the weight for each $p_{n}$ in $\mathcal{C}$. We also assume that the sum of all weights is one, i.e. ${\sum_{1\leqslant n\leqslant20}}\,r(p_{n})=1$. As the result, the degree of truth (or belief) for \textquotedblleft he is young\textquotedblright\ can be described by the weight of all true $p_{n}$ in $\mathcal{C}$, i.e. $T(x)\,=\,\sum r(p_{n}(x))$.\smallskip

For any $x\leqslant20,$ $x$ satisfies all $p_{n}$ in $\mathcal{C}$ and so $T(x)=1$, i.e. any age less than $20$ has one hundred percent belief of \textit{young}. On the other hand, for any $x\geqslant40$,\ $x$ satisfies no $p_{n}$ in $\mathcal{C}$ and so $T(x)=0$, i.e. any age over $40$ has zero percent belief of \textit{young}. If $20<x<40$, it has a belief of \textit{young} between zero and one. For instance, $x=25$ satisfies three fourths of $p_n\in\mathcal{C}$ ($6\leqslant n\leqslant20$) and so $T(25)={\sum_{6\leqslant n\leqslant20}}\,r(p_{n})=0.75$. Furthermore, $T(x)$ is monotonically decreasing which is consistent with common sense that a smaller age always has a higher belief of \textit{young} than a bigger age. \smallskip

This precise definition of the truth degree can also be used to rigorously define the logical operations such as conjunction, disjunction and implication for predicates in multi-valued logic. In addition, a special case of predicate bundles known as set bundles is introduced to provide a rigorous foundation for fuzzy sets. Furthermore, a rigorous model of modal logic based on predicate bundles is given and shown to be an improvement over the Kripke system (section \ref{SectionModalLogic}). Finally, a complete solution to the sorites paradox is given based on predicate bundles and degree of truth (conclusion \ref{38}).\smallskip

\section{Predicate Bundles for Multi-Valued Logic}

In this section, we will introduce a new model for multi-valued logic based on the notion of predicate bundles. First, we introduce the weight measure for a predicate bundle.  

\subsection{Weight Measure and Predicate Bundle}

In this paper, we will consider an infinitary language of $\mathcal{L}_{\omega_1,\omega}$ that has logical connectives $\wedge, \vee, \urcorner, \Rightarrow, \Leftrightarrow, \forall, \exists,\,\bigwedge, \bigvee$ for finite conjunction and disjunction, negation, implication, equivalence, universal and existential quantification, infinite conjunction and disjunction, as well as $=, <, >$ for equality and inequality. A predicate bundle is a collection of predicates in $\mathcal{L}_{\omega_1,\omega}$ equipped with a weight measure on each predicate in the bundle.  

\begin{definition}
	\label{Sigma_field}	\ Suppose $\mathscr{P}$ is a collection of predicates (up to equivalence) in $\mathcal{L}_{\omega_1,\omega}$.\footnote{Predicates are formulas with free variables. Notice that predicates in $\mathscr{P}$ can be inconsistent.} A \textbf{$\sigma-$field} on $\mathscr{P}$ is a nonempty collection of subsets of $\mathscr{P}$ closed under complement and countable unions. 
\end{definition}

\begin{axiom}
	\label{10} \ Suppose $\mathscr{P}$ is a collection of predicates in $\mathcal{L}_{\omega_1,\omega}$ and $\sigma(\mathscr{P})$ is a $\sigma-$field on $\mathscr{P}$. The \textbf{weight measure} $r$ on $\mathscr{P}$ is a real-valued function on $\sigma(\mathscr{P})$ that satisfies the following conditions.
\end{axiom}

\begin{enumerate}
	\item \textit{For any} $p\in\mathscr{P},$ $\,r(p)\geqslant0$.
	
	\item $r(\mathscr{P})=1$.
	
	\item \textit{For any (countably many) disjoint} $\mathscr{P}_{i}\in\sigma(\mathscr{P})$
	\[
	r\left(\bigcup_{i\geqslant1}\,\mathscr{P}_{i}\right) \,=\,\sum_{i\geqslant1}\,r(\mathscr{P}_{i})
	\]
	
\end{enumerate}

We can see that a weight measure is like a probability measure on $\mathscr{P}$ that is a generalization of $\Omega$ in which sample points are predicates. As a matter of fact, $\mathscr{P}$ can be reduced to a probability space by choosing its specific predicates (corollary \ref{23}). 

\begin{lemma}
	\label{14} \ Suppose $\mathscr{P}_{1},\mathscr{P}_{2}\in\sigma(\mathscr{P})$. Then
\end{lemma}

\begin{enumerate}
	\item $r(\varnothing)=0$
	
	\item $\mathscr{P}_{2}\subset\mathscr{P}_{1} \,\Longrightarrow \, r(\mathscr{P}_{1}-\mathscr{P}_{2}) \,=\, r(\mathscr{P}_{1}) -r(\mathscr{P}_{2})$
	
	\item $\mathscr{P}_{2}\subset\mathscr{P}_{1} \,\Longrightarrow \, r(\mathscr{P}_{2})\leqslant r(\mathscr{P}_{1})$
	
	\item $\mathscr{P}_{1}\subset\mathscr{P} \,\Longrightarrow \,r(\mathscr{P}_{1})\leqslant1$
\end{enumerate}

\begin{proof}
	\ We only prove (ii) and the rest are obvious.\medskip 
	
	(ii) \ Since $\mathscr{P}_{2}\subset \mathscr{P}_{1}$, $\mathscr{P}_{1}\cup \mathscr{P}_{2}\,=\,\mathscr{P}_{1}$. So by axiom \ref{10}
	\[
	r\left(\mathscr{P}_{1}\right) \,=\,r\left(\mathscr{P}_{1}\cup \mathscr{P}_{2}\right) \,=\,r\left(\left(\mathscr{P}_{1} -\mathscr{P}_{2}\right) \cup \mathscr{P}_{2}\right) \,=\,r\left(\mathscr{P}_{1}-\mathscr{P}_{2}\right)+r\left( \mathscr{P}_{2}\right)
	\]	
	And (ii) follows.\bigskip
\end{proof}

Weight measures on predicate bundles are fundamentally different from measures on sentences in \cite{Gaifman} and \cite{Scott-Krauss}. This is because while measures on sentences in \cite{Gaifman} and \cite{Scott-Krauss} are essentially the study of measures on Boolean algebras following \cite{Horn-Tarski}, weight measures are on $\sigma-$fields of predicates. For example, two valid sentences in $\mathscr{P}$ generally do not have the same weight measure compared to (1) \cite[p2]{Gaifman}. In addition, conjunction and disjunction are (generally) not preserved for predicates in $\mathscr{P}$. So (2) \cite[p2]{Gaifman} also fails for weight measures. However, (1') \cite[p2]{Gaifman} holds for weight measures because equivalent predicates are considered the same in $\mathscr{P}$ (definition \ref{Sigma_field}).\smallskip

\begin{definition}
	\label{DefPredicateBundle} \ A \textbf{bundle of predicates} is a collection of predicates in $\mathcal{L}_{\omega_1,\omega}$ with a weight measure and is denoted as $\left.\mathscr{P}\right\vert_{r}=\{p\colon p\in \mathcal{L}_{\omega_1,\omega} \,\wedge\,r(p)\}$, where $r$ is the weight measure of the bundle. Two predicate bundles $\left.\mathscr{P}_{1}\right\vert_{r_{1}}$ and $\left.\mathscr{P}_{2}\right\vert_{r_{2}}$ are \textbf{identical} if
	\[
	\left.\mathscr{P}_{1}\right\vert_{r_{1}}=\left.\mathscr{P}_{2}\right\vert_{r_{2}}\,\Longleftrightarrow\,\left(\mathscr{P}_{1} =\mathscr{P}_{2}\,\wedge\,\left(\forall p\in\mathscr{P}_{1}\right) \left(r_{1}(p)\,=\,r_{2}(p)\right)\right)
	\]
\end{definition}

Also, $\left.\mathscr{P}_{0}\right\vert_{r}$ is a \textbf{sub-bundle} of $\left.\mathscr{P}\right\vert_{r}$ if $\mathscr{P}_{0} \subset\mathscr{P}$, and is also denoted as $\left.\mathscr{P}_{0}\right\vert_{r}\subset \left.\mathscr{P}\right\vert_{r}$. For $\left.\mathscr{P}_{1}\right\vert _{r}\subset\left.\mathscr{P}\right\vert_{r}$ and $\left.\mathscr{P}_{2}\right\vert_{r} \subset\left.\mathscr{P}\right\vert_{r}$, the \textbf{intersection} and \textbf{union} of $\left.\mathscr{P}_{1}\right\vert_{r}$ and $\left.\mathscr{P}_{2}\right\vert_{r}$ are $\mathscr{P}_{1}\cap\mathscr{P}_{2}$ and $\mathscr{P}_{1}\cup\mathscr{P}_{2}$ respectively.\footnote{In the rest discussion of this paper, we assume $\left.\mathscr{P}\right \vert_{r}, \left.\mathscr{P}_{1}\right\vert_{r_{1}},\cdots, \left.\mathscr{P}_{n}\right\vert_{r_{n}}$ are predicate bundles unless further specified.}\medskip
 
Since predicates in $\mathscr{P}$ can be inconsistent, there may not be a model for $\mathscr{P}$. Therefore, we do not adopt the model-theoretic approach to predicate bundles. 

\subsection{Central Bundle and Degree of Truth}

In this section, we will formally introduce our model of multi-valued logic.

\begin{definition}
	\ Let $\mathfrak{F}$ be a class of predicates in $\mathcal{L}_{\omega_1,\omega}$ of formal logic. The \textbf{multi-valued logic} $\mathfrak{M}$ is a superclass of $\mathfrak{F}$ including predicates that can not be measured by two truth values. Instead, such a $\phi$ in $\mathfrak{M}$ is assigned with a predicate bundle and measured by a degree of truth, a quantity that takes on any number in $[0,1]$ and is denoted as $T(\phi)$. $\phi$ is equivalent to a predicate in $\mathfrak{F}$ if $T(\phi)=1$ (completely true) or $T(\phi)=0$ (completely false).
\end{definition}

\begin{definition}
	\label{DefPredicateBundleOperations} \ Suppose $\phi,\phi_{1},\phi_{2}$ are predicates in the multi-valued logic $\mathfrak{M}$. The \textbf{logical operations of $\mathfrak{M}$} include the \textbf{negation} of $\phi$ denoted as $\urcorner\phi$, the \textbf{conjunction}, \textbf{disjunction}, \textbf{implication}, and \textbf{equivalence} of $\phi_{1}$ and $\phi_{2}$ in $\mathfrak{M}$ denoted as $\phi_{1}\barwedge\phi_{2}$, $\phi_{1}\veebar\phi_{2}$, $\phi_{1}\rightharpoonup\phi_{2}$, and $\phi_{1}\rightleftharpoons\phi_{2}$ respectively. The above forms the atomic formulas in $\mathfrak{M}$. A well-formed formula of $\mathfrak{M}$ is a syntactically valid formula from the atomic formulas in $\mathfrak{M}$.
\end{definition}

Since any predicate in $\mathfrak{M}$ is modeled by a predicate bundle, we can think it as a representation of the predicate.

\begin{definition}
	\label{DefPredicateBundleRepresentation} \ Suppose $\mathscr{D}$ is a universe of discourse, $\phi$ a predicate in $\mathfrak{M}$ whose predicate bundle is $\left.\mathscr{P}\right\vert_{r}$.\footnote{In the rest discussion of this paper, we assume $\phi,\phi_{1}, \cdots,\phi_{n}$ are predicates in $\mathfrak{M}$, $\mathscr{D}, \mathscr{D}_1, \cdots, \mathscr{D}_n$ are universes of discourse, and $\widetilde{\mathcal{P}}_i(\phi_{i})$ are predicate bundle representations of $\phi_i$ unless further specified.} The \textbf{predicate bundle representation} for $\phi$ is denoted as $\,\widetilde{\mathcal{P}}(\phi)=\left( \left.\mathscr{P}\right\vert_{r}, \mathscr{D}\right)$.
\end{definition}

Now we define degree of truth based on the notion of a central bundle.

\begin{definition}
	\label{DefCentralBundle} \ Suppose $\,\widetilde{\mathcal{P}}(\phi)=\left(\left.\mathscr{P}\right\vert_{r}, \mathscr{D}\right)$. For $x\in\mathscr{D}$, the \textbf{central bundle} of $\phi(x)$ is a sub-bundle of $\left.\mathscr{P}\right\vert_{r}$ that is defined as $\,\left[\phi(x)\right] =\{p\colon p\in\mathscr{P}\,\wedge\, \vdash p(x)\}$, and the \textbf{dual central bundle} of $\phi(x)$ is $\,\left[\phi(x)\right]^{c}=\{p\colon p\in\mathscr{P}\,\wedge\vdash \urcorner p(x)\}$. If $\mathscr{P}$ consists only of sentences,  $\left[\phi\right] =\{p\colon p\in\mathscr{P}\,\wedge\,\vdash p\}$ and $\left[\phi\right]^{c} = \{p\colon p\in\mathscr{P}\,\wedge\,\vdash \urcorner p\}$.
\end{definition}

\begin{remark}
	\ Notice that $\,\vdash p(x)$ in definition \ref{DefCentralBundle} refers to any $p$ in $\mathscr{P}$ satisfied by (the same) $x$ which is different from that of model theory, i.e. $\mathfrak{A}\vDash \mathscr{P}$ means for any $p\in\mathscr{P}$, there is a $\,x\in A$ that $\,\vdash p(x)$ \cite{Chang-Keisler}. (In other words, $x$ that satisfies distinct $p$ may be different.)
\end{remark}

\begin{definition}
	\label{DefBelief} \ For $x\in\mathscr{D}$, the \textbf{degree of truth} (or \textbf{belief}) of $\phi(x)$ is defined to be the weight of the central bundle of $\phi(x)$, i.e. $T(\phi(x))\,=\, r(\left[\phi(x)\right])$. If $\mathscr{P}$ consists only of sentences, $T(\phi)= r(\left[\phi\right])$.
\end{definition}

\begin{remark}
\ Notice that central bundles are classical sets upon which the truth degree of a predicate in $\mathfrak{M}$ is rigorously defined with the law of excluded middle. Thus LEM is retained in $\mathfrak{M}$.\footnote{This is consistent with the Hilbert's view on LEM (see \cite[p476]{Hilbert}).}
\end{remark}

Besides the infinitary language of $\mathcal{L}_{\omega_1,\omega}$ with $\wedge, \vee, \urcorner, \Rightarrow, \Leftrightarrow, \forall, \exists, \bigwedge, \bigvee$, the \textbf{language} of the multi-valued logic $\mathfrak{M}$ is as follows. Each predicate $\phi$ in $\mathfrak{M}$ has a predicate bundle representation $\,\widetilde{\mathcal{P}}(\phi)= \left(\left.\mathscr{P}\right\vert_{r}, \mathscr{D}\right)$, and is measured by a degree of truth $T(\phi)$ in the $[0,1]$ of $\Bbb{R}$. Also, the negation $\urcorner\phi$, conjunction $\phi_{1}\barwedge\phi_{2}$, disjunction $\phi_{1}\veebar\phi_{2}$, implication $\phi_{1}\rightharpoonup\phi_{2}$, and equivalence $\phi_{1}\rightleftharpoons\phi_{2}$ are measured by true degrees as $T(\urcorner\phi)$, $T(\phi_{1}\barwedge\phi_{2})$, $T(\phi_{1}\veebar\phi_{2})$, $T(\phi_{1}\rightharpoonup\phi_{2})$, and $T(\phi_{1}\rightleftharpoons\phi_{2})$ respectively. In a later section, we can see that a well-formed formula in $\mathfrak{M}$ often does not have a truth degree (corollary \ref{21}). So we have

\begin{definition}
	\label{InvalidFormula}\ A well-formed (syntactically valid) formula of $\mathfrak{M}$ is \textbf{(semantically) valid} if it has a degree of truth. Otherwise, it is a \textbf{(semantically) invalid} formula.
\end{definition}

\begin{lemma} 
	\label{11}\ For any $x\in\mathscr{D}$
\end{lemma}

\begin{enumerate}
	\item $\left[\urcorner\phi(x)\right] \,=\,\left[\phi(x)\right]^{c}\,=\,\mathscr{P}-\left[\phi(x)\right]$
	
	\item $\left[\phi(x)\right] \,\cap\,\left[\urcorner\phi(x)\right]\,=\,\varnothing$
	
	\item $\left[\phi(x)\right] \,\cup\,\left[\urcorner\phi(x)\right]\,=\,\mathscr{P}$
\end{enumerate}

\begin{proof}
	\ (i) By definition \ref{DefCentralBundle}, $\left[\urcorner\phi(x)\right] \,=\,\left[\phi(x)\right]^{c}$ and (i) follows.\medskip
	
	(ii) and (iii) follow by (i) and definition \ref{DefCentralBundle}.\smallskip
\end{proof}

\begin{corollary}
	\label{12}\ For any $x\in\mathscr{D}$
	\[T(\phi(x))+T(\urcorner\phi(x))\,=\,1\]
\end{corollary}

\begin{proof}
	\ By definition \ref{DefBelief}, lemma \ref{11} and axiom \ref{10}. \smallskip
\end{proof}


\subsection{Equivalence}\label{SectionEquivalence}

Intuitively, if for any $x$ in a universe of discourse, two predicates of $\mathfrak{M}$ always have the same truth value, they behave identically in the logic operations of $\mathfrak{M}$. So in this case, we say that two predicates are equivalent. Similar to formal logic, a predicate in a valid formula of $\mathfrak{M}$ can be replaced by an equivalent one without affecting the semantic meaning (truth values) of the formula.

\begin{definition}
\label{DefEquivalenceMultiValueLogic} \ Suppose $\,\widetilde{\mathcal{P}}(\phi_1)=\left(\left.\mathscr{P}_1 \right\vert_{r_1}, \mathscr{D}\right)$ and $\,\widetilde{\mathcal{P}}(\phi_2)=\left(\left.\mathscr{P}_2\right\vert_{r_2}, \mathscr{D}\right)$. Then the \textbf{equivalence} of $\phi_{1}$ and $\phi_{2}$ is defined as:
\begin{equation}
\forall(x\in\mathscr{D})\left(\left(\phi_{1}(x)\rightleftharpoons\phi_{2}(x)\right)\Longleftrightarrow\left(T(\phi_{1}(x))=T(\phi_{2}(x))\right)\right)	 \label{EquivalenceMultiValueLogic}
\end{equation}
If $\phi_{1}$ and $\phi_{2}$ are sentences, then
\[
	\left(\phi_{1}\rightleftharpoons\phi_{2}\right)\Longleftrightarrow\left(T(\phi_{1})=T(\phi_{2})\right)
\]
\end{definition}

\begin{remark}
	\ Note that definition \ref{DefEquivalenceMultiValueLogic} makes equivalence in $\mathfrak{M}$ a formal predicate. Obviously, \textquotedblleft $\rightleftharpoons $\textquotedblright\ is an equivalence relation that forms a partition on the predicates of $\mathfrak{M}$.
\end{remark}

\begin{lemma}
	\ For any $x\in\mathscr{D}$
\end{lemma}

\begin{enumerate}
	\item $\phi_{1}(x)\rightleftharpoons \phi_{1}(x)$
	
	\item $\left(\phi_{1}(x)\rightleftharpoons\phi_{2}(x)\right) \,\Longrightarrow \,\left(\phi_{2}(x)\rightleftharpoons \phi_{1}(x)\right)$
	
	\item $\left(\phi_{1}(x)\rightleftharpoons\phi_{2}(x)\,\wedge \,\phi_{2}(x)\rightleftharpoons\phi_{3}(x)\right) \,\,\Longrightarrow \,\,\left(\phi_{1}(x)\rightleftharpoons \phi_{3}(x)\right)$
\end{enumerate}

\begin{corollary}
	\label{36} \ For any $x\in\mathscr{D}$
\end{corollary}

\begin{enumerate}
	\item $\left(\left[\phi_{1}(x)\right] =\left[\phi_{2}(x)\right]\right) \,\Longrightarrow\,\left(T(\phi_{1}(x))=T(\phi_{2}(x))\right)$
	
	\item $\left(\left[\phi_{1}(x)\right] =\left[\phi_{2}(x)\right]\right) \,\Longrightarrow\,\left(\phi_{1}(x)\rightleftharpoons \phi_{2}(x)\right)$
\end{enumerate}

\begin{proof}
	\ By definition \ref{DefBelief} and \ref{DefEquivalenceMultiValueLogic}.\smallskip
\end{proof}

\begin{corollary}
	\label{37} \ For any $x\in\mathscr{D},\,\,\urcorner\urcorner\phi(x)\,\rightleftharpoons\,\phi(x)$
\end{corollary}

\begin{proof}
	\ By lemma \ref{11}(i)
	\[
	\left[\urcorner\urcorner\phi(x)\right] \,=\,\left[\urcorner\phi(x)\right]^{c}\,=\,\left[\phi(x)\right]
	\]
	
	So it follows by corollary \ref{36}. \bigskip
\end{proof}

The fact that equivalence in $\mathfrak{M}$ is formal allows us to apply the identity theory in formal logic to model linguistic expressions (involving adjectives and adverbs) in a natural language. Due to the diversity of adjectives, we will only discuss adjectives that can be represented by numerical values in this paper. For instances, \textit{young}, \textit{middle-aged} and \textit{old} can be represented by a person's age; \textit{short}, \textit{medium} and \textit{tall} can be represented by a person's height. First, let's take a look at following examples.

\begin{problem}
	\label{16} \ Suppose $\phi_{1}(x_{1})\rightleftharpoons\,$\textquotedblleft$\,x_{1}$ is young\textquotedblright\ ($x_{1}$ is the age for a person), $\widetilde{\mathcal{P}}(\phi_1)=(\left.\mathscr{P}_1\right\vert_{r_1}, \allowbreak [0,150])$\footnote{We assume that human age is no more than $150$ years old.} and $\left.\mathscr{P}_{1}\right\vert_{r_{1}}\,=\,\{q_{n}(z)\colon q_{n}(z)\Leftrightarrow (z\leqslant n+19) \,\wedge\,1\leqslant n\leqslant20\,\wedge\,r_{1}(q_{n})=0.05\,\}$. Find $T(\phi_{1}(x_{1}))$ and $T(\urcorner\phi_{1}(x_{1}))$ for $x_{1}\,=\,20$, $30$, $40\left(\text{years old}\right)$.
\end{problem}

\begin{proof}
	[Solution]\ \ Since $20$ satisfies each $q_{n}(z)$ in $\mathscr{P}_{1}$, $\left[\phi_{1}(20)\right] =\mathscr{P}_{1}$. So by definition \ref{DefBelief}, $T(\phi_{1}(20))=1$. Also, $T(\phi_{1}(40))=0$ for $\left[\phi_{1}(40)\right]=\varnothing$. Since $\left[\phi_{1}(30)\right] =\{q_{n}(z)\colon q_{n}(z)\in\mathscr{P}_{1}\,\wedge\,11\leqslant n\leqslant20\}$, $T(\phi_{1}(30)) =0.5$.\smallskip
	
	Since $\urcorner\phi_{1}(x_{1})\,\rightleftharpoons$ \textquotedblleft$x_{1}$ is not young\textquotedblright, by lemma \ref{11}, $\left[\urcorner\phi_{1}(20)\right]=\left[\phi_{1}(20)\right]^c=\varnothing$, $\left[\urcorner\phi_{1}(40)\right] =\left[\phi_{1}(40)\right]^c=\mathscr{P}_{1}$, $\left[\urcorner\phi_{1}(30)\right] \,=\,\mathscr{P}_{1}-\left[\phi_{1}(30)\right]$. Thus $T(\urcorner\phi_{1}(20))=0$, $T(\urcorner\phi_{1}(40))=1$, $T(\urcorner\phi_{1}(30))=0.5$.\smallskip
\end{proof}

\begin{problem}
	\label{32} \ Suppose $\phi_{2}(x_{2})\rightleftharpoons\,$\textquotedblleft $\,x_{2}$ is tall\textquotedblright\ ($x_{2}$ is the height for a person), $\widetilde{\mathcal{P}}(\phi_2)=(\left.\mathscr{P}_2\right\vert_{r_2}, \allowbreak [0,3])$\footnote{We assume that human height is no more than $3$ meters.} and $\left.\mathscr{P}_{2}\right\vert_{r_{2}}\,=\,\{q_{n}(z)\colon q_{n}(z)\Leftrightarrow (\alpha(n)\leqslant z) \,\wedge\, \alpha(n) =n/100+1.65 \,\wedge\, 1\leqslant n\leqslant 20\,\wedge\,r_{2}(q_{n})=0.05\,\}$. Find $T(\phi_{2}(x_{2}))$ for $x_{2}\,=\,1.7$, $1.75$, $1.8\left(m\right)$.
\end{problem}

\begin{proof}
	[Solution] \ Obviously, $1.66\leqslant\alpha(n)\leqslant1.85$. Since $\left[\phi_{2}(1.7)\right] =\{q_{n}(z)\colon q_{n}(z)\in\mathscr{P}_{2}\,\wedge\,1\leqslant n\leqslant 5\}$, by definition \ref{DefBelief}, $T(\phi_{2}(1.7))=0.25$. Likewise, $\,T(\phi_{2}(1.75))=0.5\,$ and $\,T(\phi_{2}(1.8))=0.75$.\smallskip
\end{proof}

\begin{corollary}
	\label{23} \ Suppose $\widetilde{\mathcal{P}}(\phi)=(\left.\mathscr{P}\right\vert_{r}, \mathscr{D})$. If $\left.\mathscr{P}\right\vert_{r}$ consists of predicates which are each satisfied by a unique singleton in $\mathscr{D}$ and vice versa, then $\mathscr{D}$ is reduced to a probability space and $T(\phi(x))$ is identical to the weight of a singleton (probability of a sample).
\end{corollary}

\begin{proof}
	\ For any $p\in\mathscr{P}$, if $\vdash p(x)$ ($x\in\mathscr{D}$), then for any $y\in \mathscr{D}$ and $y\neq x$, $\vdash\urcorner p(y)$. Thus $\left[\phi(x)\right]=\{p\}$ and $T(\phi(x))=r(p)$. Since for any $x\in\mathscr{D}$, there is a $p\in\mathscr{P}$ that $\vdash p(x)$, $\mathscr{D}$ is reduced to a probability space.\bigskip
\end{proof}

In next sections, we will investigate logical operations such as conjunction, disjunction and implication in $\mathfrak{M}$.

\subsection{Conjunction and Disjunction}

The conjunction and disjunction in the multi-valued logic $\mathfrak{M}$ can model the connectives \textit{and }and\textit{\ or} in linguistics, e.g. \textquotedblleft he is young and tall\textquotedblright, \textquotedblleft it is big or heavy\textquotedblright, and so on. They are based upon joint bundles that are Cartesian products of predicate bundles.

\begin{definition}
	\label{DefJointPredicateBundles} \ Suppose $\left.\mathscr{P}_{i}\right\vert_{r_{i}}=\,\{p_{i}\colon p_{i}\in\mathcal{L}_{\omega_1,\omega}\,\wedge\,r_{i}(p_{i})\}$ $\left(1\leqslant i\leqslant n\right)$ are $n$ predicate bundles. A \textbf{joint bundle} of $\left.\mathscr{P}_{i} \right\vert_{r_{i}}$ is defined as:
	\[
	\left.\left(\mathscr{P}_{1}\times\cdots\times\mathscr{P}_{n}\right)\right\vert_{r}\,=\,\{\left(p_{1},\cdots,p_{n}\right) \colon p_{i}\in\mathscr{P}_{i}\,\wedge\,r\left(  \left(  p_{1},\cdots,p_{n}\right)\right) \,\wedge\,1\leqslant i\leqslant n\,\}
	\]
	Where $r$ is a weight measure on $\sigma(\mathscr{P}_{1}\times\cdots\times\mathscr{P}_{n})$.
\end{definition}

The following axiom holds for a joint bundle.

\begin{axiom}
	\label{17} \ Suppose $\left.\mathscr{P}_{i}\right\vert_{r_{i}}$ $\left(1\leqslant i\leqslant n\right)$ are given in definition \ref{DefJointPredicateBundles}. For any $\mathscr{I}_{i}\in\sigma(\mathscr{P}_{i})$
	\[
	r\left(\mathscr{P}_{1}\times\cdots\times\mathscr{I}_{i}\times\cdots\times\mathscr{P}_{n}\right) \,=\,r(\mathscr{P}_{1})\cdots r(\mathscr{I}_{i})\cdots r(\mathscr{P}_{n})\,=\,r(\mathscr{I}_{i})
	\]
\end{axiom}

\begin{corollary}
	\label{52}\qquad$r\left(\mathscr{P}_{1}\times\cdots\times\mathscr{P}_{n}\right) \,=\,1$.
\end{corollary}

\begin{definition}
	\label{DefSetBundleRepresentationConjunction} \ Suppose $\widetilde{\mathcal{P}}(\phi_{1})=\left(\left.\mathscr{P}_{1} \right\vert_{r_{1}}, \mathscr{D}_1\right)$ and $\widetilde{\mathcal{P}}(\phi_{2})=\left(\left.\mathscr{P}_{2}\right \vert_{r_{2}}, \mathscr{D}_2\right)$. Then the \textbf{predicate bundle representation of the conjunction of} $\,\phi_{1}$ and $\,\phi_{2}$ is:
	\[
	\widetilde{\mathcal{P}}(\phi_{1}\barwedge\phi_{2})\,=\,\left(\left.\left(\mathscr{P}_{1}\times \mathscr{P}_{2}\right)\right\vert_{r}, \mathscr{D}_1\times\mathscr{D}_2\right)
	\]
	For $x\in\mathscr{D}_1$ and $y\in\mathscr{D}_2$, the \textbf{central bundle} of $\phi_{1}\barwedge\phi_{2}$ is:
	\begin{equation}
		\left[\phi_{1}(x)\barwedge\phi_{2}(y)\right] \,=\,\{(p,q) \colon (p\in\mathscr{P}_{1}\,\wedge\, q\in\mathscr{P}_{2}) \,\wedge\, (\vdash p(x)\,\wedge\, \vdash q(y))\} \label{ConjunctionCentralClass}
	\end{equation}
	If $\mathscr{P}_1$ and $\mathscr{P}_2$ consist only of sentences, the \textbf{central bundle} of $\phi_{1}\barwedge\phi_{2}$ is:
	\begin{equation}
		\left[\phi_{1}\barwedge\phi_{2}\right] \,=\,\{(p,q) \colon (p\in\mathscr{P}_{1}\,\wedge\,q\in\mathscr{P}_{2}) \,\wedge\, (\vdash p \,\wedge\,\vdash q)\}
	\end{equation}
	The \textbf{degree of truth} for $\,\phi_{1}\barwedge\phi_{2}$ is: $\,T(\phi_{1}(x)\barwedge\phi_{2}(y)) \,= \,r(\left[\phi_{1}(x) \barwedge \phi_{2}(y)\right])$ or $\,T(\phi_{1}\barwedge\phi_{2}) \,= \,r(\left[\phi_{1} \barwedge \phi_{2}\right])$.
\end{definition}

\begin{definition}
	\label{DefSetBundleRepresentationDisjunction} \ The \textbf{predicate bundle representation of the disjunction of} $\,\phi_{1}$ and $\,\phi_{2}$ is:
	\[
	\widetilde{\mathcal{P}}(\phi_{1}\veebar\phi_{2})\,=\,\left(\left.\left(\mathscr{P}_{1}\times \mathscr{P}_{2}\right)\right\vert_{r}, \mathscr{D}_1\times\mathscr{D}_2\right)
	\]
	For $x\in\mathscr{D}_1$ and $y\in\mathscr{D}_2$, the \textbf{central bundle} of $\phi_{1}\veebar\phi_{2}$ is:
	\begin{equation}
		\left[\phi_{1}(x)\veebar\phi_{2}(y)\right] \,=\,\{(p,q) \colon (p\in\mathscr{P}_{1} \,\wedge\,q\in\mathscr{P}_{2}) \,\wedge\, (\vdash p(x)\,\vee\, \vdash q(y))\} \label{DisjunctionCentralClass}
	\end{equation}
	If $\mathscr{P}_1$ and $\mathscr{P}_2$ consist only of sentences, the \textbf{central bundle} of $\phi_{1}\veebar\phi_{2}$ is:
	\begin{equation}
		\left[\phi_{1}\veebar\phi_{2}\right] \,=\,\{(p,q) \colon (p\in\mathscr{P}_{1}\,\wedge\, q\in\mathscr{P}_{2}) \,\wedge\, (\vdash p \,\vee\,\vdash q)\}
	\end{equation}
	The \textbf{degree of truth} for $\,\phi_{1}\veebar\phi_{2}$ is: $\,T(\phi_{1}(x)\veebar\phi_{2}(y)) \,= \,r(\left[\phi_{1}(x) \veebar \phi_{2}(y)\right])$ or $\,T(\phi_{1}\veebar\phi_{2}) \,= \,r(\left[\phi_{1} \veebar \phi_{2}\right])$.
\end{definition}

\begin{remark}
\ $\phi_{1}(x)\barwedge\phi_{2}(y)$ can be written as $(\phi_{1}\barwedge\phi_{2})(x,y)$ and $\phi_{1}(x)\veebar\phi_{2}(y)$ as $(\phi_{1}\veebar\phi_{2})(x,y)$.
\end{remark}

\begin{remark}
	\label{3} \ If $\phi_{1}$ and $\phi_{2}$ have the same predicate bundle, i.e. $\widetilde{\mathcal{P}}(\phi_{1})=\left(\left.\mathscr{P}_{1}\right\vert_{r_{1}}, \mathscr{D}_1\right)$ and $\widetilde{\mathcal{P}}(\phi_{2})=\left(\left.\mathscr{P}_{1}\right\vert_{r_{1}}, \mathscr{D}_1\right)$, then
	\begin{align*}
		\left[\phi_{1}(x)\barwedge\phi_{2}(y)\right] & \,=\,\{p\colon p\in\mathscr{P}_{1}\,\wedge\,(\vdash p(x) \,\wedge\,\vdash p(y))\}\,=\,\left[\phi_{1}(x)\right] \,\cap\,\left[\phi_{2}(y)\right] 
		\\
		\left[\phi_{1}(x)\veebar\phi_{2}(y)\right] & \,=\,\{p\colon p\in\mathscr{P}_{1}\,\wedge\,(\vdash p(x) \,\vee\,\vdash p(y))\}\,=\,\left[\phi_{1}(x)\right] \,\cup\,\left[\phi_{2}(y)\right] 
	\end{align*}	
\end{remark}

\begin{theorem}
	\label{13}\ For $x\in\mathscr{D}_1$ and $y\in\mathscr{D}_2$
\end{theorem}

\begin{enumerate}
	\item $\left[\phi_{1}(x)\barwedge\phi_{2}(y)\right] \,=\,\left[\phi_{1}(x)\right] \times\mathscr{P}_{2}\,\cap\,\mathscr{P}_{1}\times \left[\phi_{2}(y)\right] \,=\,\left[\phi_{1}(x)\right] \times\left[\phi_{2}(y)\right]$
	
	\item $\left[\phi_{1}(x)\veebar\phi_{2}(y)\right] \,=\,\left[\phi_{1}(x)\right] \times\mathscr{P}_{2}\,\cup\,\mathscr{P}_{1}\times \left[\phi_{2}(y)\right]$
\end{enumerate}

\begin{proof}
	\ (i) \ By definition \ref{DefCentralBundle} and (\ref{ConjunctionCentralClass})
	\begin{align*}
		\left(p,q\right) \in\left[\phi_{1}(x)\barwedge\phi_{2}(y)\right] & \,\,\Longleftrightarrow\,\, (p\in\mathscr{P}_{1}\,\wedge\, q\in\mathscr{P}_{2}) \,\wedge\, (\vdash p(x)\,\wedge\, \vdash q(y))
		\\
		& \,\,\Longleftrightarrow\,\, \left(p\in\left[\phi_{1}(x)\right] \,\wedge \,q\in\mathscr{P}_{2}\right) \,\wedge\,\left(p\in\mathscr{P}_{1} \,\wedge\,q\in\left[\phi_{2}(y)\right]\right) 
		\\
		& \,\,\Longleftrightarrow\,\, \left(p,q\right)\in\left[\phi_{1}(x)\right]\times\mathscr{P}_{2} \,\cap\, \mathscr{P}_{1}\times \left[\phi_{2}(y)\right]
		\\
		& \,\,\Longleftrightarrow\,\, \left(p,q\right) \in\left[\phi_{1}(x)\right]\times\left[\phi_{2}(y)\right]
	\end{align*}
	
	(ii) is proved by (\ref{DisjunctionCentralClass}) and is similar to (i).\smallskip
\end{proof}

\begin{corollary}
	\label{19}\ If $\mathscr{P}_1$ and $\mathscr{P}_2$ consist only of sentences, then
\end{corollary}

\begin{enumerate}
	\item $\left[\phi_{1}\barwedge\phi_{2}\right] \,=\,\left[\phi_{1}\right] \times\mathscr{P}_{2}\,\cap\,\mathscr{P}_{1}\times \left[\phi_{2}\right] \,=\,\left[\phi_{1}\right] \times\left[\phi_{2}\right]$
	
	\item $\left[\phi_{1}\veebar\phi_{2}\right] \,=\,\left[\phi_{1}\right] \times\mathscr{P}_{2}\,\cup\,\mathscr{P}_{1}\times \left[\phi_{2}\right]$
\end{enumerate}

\begin{corollary}
	\label{20}\ For $x\in\mathscr{D}_1$ and $y\in\mathscr{D}_2$
\end{corollary}

\begin{enumerate}
	\item $T(\phi_{1}(x)\barwedge\phi_{2}(y))\,\leqslant\,\min\{T(\phi_{1}(x)),\,T(\phi_{2}(y))\}$
	
	\item $\max\{T(\phi_{1}(x)),\,T(\phi_{2}(y))\}\,\leqslant\,T(\phi_{1}(x)\veebar\phi_{2}(y))\leqslant 1$
\end{enumerate}

\begin{proof}
	\ (i) \ By theorem \ref{13}(i), lemma \ref{14} and axiom \ref{17}
	\[
	\left[\phi_{1}(x)\barwedge\phi_{2}(y)\right] \,\subset\,\left[\phi_{1}(x)\right]\times\mathscr{P}_{2}\text{ \ \ and \ \ }r(\left[\phi_{1}(x)\barwedge\phi_{2}(y)\right])\leqslant r(\left[\phi_{1}(x)\right]\times\mathscr{P}_{2}) \,=\,r(\left[\phi_{1}(x)\right])
	\]
	
	Likewise, $r(\left[\phi_{1}(x)\barwedge\phi_{2}(y)\right])\leqslant r(\left[\phi_{2}(y)\right])$ and (i) follows. \medskip
	
	(ii) follows by theorem \ref{13}(ii) and is similar to (i). \smallskip
\end{proof}

\begin{corollary}
	\label{15}\ If $\mathscr{P}_1$ and $\mathscr{P}_2$ consist only of sentences, then
\end{corollary}

\begin{enumerate}
	\item $T(\phi_{1}\barwedge\phi_{2})\,\leqslant\,\min\{T(\phi_{1}),\,T(\phi_{2})\}$
	
	\item $\max\{T(\phi_{1}),\,T(\phi_{2})\}\,\leqslant\,T(\phi_{1}\veebar\phi_{2})\leqslant 1$
\end{enumerate}

Theorem \ref{13} shows that conjunction and disjunction in $\mathfrak{M}$ can be reduced to the intersection and union of central bundles which satisfy the algebra of sets. The general cases are given as follows.

\begin{definition}
	\label{DefCartesianForm} \ Suppose $\,\widetilde{\mathcal{P}}(\phi_{i})=\left(\left.\mathscr{P}_{i} \right \vert_{r_{i}},\mathscr{D}_i\right)\, \left(1\leqslant i\leqslant n\right)$ and $\,\vv{\left.\mathscr{P} \right\vert_{r}} =\left.\left(\mathscr{P}_{1}\times\cdots\times\mathscr{P}_{n}\right) \right\vert_{r}$. Then the \textbf{central bundle of the Cartesian product} \textbf{for} $\phi_{i}$ is:
	\[
	\left[\vv{\phi_{i}}(x_i)\right]\,=\,\mathscr{P}_{1}\times\cdots\times\mathscr{P}_{i-1}\times\left[\phi_{i}(x_i)\right] \times \mathscr{P}_{i+1}\times\cdots\times\mathscr{P}_{n}
	\]
	And the \textbf{central bundle of the Cartesian product} \textbf{for} $\urcorner\phi_{i}$ is:
	\[
	\left[\vv{\urcorner\phi_{i}}(x_i)\right]\,=\,\mathscr{P}_{1}\times\cdots\times\mathscr{P}_{i-1}\times\left [\urcorner\phi_{i}(x_i) \right]\times\mathscr{P}_{i+1}\times\cdots\times\mathscr{P}_{n}
	\]
	In addition, $\,\left[\vv{\phi_{i}}(x_i)\right]^{c}\,=\,\vv{\mathscr{P}}-\left[\vv{\phi_{i}}(x_i)\right]$.
\end{definition}

\begin{definition}
	\label{DefCartesianFormConjunctionDisjunction} \ Suppose everything is the same as in definition \ref{DefCartesianForm}. Then the \textbf{central bundle of the Cartesian product} \textbf{for} $\,\phi_{i_{1}}\barwedge\cdots \barwedge\phi_{i_{k}}\, (1\leqslant i_{1}<\cdots<i_{k}\leqslant n,\,2\leqslant k\leqslant n)$ is:
	\[
	\left[\phi_{i_{1}}(x_{i_1})\barwedge\cdots\barwedge\phi_{i_{k}}(x_{i_k})\right] \,=\,\left[\vv{\phi_{i_{1}}}(x_{i_1})\right]\cap\cdots\cap \left[\vv{\phi_{i_{k}}}(x_{i_k})\right]
	\]
	And the \textbf{central bundle of the Cartesian product} \textbf{for} $\,\phi_{i_{1}}\veebar\cdots\veebar\phi_{i_{k}}$ is:
	\[
	\left[\phi_{i_{1}}(x_{i_1})\veebar\cdots\veebar\phi_{i_{k}}(x_{i_k})\right] \,=\,\left[\vv{\phi_{i_{1}}}(x_{i_1})\right] \cup\cdots\cup \left[\vv{\phi_{i_{k}}}(x_{i_k})\right]
	\]
\end{definition}

\begin{lemma}
	\label{60}\qquad$\left[\vv{\urcorner\phi_{i}}(x_i)\right]\,=\,\left[\vv{\phi_{i}}(x_i)\right]^{c}\,\left(1\leqslant i\leqslant n\right)$
\end{lemma}

\begin{proof}
	\ By definition \ref{DefCartesianForm} and lemma \ref{11}
	\[
	\left[\vv{\urcorner\phi_{i}}(x_{i})\right]\,=\,\mathscr{P}_{1}\times\cdots\times\mathscr{P}_{i-1}\times\left(\mathscr{P}_{i}- \left[\phi_{i}(x_{i})\right]\right)\times\mathscr{P}_{i+1}\times\cdots\times\mathscr{P}_{n}\,=\,\vv{\mathscr{P}}- \left[\vv{\phi_{i}}(x_{i})\right] \,=\,\left[\vv{\phi_{i}}(x_{i})\right]^{c}
	\]
	\vspace{-1em}
\end{proof}

\begin{lemma}
	\label{35}\qquad$r\left(\left[\vv{\phi_{i}}(x_i)\right]\right)=\,r(\left[\phi_{i}(x_i)\right])\,\left(1\leqslant i\leqslant n\right)$
\end{lemma}

\begin{proof}
	\ By definition \ref{DefCartesianForm} and axiom \ref{17}. \smallskip
\end{proof}

\begin{theorem}
	\label{34}\ For $x\in\mathscr{D}_1$ and $y\in\mathscr{D}_2$
	\[
	T(\phi_{1}(x)\veebar\phi_{2}(y))\,=\,T(\phi_{1}(x))+T(\phi_{2}(y))-T(\phi_{1}(x)\barwedge\phi_{2}(y))
	\]
\end{theorem}

\begin{proof}
	\ By definition \ref{DefCartesianFormConjunctionDisjunction}
	\[
	\left[\phi_{1}(x)\veebar\phi_{2}(y)\right] \,=\,\left[\vv{\phi_{1}}(x)\right]\cup\left[\vv{\phi_{2}}(y)\right] \,=\, \left[ \vv{\phi_{1}}(x)\right]\cup\left(\left[\vv{\phi_{2}}(y)\right]-\left[\vv{\phi_{1}}(x)\right]\cap \left[\vv{\phi_{2}}(y)\right]\right)
	\]
	
	So by lemma \ref{14}
	\[
	r\left(\left[\phi_{1}(x)\veebar\phi_{2}(y)\right]\right) \,=\,r\left(\left[\vv{\phi_{1}}(x)\right]\right)+ r\left(\left[\vv{\phi_{2}}(y)\right]\right)-r\left(\left[\vv{\phi_{1}}(x)\right]\cap\left[\vv{\phi_{2}}(y)\right]\right)
	\]
	
	And it follows by lemma \ref{35}. \smallskip
\end{proof}

\begin{corollary}
	\label{25}\ If $\mathscr{P}_1$ and $\mathscr{P}_2$ consist only of sentences, then
	\[
	T(\phi_{1}\veebar\phi_{2})\,=\,T(\phi_{1})+T(\phi_{2})-T(\phi_{1}\barwedge\phi_{2})
	\]
\end{corollary}

\begin{definition}
	\label{DefConstantPredicates} \ Suppose $\mathbf{1}$ and $\mathbf{0}$ are propositions in $\mathfrak{M}$ where $\left[  \mathbf{1}\right]\,=\,\vv{\left.\mathscr{P}\right\vert_{r}}$ and $\left[\mathbf{0}\right]\,=\,\mathcal{\varnothing}$. Then $T(\mathbf{1})\,=\,1$ and $T(\mathbf{0})\,=\,0$.
\end{definition}

\begin{theorem}
	\label{44} \ Predicates in $\mathfrak{M}$ form a Boolean algebra under $T$ with respect to negation, conjunction and disjunction. More precisely, for $x\in\mathscr{D}_1$, $y\in\mathscr{D}_2$ and $z\in\mathscr{D}_3$
\end{theorem}

\begin{enumerate}
	\item $T(\phi(x)\barwedge\phi(x))\,=\,T(\phi(x))$,\quad $T(\phi(x)\veebar\phi(x))\,=\,T(\phi(x))$.
	
	\item $T(\phi_{1}(x)\barwedge\phi_{2}(y))\,=\,T(\phi_{2}(y)\barwedge\phi_{1}(x))$,\quad$T(\phi_{1}(x)\veebar\phi_{2}(y)) \,=\,T(\phi_{2}(y) \veebar\phi_{1}(x))$.
	
	\item $T(\phi_{1}(x)\barwedge\left(\phi_{2}(y)\barwedge\phi_{3}(z)\right))\,=\,T(\left(\phi_{1}(x)\barwedge\phi_{2}(y) \right) \barwedge\phi_{3}(z))$,
	\\[1ex]
	$T(\phi_{1}(x)\veebar\left(\phi_{2}(y)\veebar\phi_{3}(z)\right))\,=\, T(\left(\phi_{1}(x)\veebar\phi_{2}(y) \right) \veebar\phi_{3}(z))$.
	
	\item $T(\phi_{1}(x)\barwedge\left(\phi_{1}(x)\veebar\phi_{2}(y)\right))\,=\,T(\phi_{1}(x)\veebar\left(\phi_{1}(x) \barwedge\phi_{2}(y) \right))\,=\,T(\phi_{1}(x))$.
	
	\item  $T(\phi_{1}(x)\barwedge\left(\phi_{2}(y)\veebar\phi_{3}(z)\right))\,=\,T(\left(\phi_{1}(x)\barwedge\phi_{2}(y) \right) \veebar\left(\phi_{1}(x)\barwedge\phi_{3}(z)\right))$, 
	\\[1ex]
	$T(\phi_{1}(x)\veebar \left(\phi_{2}(y)\barwedge\phi_{3}(z)\right))\,=\, T(\left(\phi_{1}(x)\veebar\phi_{2}(y)\right)\barwedge \left(\phi_{1}(x)\veebar\phi_{3}(z)\right))$.
	
	\item $T(\urcorner\left(\phi_{1}(x)\barwedge\phi_{2}(y)\right))\,=\,T(\urcorner\phi_{1}(x)\,\veebar\,\urcorner \phi_{2}(y))$, \quad $T(\urcorner\left(\phi_{1}(x)\veebar\phi_{2}(y)\right))\,=\,T(\urcorner\phi_{1}(x)\,\barwedge\, \urcorner\phi_{2}(y))$.
	
	\item $T(\phi(x)\barwedge\mathbf{1})\,=\,T(\phi(x)\veebar\mathbf{0})\,=\,T(\phi(x))$,\quad$T(\phi(x)\barwedge\mathbf{0}) \,=\,0$,\quad$T(\phi(x)\veebar\mathbf{1})\,=\,1$.
	
	\item $T(\phi(x)\,\veebar\,\urcorner\phi(x))\,=\,1$, \quad$T(\phi(x)\,\barwedge\,\urcorner\phi(x))\,=\,0$.
\end{enumerate}

\begin{proof}
	\ (i) to (v) follow by definition \ref{DefCartesianFormConjunctionDisjunction}. \medskip
	
	(vi) \ We only prove the first one. By definition \ref{DefCartesianFormConjunctionDisjunction} and lemma \ref{60}
	\[
	\left[\urcorner\left(\phi_{1}(x)\barwedge\phi_{2}(y)\right)\right]\,=\,\left(\left[\vv{\phi_{1}}(x)\right]\cap \left[\vv{\phi_{2}}(y)\right]\right)^{c}\,=\,\left[\vv{\phi_{1}}(x)\right]^{c}\cup\left[\vv{\phi_{2}}(y)\right]^{c}\,=\, \left[  \urcorner\phi_{1}(x)\,\veebar\,\urcorner\phi_{2}(y)\right]
	\]
	
	So (vi) follows by corollary \ref{36}. \medskip
	
	(vii) \ By definition \ref{DefConstantPredicates}
	\[
	\left[\phi(x)\barwedge\mathbf{1}\right] \,=\,\left[\vv{\phi}(x)\right]\cap\vv{\mathscr{P}}\,=\,\left[\vv{\phi}(x)\right] \quad \text{and}\quad\left[\phi(x)\veebar\mathbf{0}\right]\,=\,\left[\vv{\phi}(x)\right]\cup\varnothing\,=\, \left[\vv{\phi}(x)\right]
	\]	
	
	Also, $\left[\phi(x)\barwedge\mathbf{0}\right] \,=\,\varnothing\,$ and $\,\left[\phi(x)\veebar\mathbf{1}\right] \,=\, \vv{\mathscr{P}}$. So (vii) follows. \medskip
	
	(viii) \ By definition \ref{DefCartesianFormConjunctionDisjunction} and lemma \ref{60}. \bigskip
\end{proof}

The following gives a predicate bundle in which only one predicate can be satisfied for a value.

\begin{problem}
	\label{24}\ Suppose $\phi_{0}(x_{0})\rightleftharpoons\,$\textquotedblleft$\,$the gender of $x_{0}$\textquotedblright, $\widetilde{\mathcal{P}}(\phi_{0})=\left(\left.\mathscr{P}_{0}\right\vert_{r_{0}}, \{\text{`M'},\, \text{`F'}\}\right)$ and $\allowbreak\left.\mathscr{P}_{0}\right\vert_{r_{0}}=\,\{p(u)\colon (p(u)\Leftrightarrow ``u=\text{`M'}\:" \,\vee\,p(u) \Leftrightarrow ``u= \text{`F'}\:") \,\wedge\, r_{0}(p)=0.5\,\}$. Find $T(\phi_{0}(\text{`M'}))$ and $T(\phi_{0}(\text{`F'}))$.
\end{problem}

\begin{proof}
	[Solution]\ Obviously, the central bundle $\left[\phi_{0}(\text{`M'})\right] =\{``u=\text{`M'}\,"\}$. So $T(\phi_{0}(\text{`M'}))=0.5$. Likewise, $\left[\phi_{0}(\text{`F'})\right] =\{``u=\text{`F'}\,"\}$ and $T(\phi_{0}($`F'$))=0.5$. 
\end{proof}

\begin{problem}
	\label{2}\ \ Suppose $\phi_{0}(x_{0})$ and $\widetilde{\mathcal{P}}(\phi_{0})$ are given in problem \ref{24}, $\phi_{1}(x_{1})\rightleftharpoons\,$\textquotedblleft$\,x_{1}$ is young\textquotedblright\ and $\widetilde{\mathcal{P}}(\phi_{1})=\left(\left.\mathscr{P}_{1}\right\vert_{r_{1}},[0,150]\right)$ where  $\left.\mathscr{P}_{1}\right\vert_{r_{1}}$ is given in problem \ref{16}. Find $T(\phi_{0}($`M'$) \,\barwedge\,\phi_{1}(x_{1}))$ and $T(\phi_{0}($`F'$)\,\barwedge\,\phi_{1}(x_{1}))$ for $x_{1}\,=\,20,\,30,\,40\left(\text{years old}\right)$.
\end{problem}

\begin{proof}
	[Solution]\ Let $\left.\left(\mathscr{P}_{0}\times\mathscr{P}_{1}\right)\right\vert_{r_{0,1}}= \,\{\left(p_{k},q_{n}\right) \colon p_{k}\in\mathscr{P}_{0}\,\wedge\,p_1\Leftrightarrow``u=\text{`M'}\," \,\wedge\, p_2\Leftrightarrow``u=\text{`F'}\,"\,\wedge\,q_{n}\in\mathscr{P}_{1}\,\wedge\, r_{0,1}(\left(p_{k},q_{n}\right))\,\wedge\, 1\leqslant k \leqslant2 \,\wedge\,1\leqslant n\leqslant20\}$ where
	\[
	r_{0,1}\left(\left(p_1,q_{n}\right)\right) \,=\,r_{0,1}\left(\left(p_2,q_{n}\right)\right) \,=\,r_{0} (p_1)\,r_{1}(q_{n})\,=\,r_{0}(p_2)\,r_{1}(q_{n})\,=\,0.025
	\]
	
	By theorem \ref{19}(i) and problem \ref{16}, we have \medskip
	
	$\left[\phi_{0}(\text{`M'})\,\barwedge\,\phi_{1}(20)\right] \,=\,\{p_1\}\times\mathscr{P}_{1}$ \ and \ $T(\phi_{0}($`M'$) \,\barwedge\,\phi_{1}(20))\,=\,\,0.5.$ \smallskip
	
	$\left[\phi_{0}(\text{`M'})\,\barwedge\,\phi_{1}(40)\right] \,=\,\{p_1\}\times\varnothing\,=\,\varnothing\ \ $and \ $T(\phi_{0}($`M'$)\,\barwedge\,\phi_{1}(40))\,=\,0.$ \smallskip
	
	$\left[\phi_{0}(\text{`M'})\,\barwedge\,\phi_{1}(30)\right] \,=\,\{\left(p_1,q_{n}\right) \colon q_{n}\in \mathscr{P}_{1} \,\wedge\,11\leqslant n\leqslant20\}\ \ $and \ $T(\phi_{0}($`M'$)\,\barwedge\,\phi_{1}(30))\,=\,0.25.$ \medskip
	
	Since $r_{0}(p_1)=\,r_{0}(p_2)\,=\,0.5$, $\ T(\phi_{0}($`F'$)\,\barwedge\, \phi_{1}(x_{1})) \,=\, T(\phi_{0}($`M'$) \,\barwedge\,\phi_{1}(x_{1}))$. \smallskip
\end{proof}

\begin{problem}
	\label{26} \ Suppose $\phi_{0}(x_{0})$ and $\widetilde{\mathcal{P}}(\phi_{0})$ are given in problem \ref{24}, $\phi_{2}(x_{2})\rightleftharpoons\,$\textquotedblleft$\,x_{2}$ is tall\textquotedblright\ and $\widetilde{\mathcal{P}} (\phi_{2})=\left(\left.\mathscr{P}_{2}\right\vert_{r_{2}},[0,3]\right)$ where $\left.\mathscr{P}_{2}\right\vert_{r_{2}}$ is given in problem \ref{32}. Find $T(\phi_{0}($`M'$)\,\barwedge\,\phi_{2}(x_{2}))$ and $T(\phi_{0}($`F'$)\,\barwedge\,\phi_{2}(x_{2}))$ for $x_{2}\,=\,1.7,\,1.75,\,1.8\left(m\right)$.
\end{problem}

\begin{proof}
	[Solution]\ Because of the difference in height between men and women, weight measures for males and females are different. So let $\left.\left(\mathscr{P}_{0}\times\mathscr{P}_{2}\right) \right\vert_{r_{0,2}}=\,\{\left(p_{k},q_{n} \right) \colon p_{k}\in \mathscr{P}_{0}\,\wedge\,p_1\Leftrightarrow``u=\text{`M'}\," \,\wedge\, p_2\Leftrightarrow``u=\text{`F'}\,"\, \wedge\,q_{n}\in\mathscr{P}_{2}\,\wedge\,r_{0,2} \left(\left(p_{k},q_{n} \right)\right)\,\wedge\, 1\leqslant k \leqslant2 \,\wedge\,1\leqslant n\leqslant20\}$ where
	\begin{align*}
		r_{0,2}\left(\left(p_1,q_{n}\right)\right) & \,=\, 0.05\text{, \ \ }11\leqslant n\leqslant20\quad\left(1.76 \leqslant \alpha(n)\leqslant1.85\right) 
		\\
		r_{0,2}\left(\left(p_1,q_{n}\right)\right) & \,=\, 0 \text{, \ \ \ \quad}1\leqslant n\leqslant10\quad \left(1.66 \leqslant \alpha(n)\leqslant1.75\right) 
		\\
		r_{0,2}\left(\left(p_2,q_{n}\right)\right) & \,=\, 0\text{, \ \ \ \quad}11\leqslant n\leqslant20\quad \left(1.76 \leqslant \alpha(n)\leqslant1.85\right) 
		\\
		r_{0,2}\left(\left(p_2,q_{n}\right)\right) & \,=\, 0.05\text{, \ \ }1\leqslant n\leqslant10\text\quad \left(1.66 \leqslant \alpha(n)\leqslant 1.75\right)
	\end{align*}
	
	By problem \ref{32}, we have \medskip
	
	$\left[\phi_{0}(\text{`M'})\,\barwedge\,\phi_{2}(1.7)\right] \,=\,\{\left(p_1,q_{n}\right) \colon q_{n}\in \mathscr{P}_{2} \,\wedge\,1\leqslant n\leqslant5\}\:$ and $\:T(\phi_{0}(\text{`M'})\,\barwedge\,\phi_{2}(1.7))\,=\,0$. \smallskip
	
	$\left[\phi_{0}(\text{`M'})\,\barwedge\,\phi_{2}(1.75)\right]\,=\,\{\left(p_1,q_{n}\right) \colon q_{n}\in \mathscr{P}_{2} \,\wedge\,1\leqslant n\leqslant10\}\:$ and $\:T(\phi_{0}(\text{`M'})\,\barwedge\,\phi_{2}(1.75))\,=\,0$. \smallskip
	
	$\left[\phi_{0}(\text{`M'})\,\barwedge\,\phi_{2}(1.8)\right] \,=\,\{\left(p_1,q_{n}\right) \colon q_{n}\in \mathscr{P}_{2} \,\wedge\,1\leqslant n\leqslant15\}\:$ and $\:T(\phi_{0}(\text{`M'})\,\barwedge\,\phi_{2}(1.8)) \,=\,0.25$. \medskip
	
	Likewise \medskip
	
	$\left[\phi_{0}(\text{`F'})\,\barwedge\,\phi_{2}(1.7)\right] \,=\,\{\left(p_2,q_{n}\right) \colon q_{n}\in \mathscr{P}_{2} \,\wedge\,1\leqslant n\leqslant5\}\:$ and $\:T(\phi_{0}(\text{`F'})\,\barwedge\,\phi_{2}(1.7)) \,=\,0.25$. \smallskip
	
	$\left[\phi_{0}(\text{`F'})\,\barwedge\,\phi_{2}(1.75)\right]\,=\,\{\left(p_2,q_{n}\right) \colon q_{n}\in \mathscr{P}_{2} \,\wedge\,1\leqslant n\leqslant10\}\:$ and $\:T(\phi_{0}(\text{`F'})\,\barwedge\,\phi_{2}(1.75)) \,=\,0.5$. \smallskip
	
	$\left[\phi_{0}(\text{`F'})\,\barwedge\,\phi_{2}(1.8)\right] \,=\,\{\left(p_2,q_{n}\right) \colon q_{n}\in \mathscr{P}_{2} \,\wedge\,1\leqslant n\leqslant15\}\:$ and $\:T(\phi_{0}(\text{`F'})\,\barwedge\,\phi_{2}(1.8)) \,=\,0.5$. 
\end{proof}

\subsection{Implication and Irrelevancy}

The implication in $\mathfrak{M}$ can model the connective \textit{if} in linguistics, e.g. \textquotedblleft if he is big, he is heavy\textquotedblright, \textquotedblleft if he is quick, he will catch the bus\textquotedblright, and so on. Since a predicate bundle generalizes a probability space, the implication in $\mathfrak{M}$ is closely tied to the conditional probability, $P(B/A)\,=\,P(A\cap B)/P(A)$ in which samples of $A\cap B$ are from $\Omega$ but samples of $B/A$ are from $A$. As the result, $\phi_{1}\barwedge\phi_{2}$ involves the whole bundle, while $\phi_{1}\rightharpoonup\phi_{2}$ only refers to the central bundle of $\phi_{1}$. So we have the following definition.

\begin{definition}
	\ For $x\in\mathscr{D}_1$ and $y\in\mathscr{D}_2$, the \textbf{degree of truth for implication} $\phi_{1}\rightharpoonup\phi_{2}$ is defined as:
	\begin{align}
		T(\phi_{1}(x)\rightharpoonup\phi_{2}(y))  & \,=\, \frac{T(\phi_{1}(x)\barwedge\phi_{2}(y))}{T(\phi_{1}(x))},\,\qquad T(\phi_{1}(x)) \neq 0\label{DegreeTruthImplication}\\
		& \,=\, 1,\qquad\qquad\qquad\qquad\quad T(\phi_{1}(x)) =0\nonumber
	\end{align}
	If $\mathscr{P}_1$ and $\mathscr{P}_2$ consist only of sentences, the \textbf{degree of truth for implication} $\phi_{1}\rightharpoonup\phi_{2}$ is:
	\begin{align}
		T(\phi_{1}\rightharpoonup\phi_{2})  & \,=\, \frac{T(\phi_{1}\barwedge\phi_{2})}{T(\phi_{1})},\qquad\qquad T(\phi_{1}) \neq 0\label{DegreeTruthImplicationSentence}\\
		& \,=\, 1,\:\qquad\qquad\qquad\qquad T(\phi_{1}) =0\nonumber
	\end{align}	
\end{definition}

\begin{remark}
	\ The second condition in (\ref{DegreeTruthImplication}) and (\ref{DegreeTruthImplicationSentence}) makes it consistent with the implication in propositional logic. Also, note that there is no central bundle for implication in general.
\end{remark}

The consistency of (\ref{DegreeTruthImplication}) and (\ref{DegreeTruthImplicationSentence}) are in the following lemmas.

\begin{lemma}
	\ For $x\in\mathscr{D}_1$ and $y\in\mathscr{D}_2$
	\[
	T(\phi_{1}(x)\rightharpoonup\phi_{2}(y))\leqslant 1
	\]
\end{lemma}

\begin{proof}
	\ By corollary \ref{20}(i) and (\ref{DegreeTruthImplication}).
\end{proof}

\begin{lemma}
	\ If $\mathscr{P}_1$ and $\mathscr{P}_2$ consist only of sentences, then
	\[
	T(\phi_{1}\rightharpoonup\phi_{2})\leqslant 1
	\]
\end{lemma}

\begin{proof}
	\ By corollary \ref{15}(i) and (\ref{DegreeTruthImplicationSentence}).
\end{proof}

\begin{corollary}
	\ For $x\in\mathscr{D}$
\end{corollary}

\begin{enumerate}
	\item $T(\phi(x)\rightharpoonup\mathbf{1})\,=\,T(\mathbf{0}\rightharpoonup \phi(x))\,=\,T(\phi(x)\rightharpoonup \phi(x))\,=\,1$
	
	\item $T(\mathbf{1}\rightharpoonup\phi(x))\,=\,T(\phi(x))$
	
	\item $T(\phi(x))\neq 0\,\,\Longrightarrow\,\,T(\phi(x)\rightharpoonup \mathbf{0})\,=\,0$
\end{enumerate}

\begin{proof}
	\ By (\ref{DegreeTruthImplication}) and theorem \ref{44}. \medskip
\end{proof}

(\ref{DegreeTruthImplication}) also implies modus ponens in propositional logic.

\begin{corollary}
	\label{41} \ For $x\in\mathscr{D}_1$ and $y\in\mathscr{D}_2$
	\[
	\left(T(\phi_{1}(x))=1\,\wedge\,T(\phi_{1}(x)\rightharpoonup\phi_{2}(y))=1\right) \,\Longrightarrow\,T(\phi_{2}(y))=1
	\]
\end{corollary}

\begin{proof}
	\ By (\ref{DegreeTruthImplication}), $T(\phi_{1}(x)\barwedge\phi_{2}(y))\,=\,T(\phi_{1}(x))\,=\,1$. So $T(\phi_{2}(y))\,\geqslant\, T(\phi_{1}(x) \barwedge\phi_{2}(y))\,=\,1.$ \smallskip
\end{proof}

\begin{theorem}
	\ Suppose $\varphi\in\mathfrak{M}$. For $x\in\mathscr{D}_1$ and $y\in\mathscr{D}_2$, $\,T(\phi(x))\neq0$ and $T(\varphi(y))\neq0$.
\end{theorem}

\begin{enumerate}
	\item $T(\phi(x))\,=\,T(\phi(x)\barwedge\varphi(y))+T(\phi(x)\,\barwedge\,\urcorner\varphi(y))$
	
	\item $T(\phi(x)\rightharpoonup\varphi(y))+T(\phi(x)\rightharpoonup\urcorner\varphi(y))\,=\,1$
	
	\item $T(\phi(x))\,=\,T(\varphi(y)\rightharpoonup\phi(x))\,T(\varphi(y))+T(\urcorner\varphi(y)\rightharpoonup \phi(x))\,T(\urcorner\varphi(y))$
	
	\item $T(\phi(x)\rightharpoonup\varphi(y))\,=\,\dfrac{T(\varphi(y)\rightharpoonup\phi(x))\,T(\varphi(y))} {T(\varphi(y)\rightharpoonup\phi(x))\,T(\varphi(y))+T(\urcorner\varphi(y)\rightharpoonup\phi(x))\, T(\urcorner\varphi(y))}$
\end{enumerate}

\begin{proof}
	\ (i) \ By theorem \ref{44} and (\ref{EquivalenceMultiValueLogic})
	\[
	\phi(x)\,\rightleftharpoons\,\phi(x)\barwedge\mathbf{1}\,\rightleftharpoons\,\phi(x)\barwedge\left(  \varphi(y)\,\veebar\, \urcorner\varphi(y)\right)\,\rightleftharpoons\,\left(\phi(x)\barwedge\varphi(y)\right)\veebar\left(\phi(x)\,\barwedge\, \urcorner\varphi(y)\right)
	\]
	
	And
	\[
	\left(\phi(x)\barwedge\varphi(y)\right)\barwedge\left(\phi(x)\,\barwedge\,\urcorner\varphi(y)\right)  \,\rightleftharpoons\, \phi(x)\barwedge\left(\varphi(y)\,\barwedge\,\urcorner\varphi(y)\right) \,\rightleftharpoons\,\phi(x)\barwedge\mathbf{0}\, \rightleftharpoons\,\mathbf{0}
	\]
	
	So (i) follows by theorem \ref{34}. \medskip
	
	(ii) and (iii) follow by (i) and (\ref{DegreeTruthImplication}). (iv) follows by (iii). \smallskip
\end{proof}

\begin{corollary}
	\label{47} \ There are $x\in\mathscr{D}_1$ and $y\in\mathscr{D}_2$ that
\end{corollary}

\begin{enumerate}
	\item $\left(\phi_{1}(x)\rightharpoonup\phi_{2}(y)\right)\,\not \leftrightharpoons\,\left(\urcorner\phi_{2}(y) \rightharpoonup\urcorner\phi_{1}(x)\right)$
	
	\item $\left(\phi_{1}(x)\rightharpoonup\phi_{2}(y)\right)\,\not \leftrightharpoons \,\left(\urcorner\phi_{1}(x)\veebar \phi_{2}(y)\right)$
\end{enumerate}

\begin{proof}
	\ (i) \ By theorem \ref{44}(vi) and corollary \ref{12}, there are $x\in\mathscr{D}_1$ and $y\in\mathscr{D}_2$ that
	\[
	T(\urcorner\phi_{2}(y)\rightharpoonup\urcorner\phi_{1}(x))\,=\,\frac{1-T(\phi_{1}(x)\veebar\phi_{2}(y))}{1-T(\phi_{2}(y))}\,\neq\, \frac{T(\phi_{1}(x)\barwedge\phi_{2}(y))}{T(\phi_{1}(x))}\,=\,T(\phi_{1}(x)\rightharpoonup\phi_{2}(y))
	\]

	So (i) follows by (\ref{EquivalenceMultiValueLogic}).\medskip
	
	(ii) is proved similarly. 
\end{proof}

\begin{corollary}
	\label{21} \ There are well-formed formulas in $\mathfrak{M}$ that are invalid.
\end{corollary}

\begin{proof}
	By (\ref{DegreeTruthImplication}), if a well-formed formula in $\mathfrak{M}$ contains multiple implications, it does not have a truth degree since it has no central bundle and is invalid by definition \ref{InvalidFormula}. For example, $\left(\phi_{1}\rightharpoonup\phi_{2}\right) \barwedge\left(\phi_{1}\rightharpoonup\phi_{3}\right)$ and $\left(\phi_{1}\rightharpoonup\phi_{2}\right) \rightharpoonup\phi_{3}$ are well-formed but invalid because their degrees of truth can not be defined.
\end{proof}

\begin{remark}
	\label{33}\ Corollary \ref{47} and \ref{21} show that even some of the basic tautologies in formal logic such as $(p\Rightarrow q)\Leftrightarrow(\urcorner q\Rightarrow \urcorner p)$ and $(p\Leftrightarrow q)\Leftrightarrow(p\Rightarrow q)\wedge(q\Rightarrow p)$ can fail in $\mathfrak{M}$.
\end{remark}

\begin{remark}
	\label{}\ Corollary \ref{21} is consistent with linguistics because even if we commonly hear something like \textquotedblleft if he is big, he is heavy\textquotedblright\ and \textquotedblleft a big and heavy object is slow\textquotedblright\ in a natural language, we never hear something like \textquotedblleft a big object implying a heavy object implies that the object is slow\textquotedblright\ because it is semantically invalid ($\left(\phi_{1}\rightharpoonup\phi_{2}\right) \rightharpoonup\phi_{3}$ has no truth degree).
\end{remark}


\begin{definition}
	\label{DefIrrelevantBundles} \ Suppose\ $\left.\mathscr{P}_{1}\right\vert_{r_{1}}$ and $\left.\mathscr{P}_{2}\right \vert_{r_{2}}$ are two predicate bundles. If for each $p\in\mathscr{P}_{1}$ and $q\in\mathscr{P}_{2}$, $r(\left(p,q\right))\,=\,r_{1}(p)\,r_{2}(q)$, then $\left.\mathscr{P}_{1}\right\vert_{r_{1}}$ and $\left.  \mathscr{P}_{2}\right\vert _{r_{2}}$ are known as \textbf{irrelevant} (\textbf{unrelated}) and $\left.\left(\mathscr{P}_{1}\times \mathscr{P}_{2}\right) \right\vert_{r} \,=\,\left.\mathscr{P}_{1}\right\vert_{r_{1}} \times\left.\mathscr{P}_{2} \right\vert_{r_{2}}$. Otherwise, they are \textbf{relevant} (\textbf{related}).
\end{definition}

\begin{definition}
	\label{DefIrrelevantPredicates} \ If $\,\forall x\forall y (T(\phi_{1}(x)\rightharpoonup\phi_{2}(y))=T(\phi_{2}(y)))$, then $\phi_{1}$ and $\phi_{2}$ are known as \textbf{irrelevant}. Otherwise they are \textbf{relevant}.
\end{definition}

\begin{lemma}
	\label{18} \ $\phi_{1}$ and $\phi_{2}$ are irrelevant if and only if $\,\forall x\forall y(T(\phi_{1}(x)\barwedge \phi_{2}(y)) \,=\, T(\phi_{1}(x)) \,T(\phi_{2}(y)))$.
\end{lemma}

\begin{proof}
	\ By definition \ref{DefIrrelevantPredicates} and (\ref{DegreeTruthImplication}). \bigskip
\end{proof}

The above suggests that the notion of irrelevancy is analogous to independence in probability theory. However, conventionally we call two unrelated concepts irrelevant rather than independent in logic. Often we can determine whether two concepts are irrelevant or not by intuition. For example, adjectives \textit{tall} and \textit{heavy} are related to gender because men are taller and heavier than women on average. However, age is not related to gender in general because each age group has different genders (of the same proportion). So \textit{young} and \textit{old} are unrelated to gender. As another example, adjectives \textit{big} and \textit{heavy} are related because a big object tends to be heavier than a small object.

\begin{theorem}
	\label{27} \ If $\left.\mathscr{P}_{1}\right\vert_{r_{1}}$ and $\left.\mathscr{P}_{2}\right\vert_{r_{2}}$ are irrelevant, $\phi_{1}$ and $\phi_{2}$ are irrelevant.
\end{theorem}

\begin{proof}
	\ By definition \ref{DefIrrelevantBundles} and theorem \ref{19}(i), for $x\in\mathscr{D}_1$ and $y\in\mathscr{D}_2$
	\begin{align*}
		r\left(\left[\phi_{1}(x)\barwedge\phi_{2}(y)\right]\right)  & \:=\: r\left(\left[\phi_{1}(x)\right]\times\left[\phi_{2}(y)\right]\right) 
		\\
		& \:=\: r\left(\underset{p\in\left[\phi_{1}(x)\right]}{\bigcup}\quad \smashoperator{\bigcup_{q\in\left[\phi_{2}(y)\right]}}\,\,\{\left(p,q\right)\}\right) 
		\\
		& \:=\:\underset{p\in\left[\phi_{1}(x)\right]}{\sum}\quad\smashoperator{\sum_{q\in\left[\phi_{2}(y)\right]}}\, r(\left(p,q\right))
		\\
		& \:=\:\quad\smashoperator{\sum_{p\in\left[\phi_{1}(x)\right]}}\,r_{1}(p)\,\, \smashoperator{\sum_{q\in\left[\phi_{2}(y) \right]}}\,r_{2}(q)
		\\
		& \:=\: r_{1}(\left[\phi_{1}(x)\right])\,r_{2}(\left[\phi_{2}(y)\right])
	\end{align*}
	
	So it follows by lemma \ref{18}. \smallskip
\end{proof}

\begin{problem}
	\ Suppose $\phi_{0}(x_{0})$ and $\widetilde{\mathcal{P}}(\phi_{0})$ are given in problem \ref{24}, $\phi_{1}(x_{1})$ and $\widetilde{\mathcal{P}}(\phi_{1})$ are given in problem \ref{2}. Find out if $\phi_{0}$ and $\phi_{1}$ are irrelevant.
\end{problem}

\begin{proof}
	[Solution]\ \ By problem \ref{2}, for any $1\leqslant n\leqslant 20$
	\[
	r_{0,1}\left(\left(p_1,q_{n}\right)\right) \,=\, r_{0,1}\left(\left(p_2,q_{n}\right)\right) \,=\,r_{0}(p_1)\, r_{1}(q_{n}) \,=\,r_{0}(p_2)\,r_{1}(q_{n})
	\]
	
	So by definition \ref{DefIrrelevantBundles} and theorem \ref{27}, $\phi_{0}$ and $\phi_{1}$ are irrelevant.\medskip
	
	For example, as calculated in problem \ref{16} and \ref{2}, $T(\phi_{1}(20))=1$ and $T(\phi_{0}($`M'$) \barwedge\phi_{1}(20)) \,=\, T(\phi_{0}($`F'$)\barwedge\phi_{1}(20))\,=\,0.5$. So
	\[
	T(\phi_{0}(\text{`M'})\rightharpoonup\phi_{1}(20))\,=\,\frac{T(\phi_{0}(\text{`M'})\barwedge\phi_{1}(20))}{T(\phi_{0}(\text{`M'}))}\,=\,1\,=\,T(\phi_{1}(20))
	\]
	
	And $T(\phi_{0}($`F'$)\rightharpoonup\phi_{1}(20))\,=\,1\,=\,T(\phi_{1}(20))$. \medskip
	
	Furthermore, $T(\phi_{0}($`M'$)\rightharpoonup\phi_{1}(30))\,=\,T(\phi_{0} ($`F'$)\rightharpoonup\phi_{1}(30)) \,=\,0.5 \,=\, T(\phi_{1}(30))$. \smallskip
\end{proof}

\begin{problem}
	\ Suppose $\phi_{0}(x_{0})$ and $\widetilde{\mathcal{P}}(\phi_{0})$ are given in problem \ref{24}, $\phi_{2}(x_{2})$ and $\widetilde{\mathcal{P}}(\phi_{2})$ are given in problem \ref{26}. Find out if $\phi_{0}$ and $\phi_{2}$ are irrelevant.
\end{problem}

\begin{proof}
	[Solution] \ By problem \ref{26}, $r_{0,2}\left(\left(p_1,q_{n}\right)\right) \,\neq\,r_{0}(p_1)\,r_{2}(q_{n})$. So $\phi_{0}$ and $\phi_{2}$ are relevant. \medskip
	
	For example, as calculated in problem \ref{32} and \ref{26}, $T(\phi_{2}(1.8))\,=\,0.75$, $T(\phi_{0}($`M'$)\barwedge \phi_{2}(1.8)) \,=\,0.25$ and $T(\phi_{0}($`F'$)\barwedge\phi_{2}(1.8))\,=\,0.5$. So
	\[
	T(\phi_{0}(\text{`M'})\rightharpoonup\phi_{2}(1.8))\,=\,\frac{T(\phi_{0}(\text{`M'})\barwedge\phi_{2}(1.8))} {T(\phi_{0}(\text{`M'}))} \,=\,0.5\,\neq\,T(\phi_{2}(1.8))
	\]
	
	And $T(\phi_{0}($`F'$)\rightharpoonup\phi_{2}(1.8))\,=\,1\,\neq\,T(\phi_{2}(1.8))$.\bigskip
\end{proof}

The above results are consistent with human perception on age and height. For instance, a height of $1.8$ meters has $100$ percent belief of \textit{tall} for women but has only $50$ percent belief of \textit{tall} for men. So gender and height are relevant. However, an age of 20 years old has $100 $ percent belief of \textit{young} for both men and women, and an age of 30 years old has $50$ percent belief of \textit{young} for both men and women. Thus gender and age are irrelevant. \smallskip

\section{Set Bundles and Applications}

\subsection{Introduction}

In modeling the adjectives such as \textit{young} (section \ref{SectionIntroduction}), we can adopt another way by using sets instead of predicates. For instance, \textquotedblleft he is less than $20$ years old\textquotedblright\ is true if and only if $(\{x\} \subset(0,20)$ ($x$ is the age of \textquotedblleft he\textquotedblright), but \textquotedblleft he is young\textquotedblright\ does not have such a representation because many other standards such as less than $21$ or $22$ years old or whatsoever, also qualify for \textit{young}. However, we can extend the object set to a bundle of sets, each of which carries a weight. So let a set bundle for \textit{young} be
\[
\mathcal{C}\,=\,\{I_{n}\colon I_{n}=\left[0,i\right] \,\wedge\,20\leqslant i\leqslant39\,\wedge\,n=i-19\,\wedge\,r(I_{n})=1/20\,\}
\]
where $r(I_{n})$ is the weight for each $I_{n}$ (see Figure \ref{Fig1}). We also assume that the sum of all weights is one, i.e. $\underset{1\leqslant n\leqslant20}{\sum}r(I_{n})=1$. As the result, the degree of truth (or belief) for \textquotedblleft he is young\textquotedblright\ can be described by the weight of all such sets that $x\in I_{n}$, i.e. $T(x)\,=\,\sum r(I_{n})$ that $\{x\}\subset I_{n}$ $\left(\text{or }x\leqslant i\right)$.\smallskip

\begin{figure}[h]
	\centering
	\includegraphics{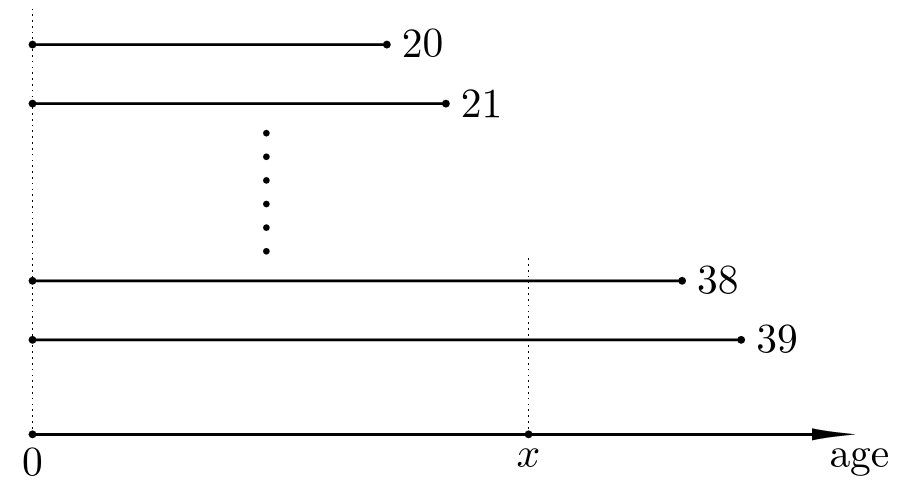}
	\caption{Diagram of a set bundle for \textit{young}.}
	\label{Fig1}
	\centering
\end{figure}

For any $x\leqslant20,$ $x$ is in all $I_{n}$ of $\mathcal{C}$ and so $T(x)=1$, i.e. any age less than $20$ has one hundred percent belief of \textit{young}. On the other hand, for any $x\geqslant40$,\ no $I_{n}$ in $\mathcal{C}$ contains $x$ and so $T(x)=0$, i.e. any age over $40$ has zero percent belief of \textit{young}. If $20<x<40$, it has a belief of \textit{young} between zero and one. For example, if $x=25,$ then $T(x)=0.75$ because three fourths of $\mathcal{C}$ contain $x$.

\subsection{Set Bundle Representation}

In first-order logic, many predicates can be represented by the inclusion of a pair of sets as discussed in the previous section. As a result, such a predicate can be expanded into the multi-valued logic $\mathfrak{M}$ by the inclusion of a set with a bundle of sets. The advantage of set bundles over predicate bundles is that set bundles often have simpler representations and are easier to model than predicate bundles. This is especially true if a bundle is uncountable or joint (setion \ref{ContinuousBundle} and \ref{JointBundle}).\smallskip

Evidently, a bundle of sets is a special case of a predicate bundle (definition \ref{DefPredicateBundle}) in which each predicate is replaced by a set. First, we define the weight measure on a collection of sets.

\begin{definition}
	\label{Sigma_field_SetBundle} \ Suppose $\mathcal{G}$ is a collection of sets. A \textbf{$\sigma-$field} on $\mathcal{G}$ is a nonempty collection of subsets of $\mathcal{G}$ closed under complement and countable unions. 
\end{definition}

\begin{axiom}
	\label{1} \ Suppose $\mathcal{G}$ is a collection of sets and $\sigma(\mathcal{G})$ is a $\sigma-$field on $\mathcal{G}$. The \textbf{weight measure} $r$ on $\mathcal{G}$ is a real-valued function on $\sigma(\mathcal{G})$ that satisfies the following conditions.
\end{axiom}

\begin{enumerate}
	\item \textit{For any} $X\in\mathcal{G},$ $\,r(X)\geqslant0$.
	
	\item $r(\mathcal{G})=1$.
	
	\item \textit{For any (countably many) disjoint} $\mathcal{G}_{i}\in\sigma(\mathcal{G})$
	\[
	r\left(\bigcup_{i\geqslant1}\,\mathcal{G}_{i}\right) \,=\,\sum_{i\geqslant1}\,r(\mathcal{G}_{i})
	\]	
\end{enumerate}

\begin{corollary}
	\label{50} \ Suppose $\mathcal{G}_{1},\mathcal{G}_{2}\in\sigma(\mathcal{G})$. Then
\end{corollary}

\begin{enumerate}
	\item $r(\varnothing)=0$
	
	\item $\mathcal{G}_{2}\subset\mathcal{G}_{1} \,\Longrightarrow \, r(\mathcal{G}_{1}-\mathcal{G}_{2}) \,=\, r(\mathcal{G}_{1}) -r(\mathcal{G}_{2})$
	
	\item $\mathcal{G}_{2}\subset\mathcal{G}_{1} \,\Longrightarrow \, r(\mathcal{G}_{2})\leqslant r(\mathcal{G}_{1})$
	
	\item $\mathcal{G}_{1}\subset\mathcal{G} \,\Longrightarrow \,r(\mathcal{G}_{1})\leqslant1$
\end{enumerate}

\begin{definition}
	\label{DefSetBundle} \ A \textbf{bundle of sets} is a collection of sets with a weight measure and is denoted as $\left.\mathcal{G}\right\vert_{r}=\{X\colon X\in \mathscr{E}\,\wedge\,r(X)\}$, where $r$ is the weight measure of the bundle and $\mathscr{E}$ a universe of discourse. Two set bundles $\left.\mathcal{G}_{1}\right\vert_{r_{1}}$ and $\left.\mathcal{G}_{2}\right\vert_{r_{2}}$ are \textbf{identical} if
	\[
	\left.\mathcal{G}_{1}\right\vert_{r_{1}}=\left.\mathcal{G}_{2}\right\vert_{r_{2}}\,\Longleftrightarrow\,\left(\mathcal{G}_{1} =\mathcal{G}_{2}\,\wedge\,\left(\forall X\in\mathcal{G}_{1}\right) \left(r_{1}(X)\,=\,r_{2}(X)\right) \right)
	\]
\end{definition}

\begin{definition}
	\ $\left.\mathcal{G}_{0}\right\vert_{r}$ is a \textbf{sub-bundle} of $\left.\mathcal{G}\right\vert_{r}$ if $\,\mathcal{G}_{0} \subset\mathcal{G}$. For $\,\left.\mathcal{G}_{1}\right\vert_{r}\subset\left.\mathcal{G}\right\vert_{r}$ and $\,\left.\mathcal{G}_{2}\right\vert_{r}\subset\left.\mathcal{G}\right\vert_{r}$, the \textbf{intersection} and \textbf{union} of $\left.\mathcal{G}_{1}\right\vert_{r}$ and $\left.\mathcal{G}_{2}\right\vert_{r}$ are $\mathcal{G}_{1}\cap\mathcal{G}_{2}$ and $\mathcal{G}_{1}\cup\mathcal{G}_{2}$ respectively.
\end{definition}

\begin{definition}
	\label{DefSetBundleRepresentation}\ \ A \textbf{set bundle representation} of a predicate $\phi$ in $\mathfrak{M}$ is defined as either $\widetilde{\mathcal{S}}(\phi)=\left(A, \left.\mathcal{G}\right\vert_{r},\subset\right)$ specifying whether $A\subset X$ for $X\in\left.\mathcal{G}\right\vert_{r}$, or $\widetilde{\mathcal{S}}(\phi)=\left(A,\left.\mathcal{G} \right\vert_{r},\not \subset\right)$ specifying whether $A\not\subset X$ for $X\in\left.\mathcal{G}\right\vert_{r}$, where $A$ $\left(A\neq\varnothing\right)$ is a \textbf{(subject) set} and $\left.\mathcal{G}\right\vert_{r}$ an \textbf{(object) set bundle}. In general, $\widetilde{\mathcal{S}}(\phi)\,=\,\left(A,\left.\mathcal{G}\right\vert_{r},\bigodot\right)$ where $\bigodot$ denotes $\subset$ or $\not \subset$.
\end{definition}

\begin{definition}
	\label{DefSetBundleRepresentationNegation} \ If $\widetilde{\mathcal{S}}(\phi)\,=\,\left(A,\left.\mathcal{G}\right\vert_{r}, \subset\right)$, the set bundle representation of \textbf{negation }$\urcorner\phi$ is defined as $\widetilde{\mathcal{S}} (\urcorner\phi) \,=\, \left(A,\left.\mathcal{G}\right\vert_{r},\not \subset \right)$. If $\widetilde{\mathcal{S}}(\phi) \,=\,\left(A,\left.\mathcal{G}\right\vert_{r},\not \subset \right)$, $\widetilde{\mathcal{S}}(\urcorner\phi)\, =\, \left(A,\left.\mathcal{G}\right\vert_{r},\subset\right)$.
\end{definition}

Now we define degree of truth (or belief) based on the notion of a central set bundle.

\begin{definition}
	\label{DefCentralClass} \ If $\widetilde{\mathcal{S}}(\phi)=\left(A,\left.\mathcal{G}\right\vert_{r},\subset\right)$, the \textbf{central bundle} of $\phi$ is a sub-bundle of $\left.\mathcal{G}\right\vert_{r}$ that is defined as $\left[\phi\right]  =\{X\colon X\in\mathcal{G}\,\wedge\, A\subset X\}$, and the \textbf{dual central bundle} of $\phi$ is $\left[\phi\right]^{c} = \{X\colon X\in\mathcal{G}\,\wedge\,A\not \subset X\}$. If $\widetilde{\mathcal{S}}(\phi)=\left(A,\left.\mathcal{G} \right \vert_{r},\not \subset \right)$, $\left[\phi\right]=\{X\colon X\in\mathcal{G}\wedge A\not \subset X\}$ and $\left[\phi\right]^{c}= \{X\colon X\in\mathcal{G}\wedge A\subset X\}$.
\end{definition}

\begin{definition}
	\label{DefBeliefSetbundle} \ The \textbf{degree of truth} (or \textbf{belief}) of $\phi$ in $\mathfrak{M}$ is defined to be the weight of the central bundle of $\phi$, i.e. $T(\phi)\,=\,r(\left[\phi\right])$.
\end{definition}

Lemma \ref{11}, corollary \ref{12} and section \ref{SectionEquivalence} also hold for a bundle of sets. Analogous to predicate bundles, the following results hold for the Cartesian products of set bundles.

\begin{definition}
	\label{DefJointSetBundles} \ Suppose $\left.\mathcal{G}_{i}\right\vert _{r_{i}}=\,\{X_{i}\colon X_{i}\in \mathscr{E}_{i}\,\wedge\,r_{i}(X_{i})\}$ $\left(1\leqslant i\leqslant n\right)$ are $n$ set bundles where $\mathscr{E}_{i}$ are universes of discourse. A \textbf{joint bundle} of $\left.\mathcal{G}_{i} \right\vert_{r_{i}}$ is defined as:
	\[
	\left.\left(\mathcal{G}_{1}\times\cdots\times\mathcal{G}_{n}\right)\right\vert_{r}\,=\,\{\left(X_{1},\cdots,X_{n}\right) \colon X_{i}\in\mathcal{G}_{i}\,\wedge\,r\left(  \left(  X_{1},\cdots,X_{n}\right)\right) \,\wedge\,1\leqslant i\leqslant n\,\}
	\]
	Where $r$ is a weight measure on $\sigma(\mathcal{G}_{1}\times\cdots\times\mathcal{G}_{n})$.
\end{definition}

\begin{axiom}
	\label{51} \ Suppose $\left.\mathcal{G}_{i}\right\vert_{r_{i}}$ $\left(1\leqslant i\leqslant n\right)$ are given in definition
	\ref{DefJointSetBundles}. For any $\mathcal{A}_{i}\in\sigma(\mathcal{G}_{i})$, there is
	\[
	r\left(\mathcal{G}_{1}\times\cdots\times\mathcal{A}_{i}\times\cdots\times\mathcal{G}_{n}\right) \,=\,r(\mathcal{G}_{1})\cdots r(\mathcal{A}_{i})\cdots r(\mathcal{G}_{n})\,=\,r(\mathcal{A}_{i})
	\]
\end{axiom}

\begin{corollary}
	\label{22}\quad$r\left(\mathcal{G}_{1}\times\cdots\times\mathcal{G}_{n}\right) \,=\,1$
\end{corollary}

\subsection{Membership of Set Bundles}

In this section, we will introduce membership for a set bundle. First, let's understand what \textquotedblleft$x$ belongs to a set bundle $\left.\mathcal{G}\right\vert_{r}$\textquotedblright\ (or \textquotedblleft$x$ is a member of $\left.\mathcal{G}\right\vert_{r}$\textquotedblright) means. Since for any $Y\in\mathcal{G}$, $x\in Y$ is a formal predicate, \textquotedblleft$x$ belongs to $\left.\mathcal{G}\right\vert_{r}$\textquotedblright\ must be a predicate in $\mathfrak{M}$ because it involves a bundle of (formal) predicates. So let $\phi(x)\rightleftharpoons\,$\textquotedblleft $x$ belongs to $\left.\mathcal{G}\right\vert_{r}$\textquotedblright\ and $\widetilde{\mathcal{P}}(\phi)=\{\left.\mathscr{P}\right\vert_{r},\,\bigcup\mathcal{G}\}$ where $\left.\mathscr{P}\right\vert_{r}=\{p(u)\colon p(u)\Leftrightarrow u\in Y\,\wedge \,Y\in \mathcal{G}\}$. Obviously, the central bundle $\left[\phi(x)\right]$ includes all sets that contain $x$. Thus its degree of truth which is the weight measure of (percentage of) sets containing $x$ can be defined as the membership degree of $x$ in $\mathcal{G}$. So we have

\begin{definition}
	\label{DefMembershipSetbundle} \ Suppose $\left.\mathcal{G}\right\vert_{r}$ is a set bundle and $x\in\bigcup\mathcal{G}$. The \textbf{membership} of $x$ in $\left.\mathcal{G}\right\vert_{r}$ is denoted as $\,x\sbin\left.\mathcal{G}\right\vert_{r}$ which is a predicate in $\mathfrak{M}$, where $\widetilde{\mathcal{P}}(x\sbin\mathcal{G})=\{\left.\mathscr{P}\right\vert_{r},\, \bigcup\mathcal{G}\}$ and $\left.\mathscr{P}\right\vert_{r}=\{p(u)\colon p(u)\Leftrightarrow u\in Y\,\wedge \,Y\in \mathcal{G}\}$.  The \textbf{membership degree} of $x$ in $\left.\mathcal{G}\right\vert_{r}$ is defined as $\,T(x\sbin\mathcal{G})$.
\end{definition}

\begin{lemma}
	\label{29} \quad $\left[x\sbin\mathcal{G}\right]\,=\,\{Y\colon x\in Y\,\wedge\,Y\in\mathcal{G}\}$
\end{lemma}

\begin{proof}
	\ Let $\left.\mathscr{P}\right\vert_{r}=\{p(u)\colon p(u)\Leftrightarrow u\in Y\,\wedge \,Y\in \mathcal{G}\}$.
	By definition \ref{DefMembershipSetbundle}, $\left[x\sbin\mathcal{G}\right]=\{p(u)\colon \allowbreak \vdash p(x)\wedge p(u)\in \mathscr{P}\}$. So it follows by
	\[
	p(x)\in\mathscr{P}\:\Longleftrightarrow\: x\in Y\,\wedge\,Y\in\mathcal{G}\:\Longleftrightarrow\: Y\in \left[x\sbin\mathcal{G}\right]\vspace{-18pt}
	\]
\end{proof}

\begin{figure}[h]
	\centering
	\includegraphics{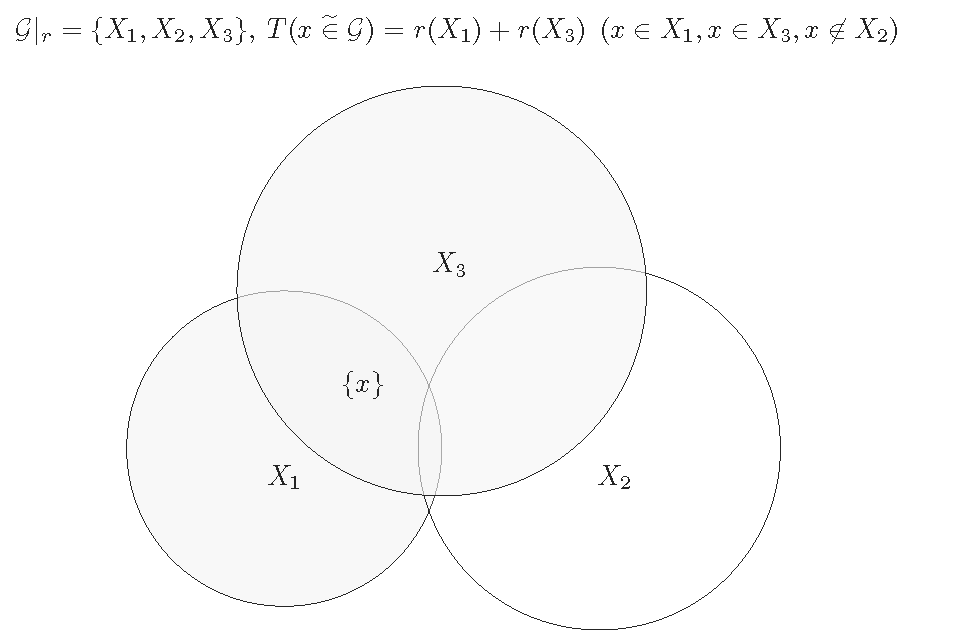}
	\caption{Diagram for membership of a set bundle.}
	\label{Fig2}
	\centering
\end{figure}

A diagram for membership of a set bundle is shown in Figure \ref{Fig2}. 

\begin{lemma}
	\label{43} \ $T(x\sbin\mathcal{G})$ \textit{is the percentage (weight) of the sets in }$\left.\mathcal{G}\right\vert_{r} $\textit{\ that contain }$x$.\footnote{Together with definition \ref{DefMembershipSetbundle}, we can see that $\mathcal{G}$ can be viewed as a rigorous model of a fuzzy set with $T(x\sbin\mathcal{G})$ being its fuzzy membership. See section \ref{SectionSetBundleFuzzySet} for more detail.}
\end{lemma}

\begin{proof}
	\ \ By definition \ref{DefMembershipSetbundle}, $\left[x\sbin\mathcal{G}\right]$ is all sets in $\left.\mathcal{G}\right\vert_{r}$ that contain $x$. So by definition \ref{DefBeliefSetbundle}, $T(x\sbin\mathcal{G})$ is the percentage (weight) of the sets in $\left.\mathcal{G}\right\vert_{r}$\ that contain $x$. \bigskip
\end{proof}

Next, we will discuss membership degrees of intersection and union of two set bundles.

\begin{theorem}
	\label{39} \ Suppose $\vv{\mathcal{G}}\vert_{r}=\left.\left(\mathcal{G}_{1} \times\cdots\times \mathcal{G}_{n}\right)\right\vert_{r}$, $\mathcal{C}_{1},\mathcal{C}_{2}\in\sigma(\vv{\mathcal{G}})$.\footnote{In the rest discussion, we assume that $\mathcal{G}_{1},\cdots, \mathcal{G}_{n}$ are set bundles, and $\mathcal{C}_{1}, \cdots, \mathcal{C}_{n}$ are in $\sigma(\vv{\mathcal{G}})$.}
\end{theorem}

\begin{enumerate}
	\item $\left[x\sbin\mathcal{C}_{1}\cup\mathcal{C}_{2}\right]  \,=\,\left[x\sbin\mathcal{C}_{1}\right] \cup \left[x\sbin\mathcal{C}_{2}\right]$
	
	\item $\left[x\sbin\mathcal{C}_{1}\cap\mathcal{C}_{2}\right]  \,=\,\left[x\sbin\mathcal{C}_{1}\right] \cap \left[x\sbin\mathcal{C}_{2}\right]$
	
	\item $\mathcal{C}_{2}\subset\mathcal{C}_{1}\,\Longrightarrow\,\left[x\sbin\mathcal{C}_{1}-\mathcal{C}_{2}\right] \,=\, \left[x\sbin\mathcal{C}_{1}\right] -\left[x\sbin\mathcal{C}_{2}\right]$
	
	\item $\mathcal{C}_{2}\subset\mathcal{C}_{1}\,\Longrightarrow\,\left[x\sbin\mathcal{C}_{2}\right] \subset \left[x\sbin\mathcal{C}_{1}\right]$
	
	\item $\left[x\sbin\mathcal{C}_{1}\cup\mathcal{C}_{2}\right]  \,=\,\left[x\sbin\mathcal{C}_{1}\right] \cup \left(\left[x\sbin\mathcal{C}_{2}\right] -\left[x\sbin\mathcal{C}_{1}\right] \cap\left[x\sbin\mathcal{C}_{2}\right]\right)$
\end{enumerate}

\begin{proof}
	\ (i) \ By lemma \ref{29}
	\begin{align*}
		Y\in\left[x\sbin\mathcal{C}_{1}\cup\mathcal{C}_{2}\right] & \: \Longleftrightarrow \: x\in Y\,\wedge\,Y\in\mathcal{C}_{1} \cup\mathcal{C}_{2}
		\\
		& \: \Longleftrightarrow \: \left(x\in Y\wedge\,Y\in\mathcal{C}_{1}\right) \,\vee\, \left(x\in Y\wedge\,Y\in\mathcal{C}_{2}\right) 
		\\
		& \: \Longleftrightarrow \: Y\in\left[x\sbin\mathcal{C}_{1}\right]  \cup \left[x\sbin\mathcal{C}_{2}\right]
	\end{align*}
	
	(ii) is proved similar to (i). \medskip
	
	(iii) \ Since $\,\mathcal{C}_{1}=\mathcal{C}_{2}\cup\left(\mathcal{C}_{1}-\mathcal{C}_{2}\right)\,$ and $\,\mathcal{C}_{2}\cap\left(\mathcal{C}_{1}-\mathcal{C}_{2}\right)=\varnothing$, by (i) and (ii)
	\[
	\left[x\sbin\mathcal{C}_{1}\right] = \left[x\sbin\mathcal{C}_{2}\right] \cup \left[x\sbin\mathcal{C}_{1} -\mathcal{C}_{2}\right] \text{\ \ and \ }\left[x\sbin\mathcal{C}_{2}\right] \cap \left[x\sbin\mathcal{C}_{1} -\mathcal{C}_{2}\right] =\, \varnothing
	\]
	
	So $\left[x\sbin\mathcal{C}_{1}-\mathcal{C}_{2}\right] = \left[x\sbin\mathcal{C}_{1}\right] -\left[x\sbin\mathcal{C}_{2}\right]$. \medskip
	
	(iv) \ By (iii).\medskip
	
	(v) \ Since $\,\mathcal{C}_{1}\cup\mathcal{C}_{2} \,=\,\mathcal{C}_{1}\cup\left(\mathcal{C}_{2}-\mathcal{C}_{1} \cap\mathcal{C}_{2} \right)$, it follows by (i), (ii) and (iii). \smallskip
\end{proof}

\begin{corollary}
	\label{40}\qquad
\end{corollary}

\begin{enumerate}
	\item $T(x\sbin\mathcal{C}_{1}\cup\mathcal{C}_{2})\,=\,T(x\sbin\mathcal{C}_{1}) + T(x\sbin\mathcal{C}_{2}) - T(x\sbin\mathcal{C}_{1}\cap\mathcal{C}_{2})$
	
	\item $T(x\sbin\mathcal{C}_{1}\cup\mathcal{C}_{2})\,=\,T(x\sbin\mathcal{C}_{1}) + T(x\sbin\mathcal{C}_{2}) - r\left(\left[x\sbin\mathcal{C}_{1}\right] \cap\left[x\sbin\mathcal{C}_{2}\right]\right)$
\end{enumerate}

\begin{proof}
	\ (i) \ By definition \ref{DefBelief}, lemma \ref{14} and theorem \ref{39}.\medskip
	
	(ii) \ By (i) and theorem \ref{39}(ii).\smallskip
\end{proof}

\begin{corollary}
	\label{31}\qquad
\end{corollary}

\begin{enumerate}
	\item $T(x\sbin\mathcal{C}_{1}\cap\mathcal{C}_{2})\,\leqslant\,\min\{T(x\sbin\mathcal{C}_{1}),\,T(x\sbin\mathcal{C}_{2})\}$
	
	\item $T(x\sbin\mathcal{C}_{1}\cup\mathcal{C}_{2})\,\geqslant\,\max\{T(x\sbin\mathcal{C}_{1}),\,T(x\sbin\mathcal{C}_{2})\}$
\end{enumerate}

\begin{proof}
	\ By lemma \ref{14} and theorem \ref{39}.\smallskip
\end{proof}

\begin{corollary}
	\label{42} \ If \ $r(\left[x\sbin\mathcal{C}_{1}\right] \cap \left[x\sbin\mathcal{C}_{2}\right]) \,=\, r(\left[x\sbin\mathcal{C}_{1}\right]) \, r(\left[x\sbin\mathcal{C}_{2}\right])$, then
\end{corollary}

\begin{enumerate}
	\item $T(x\sbin\mathcal{C}_{1}\cap\mathcal{C}_{2})\,=\,T(x\sbin\mathcal{C}_{1})\,T(x\sbin\mathcal{C}_{2})$
	
	\item $T(x\sbin\mathcal{C}_{1}\cup\mathcal{C}_{2})\,=\,1-\left(1-T(x\sbin\mathcal{C}_{1})\right) \hspace{-0.02in} \left(1-T(x\sbin\mathcal{C}_{2})\right)$
\end{enumerate}

\begin{proof}
	\ (i) \ By theorem \ref{39}(ii).\medskip
	
	(ii) \ By (i) and corollary \ref{40}(ii).\smallskip
\end{proof}

\begin{corollary}
	\label{49} \ In general, $T(x\sbin\mathcal{C}_{1}\cap\mathcal{C}_{2})$ and $\,T(x\sbin\mathcal{C}_{1}\cup\mathcal{C}_{2})$ can not be expressed by functions of $\,T(x\sbin\mathcal{C}_{1})$ and $\,T(x\sbin\mathcal{C}_{2})$. 
\end{corollary}

\begin{proof}
	\ Suppose $\left.\mathcal{G}\right\vert_{r}$ is given in problem \ref{71}. Let $\mathcal{C}_{1}=\{I_{n}\colon1\leqslant n\leqslant 10\} \in\sigma(\mathcal{G})$ and $\mathcal{C}_{2}=\{I_{n}\colon1+k\leqslant n\leqslant 10+k\,\wedge\,0\leqslant k\leqslant10\}\in \sigma(\mathcal{G})$. Then $r(\mathcal{C}_{1})=0.5$ and for any $0\leqslant k\leqslant10$, $r(\mathcal{C}_{2})=0.5$. So $T(20\sbin\mathcal{C}_{1})\,=\,T(20\sbin\mathcal{C}_{2})\,=\,0.5$. \medskip
	
	For any function $f$,  $f(T(20\sbin\mathcal{C}_{1}),\,T(20\sbin\mathcal{C}_{2}))$ has only one value. But $T(20\sbin\mathcal{C}_{1} \cap \mathcal{C}_{2})$ can range over $0$, $0.05$, $0.1$, $\cdots,0.5$. Thus $\,T(x\sbin\mathcal{C}_{1} \cap \mathcal{C}_{2})\neq f(T(x\sbin\mathcal{C}_{1}),\,T(x\sbin\mathcal{C}_{2}))$. $T(x\sbin\mathcal{C}_{1}\cup\mathcal{C}_{2})$ follows by corollary \ref{40}.\bigskip
\end{proof}

\subsection{Discrete Bundles and Belief Distributions}

In this section, we will study discrete set bundles and belief distributions. 

\begin{definition}
	\ A \textbf{discrete set bundle} is a collection of countably many sets with a weight measure. A \textbf{discrete belief  distribution} is a mathematical function that gives the beliefs in a discrete set bundle representation.
\end{definition}

\begin{theorem}
	\label{70} \ Suppose $\left.\mathcal{G}\right\vert_{r}\,=\,\{I_{n}\colon r(I_{n})=r_{n}\,\wedge\,n\in \mathbb{N}\}$ where $I_{n}$ are intervals on $\Bbb{R}$, $\alpha$ is a bounded non-decreasing right continuous function and $\beta$ a bounded non-increasing left continuous function in $\Bbb{R}^{\Bbb{N}}$, and $x\in\Bbb{R}$.
\end{theorem}

\begin{enumerate}
	\item \textit{If $I_{n}=[c,\alpha(n)]\,\left(c\leqslant a\leqslant\alpha(n)\leqslant b\right)$, let $k$ be the least integer that $\,c\leqslant x\leqslant\alpha(k)$. Then}
	\begin{align*}
		T(x\sbin\mathcal{G})  &  \:=\:\sum_{n\geqslant k} r_{n},\text{ \ \ \ \ \ }a\leqslant x\leqslant b
		\\
		& \:=\:1,\qquad\qquad c\leqslant x\leqslant a
		\\
		& \:=\:0,\qquad\qquad x<c\,\vee\,x>b
	\end{align*}
	\item \textit{If $I_{n}=[\alpha(n),c]\,\left(a\leqslant\alpha(n)\leqslant b\leqslant c\right)$, let $k$ be the biggest integer that $\,\alpha(k)\leqslant x\leqslant c$. Then}
	\begin{align*}
		T(x\sbin\mathcal{G})  &  \:=\:\sum_{n\leqslant k}r_{n},\text{ \ \ \ \ \ }a\leqslant x\leqslant b
		\\
		& \:=\:1,\qquad\qquad b\leqslant x\leqslant c
		\\
		& \:=\:0,\qquad\qquad x<a\,\vee\,x>c
	\end{align*}
	\item \textit{If $I_{n}=\left[\alpha(n),\,\beta(n)\right]$ and $\,a\leqslant\alpha(n)\leqslant b\leqslant c\leqslant\beta(n) \leqslant d$, let $k$ be the biggest integer that $\,\alpha(k)\leqslant x\leqslant b$ and $\,m$ be the biggest integer that $\,c\leqslant x\leqslant\beta(m)$. Then}
	\begin{align*}
		T(x\sbin\mathcal{G})  & \:=\:\sum_{n\leqslant k}r_{n},\text{ \ \ \ \ \ }a\leqslant x\leqslant b 
		\\
		& \:=\:\sum_{n\leqslant m}r_{n},\text{ \ \ \ \ }\,c\leqslant x\leqslant d
		\\
		& \:=\:1,\qquad\qquad \;b\leqslant x\leqslant c
		\\
		& \:=\:0,\qquad\qquad \;x<a\,\vee\,x>d
	\end{align*}
\end{enumerate}

\begin{proof}
	\ (i) \ If $a\leqslant x\leqslant b$, $x\in I_{n}$ for any $n\geqslant k$ since $\alpha(n)\geqslant\alpha(k)$. Also $x\notin I_{n}$ for any $n<k$. So by definition \ref{DefMembershipSetbundle} and \ref{DefBeliefSetbundle}, $\left[x\sbin\mathcal{G}\right] \,=\,\{I_{n}\colon n\geqslant k\}$ and $T(x\sbin\mathcal{G})\,=\underset{n\geqslant k}{\sum}r_{n}$. If $c\leqslant x\leqslant a$, $\left[x\sbin\mathcal{G}\right] =\mathcal{G}$ and $T(x\sbin\mathcal{G})=1$. If $x<c$ or $x>b$, $\left[x\sbin\mathcal{G}\right] =\varnothing$ and $T(x\sbin\mathcal{G})=0$. \medskip

	(ii) is a special case of (iii).\medskip
	
	(iii) \ First, if $\alpha(k)\leqslant x\leqslant b$, $x\in I_{n}$ for any $n\leqslant k$ since $\alpha(n)\leqslant\alpha(k)$, and $x\notin I_{n}$ for any $n>k$. So $\left[x\sbin\mathcal{G}\right]  \,=\,\{I_{n}\colon n\leqslant k\}$ and $T(x\sbin\mathcal{G}) \,=\underset{n\leqslant k}{\sum}r_{n}$. If $c\leqslant x\leqslant\beta(m)$, $x\in I_{n}$ for any $n\leqslant m$ since $\beta(n)\geqslant\beta(m)$, and $x\notin I_{n}$ for any $n>m$. So $\left[x\sbin \mathcal{G}\right] \,=\,\{I_{n}\colon n\leqslant m\}$ and $T(x\sbin\mathcal{G})\,=\underset{n\leqslant m}{\sum}r_{n}$. Obviously, for $b\leqslant x\leqslant c$, $T(x\sbin\mathcal{G})\,=\,1$, and for $x<a$ or $x>d$, $T(x\sbin\mathcal{G})\,=\,0$. \bigskip
\end{proof}

\begin{figure}[h]
	\centering
	\includegraphics{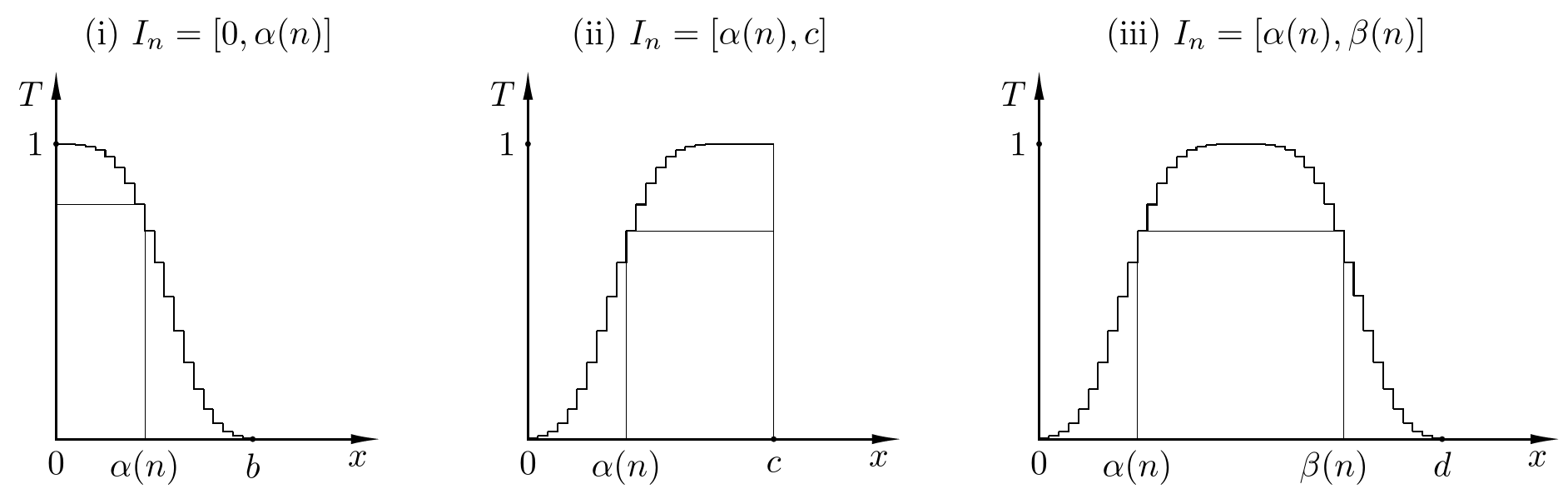}
	\caption{Diagrams of belief distributions in theorem \ref{29}.}
	\label{Fig3}
\end{figure}

The belief distributions ($T(x\sbin\mathcal{G})$ or $T(x)$) of theorem \ref{70} are shown in Figure \ref{Fig3}. These three types of discrete set bundles are important because they can model an adjective whose $T(x)$ is monotonically increasing (or decreasing). There are many such adjectives in linguistics as \textit{young}, \textit{old}, \textit{middle-aged}, and so on. In addition, the belief distributions for a non-singleton set are as follows.

\begin{theorem}
	\label{73}\ Suppose everything is the same as in theorem \ref{70} and $\widetilde{\mathcal{S}}(\phi(A))=(A, \left.\mathcal{G}\right\vert_{r},\subset)$ where $A$ is a non-singleton subset of $\Bbb{R}$. 
\end{theorem}

\begin{enumerate}
	\item \textit{If $I_{n}=[c,\alpha(n)]\,\left(c\leqslant a\leqslant\alpha(n)\leqslant b\right)$, let $k$ be the least integer that $\,c\leqslant \max\{A\}\leqslant\alpha(k)$. Then}
	
	\begin{align*}
		T(\phi)  &  \:=\:\sum_{n\geqslant k} r_{n},\text{ \ \ \ \ }\,c\leqslant \min\{A\}  
		\\
		& \:=\:1,\qquad\qquad c\leqslant \min\{A\} <\max\{A\}\leqslant a  
		\\
		& \:=\:0,\qquad\qquad \min\{A\}<c\,\vee\,\max\{A\}>b
	\end{align*}
	\item \textit{If $I_{n}=[\alpha(n),c]\,\left(a\leqslant\alpha(n)\leqslant b\leqslant c\right)$, let $k$ be the biggest integer that $\,\alpha(k)\leqslant \min\{A\}\leqslant b$. Then}
	\begin{align*}
		T(\phi)  &  \:=\:\sum_{n\leqslant k}r_{n},\text{ \ \ \ \ }\max\{A\}\leqslant c
		\\
		& \:=\:1,\qquad\qquad b\leqslant \min\{A\} <\max\{A\}\leqslant c
		\\
		& \:=\:0,\qquad\qquad \min\{A\}<a\,\vee\,\max\{A\}>c
	\end{align*}
	\item \textit{If $I_{n}=\left[\alpha(n),\,\beta(n)\right]$ and $\,a\leqslant\alpha(n)\leqslant b\leqslant c\leqslant\beta(n) \leqslant d$, let $k$ be the biggest integer that $\,\alpha(k)\leqslant \min\{A\}\leqslant b$, $\,m$ be the biggest integer that $\,c\leqslant \max\{A\}\leqslant\beta(m)$, and $l=\min\{k,m\}$. Then}
	\begin{align*}
		T(\phi)  & \:=\:\sum_{n\leqslant l}r_{n},\text{ \ \ \ \ }\,\alpha(k)\leqslant \min\{A\}<\max\{A\}\leqslant \beta(m) 
		\\
		& \:=\:1,\qquad\qquad b\leqslant \min\{A\} <\max\{A\}\leqslant c
		\\
		& \:=\:0,\qquad\qquad \min\{A\}<a\,\vee\,\max\{A\}>d
	\end{align*}
\end{enumerate}

\begin{proof}
	\ (i) \ If $c\leqslant \min\{A\}<\max\{A\}\leqslant \alpha(k),\,A\subset I_{n}$ for any $n\geqslant k$ since $\alpha(n)\geqslant\alpha(k)$. Also, $A\not\subset I_{n}$ for any $n<k$. So by definition \ref{DefBeliefSetbundle}, $[\phi] =\{I_{n}\colon n\geqslant k\}$ and $T(\phi)=\underset{n\geqslant k}{\sum}r_{n}$. If $c\leqslant \min\{A\}$ and $\max\{A\}\leqslant a,\,[\phi] =\mathcal{G}$ and $T(\phi)=1$. If $\min\{A\}<c$ or $\max\{A\}>b,\,[\phi] =\varnothing$ and $T(\phi)=0$. \medskip
	
	(ii) \ If $\alpha(k)\leqslant \min\{A\}<\max\{A\}\leqslant c,\,A\subset I_{n}$ for any $n\leqslant k$ since $\alpha(n)\leqslant\alpha(k)$, and $A\not\subset I_{n}$ for any $n>k$. So $[\phi] =\{I_{n}\colon n\leqslant k\}$ and $T(\phi)=\underset{n\leqslant k}{\sum}r_{n}$. If $b\leqslant \min\{A\}$ and $\max\{A\}\leqslant c,\,[\phi] =\mathcal{G}$ and $T(\phi)=1$. If $\min\{A\}<a$ or $\max\{A\}>c,\,[\phi] =\varnothing$ and $T(\phi)=0$. \medskip
	
	(iii) \ Since $k,m\geqslant l$, for any $n\leqslant l,\,\alpha(n)\leqslant\alpha(l)$ and $\beta(n)\geqslant\beta(l)$, i.e. $A\subset I_n$. Likewise, for any $n>l,\,A\not\subset I_n$. So $[\phi] =\{I_{n}\colon n\leqslant l\}$ and $T(\phi)\,=\underset{n\leqslant l}{\sum}r_{n}$. Obviously, for $b\leqslant \min\{A\}$ and $\max\{A\}\leqslant c,\,T(\phi)=1$. If $\min\{A\}<a$ or $\max\{A\}>d,\,T(\phi)=0$. \bigskip
\end{proof}

Problem \ref{16} and \ref{32} can be done by set bundles as follows.

\begin{problem}
	\label{71} \ Suppose $\phi_{1}(x_{1})\rightleftharpoons\,$\textquotedblleft$\,x_{1}$ is young\textquotedblright\ ($x_{1}$ is an age for the human)\ and $\widetilde{\mathcal{S}}(\phi_{1}(x_{1}))=(\{x_{1}\},\left.\mathcal{G}_{1} \right\vert_{r_{1}},\subset)$, where $\left.\mathcal{G}_{1}\right\vert_{r_{1}}\,=\,\{I_{n}\colon I_{n}=[0,\alpha(n)] \,\wedge\, \alpha(n)=n+19 \,\wedge\,1\leqslant n\leqslant20\,\wedge\,r_{1}(I_{n})=0.05\,\}$. Find $T(\phi_{1}(x_{1}))$ and $T(\urcorner\phi_{1}(x_{1}))$ for $x_{1}\,=\,20$, $30$, $40\left(\text{years old}\right)$.
\end{problem}

\begin{proof}
	[Solution]\ \ Clearly, $20\leqslant\alpha(n)\leqslant39$. Since for each $I_{n}\in\mathcal{G}_{1},\,20\in I_{n},\,\left[  \phi_{1}(20)\right] =\mathcal{G}_{1}$. So by definition \ref{DefBeliefSetbundle}, $T(\phi_{1}(20))=1$. Also, $T(\phi_{1}(40))=0$ for $\left[\phi_{1}(40)\right]=\varnothing$. Since the least $n$ that $30\in I_{n}$ is $11$, by theorem \ref{70}(i), $\left[\phi_{1}(30)\right] \,=\,\{I_{n}\colon I_{n}\in\mathcal{G}_{1} \,\wedge\,11\leqslant n\leqslant20\}$ and $T(\phi_{1}(30)) =0.5$.\medskip
	
	Since $\urcorner\phi_{1}(x_{1})\,\rightleftharpoons$ \textquotedblleft$x_{1}$ is not young\textquotedblright, $\,\widetilde{\mathcal{S}}(\urcorner\phi_{1}(x_{1}))\,=\,\left(\{x_{1}\},\left.\mathcal{G}_{1}\right\vert_{r_{1}},\not \subset \right)$. So $\left[\urcorner\phi_{1}(20)\right]=\varnothing$, $\left[\urcorner\phi_{1}(40)\right] =\mathcal{G}_{1}$, $\left[\urcorner\phi_{1}(30)\right] \,=\,\mathcal{G}_{1}-\left[\phi_{1}(30)\right]$, and $T(\urcorner\phi_{1}(20))=0$, $T(\urcorner\phi_{1}(40))=1$, $T(\urcorner\phi_{1}(30))=0.5$.
\end{proof}

\begin{problem}
	\ Suppose $\phi_{1}(A)\rightleftharpoons\,$\textquotedblleft$A$ is young\textquotedblright\ ($A$ is a collection of ages for the human)\ and $\widetilde{\mathcal{S}}(\phi_{1}(A))=(A,\left.\mathcal{G}_{1} \right\vert_{r_{1}},\subset)$ where $\left.\mathcal{G}_{1}\right\vert_{r_{1}}$ is the same as in problem \ref{71}. Find $T(\phi_{1}(A))$ for $A\,=\,\{21, 28, 35\}$.
\end{problem}

\begin{proof}
	[Solution] \ Clearly, $\min\{A\}=21$ and $\max\{A\}=35$. Since $16$ is the least $n$ that $A\subset I_{n}$, $\left[\phi_{1}(A)\right]\,=\,\{I_n\colon I_{n}\in\mathcal{G}_{1}\,\wedge\,16\leqslant n\leqslant20\}$. So by theorem \ref{73}(i), $T(\phi_{1}(A))=0.25$. 
\end{proof}

\begin{problem}
	\label{72} \ Suppose $\phi_{2}(x_{2})\rightleftharpoons\,$\textquotedblleft $\,x_{2}$ is tall\textquotedblright\ ($x_{2}$ is a height for the human)\ and $\widetilde{\mathcal{S}}(\phi_{2}(x_{2}))=(\{x_{2}\},\left.\mathcal{G}_{2} \right\vert_{r_{2}},\subset)$, where $\left.\mathcal{G}_{2}\right\vert_{r_{2}}\,=\,\{I_{n}\colon I_{n}=[\alpha(n),3] \,\wedge\, \alpha(n) =n/100+1.65 \,\wedge\, 1\leqslant n\leqslant 20\,\wedge\,r_{2}(I_{n})=0.05\,\}$.\footnote{We assume that human height is no more than $3$ meters.} Find $T(\phi_{2}(x_{2}))$ for $x_{2}\,=\,1.7$, $1.75$, $1.8\left(m\right)$.
\end{problem}

\begin{proof}
	[Solution] \ Obviously, $1.66\leqslant\alpha(n)\leqslant1.85$. So by theorem \ref{70}(ii), $T(\phi_{2}(1.7))=0.25$, $\ T(\phi_{2}(1.75)) \,=\,0.5$, $\ T(\phi_{2}(1.8))\,=\,0.75$.\smallskip
\end{proof}

\subsection{Continuous Bundles and Belief Distributions}\label{ContinuousBundle}

In this section, we will investigate continuous set bundles.

\begin{definition}
	\ A \textbf{continuous set bundle} is a collection of uncountably many sets with a weight measure. A \textbf{continuous belief  distribution} is a mathematical function that gives the beliefs in a continuous set bundle representation.
\end{definition}

First note that the weight measure for a single set in a continuous bundle $\mathcal{G}\vert_{r}$ is always zero, i.e. for any $X\in \mathcal{G}\vert_{r},\:r(X)=0$. Therefore, the weight measure can be positive only for a collection of uncountably many sets. Consequently, we need the notion of a weight density function similar to a probability density function (for a continuous random variable).

\begin{definition}
	\label{DefWeightDensityFunction} \ Suppose $\left.\mathcal{G}\right\vert_{r} \,=\,\{I_{t}\colon t\in\left[t_{a},t_{b}\right] \subset \Bbb{R}\}$ where $I_{t}=[c,\alpha(t)]$ or $[\alpha(t),\,c]$, $\alpha$ is a non-decreasing  right continuous function in $\left[a,b\right]^{\Bbb{R}}$ $\left(\left[a,b\right] \subset \Bbb{R}\right)$, and $f$ is non-negative on $[t_a,t_b]$ and zero elsewhere. Then $f$ is a \textbf{weight density function} of $\left.\mathcal{G}\right\vert_{r}$ if its weight measure on every interval $[t_{1}, t_{2}]\subset\Bbb{R}$ is given by the following Lebesgue--Stieltjes integral,
	\begin{equation}
		r\left(\{I_{t}\colon t\in\left[t_{1},t_{2}\right] \}\right) \,=\int_{t_{1}}^{t_{2}}f(t)\, d\alpha(t) \label{WeightDensityFunction}
	\end{equation}
	In particular, $r\left(\{I_{t}\colon t\in\left[t_{a},t_{b}\right]\}\right) \,=\,1$.
\end{definition}

Similar to the discrete case, we only consider the following three continuous bundles.

\begin{theorem}
	\label{74} \ Suppose $\left.\mathcal{G}\right\vert_{r}\,=\,\{I_{t}\colon t\in\left[t_{a},t_{b}\right]\subset\Bbb{R}\}$, $f$ and $f'$ are weight density functions of $\left.\mathcal{G}\right\vert_{r}$, $\alpha$ is given in definition \ref{DefWeightDensityFunction}, $\beta$ is a non-increasing left continuous function in $\left[a,b\right]^{\Bbb{R}}$ $\left(\left[a,b\right] \subset \Bbb{R}\right)$, and $x\in\Bbb{R}$.
\end{theorem}

\begin{enumerate}
	\item \textit{If $I_{t}\,=\,[c,\alpha(t)]$ and $\,c\leqslant a\leqslant\alpha(t)\leqslant b$, then}
	\begin{align*}
		T(x\sbin\mathcal{G}) & \:=\,\int_{t_x}^{t_{b}}f(t)\,d\alpha(t),\text{\ \ \ \ \ \ \ }\, a\leqslant\alpha(t_x-0)\leqslant x\leqslant\alpha(t_x+0)\leqslant b 
		\\ 
		& \:=\:1,\qquad\qquad\qquad\qquad c\leqslant x\leqslant a
		\\
		& \:=\:0,\qquad\qquad\qquad\qquad x<c\,\vee\,x>b
	\end{align*}
	
	\item \textit{If $I_{t}\,=\,[\alpha(t),c]$ and $\,a\leqslant\alpha(t) \leqslant b\leqslant c$, then}
	\begin{align*}
		T(x\sbin\mathcal{G}) & \:=\,\int_{t_{a}}^{t_x}f(t)\,d\alpha(t),\text{\ \ \ \ \ \ \ }\,a\leqslant\alpha(t_x-0) \leqslant x\leqslant\alpha(t_x+0)\leqslant b
		\\
		& \:=\:1,\qquad\qquad\qquad\qquad b\leqslant x\leqslant c
		\\
		& \:=\:0,\qquad\qquad\qquad\qquad x<a\,\vee\,x>c
	\end{align*}
	
	\item \textit{If $I_{t}\,=\left[\alpha(t),\,\beta(t)\right]$ and $\,a\leqslant\alpha(t)\leqslant b\leqslant c\leqslant\beta(t)\leqslant d$, then}
	\begin{align*}
		T(x\sbin\mathcal{G}) & \:=\,\int_{t_{a}}^{t_x}f(t)\,d\alpha(t),\text{\ \ \ \ \ \ \ \ }\,a\leqslant\alpha(t_x-0) \leqslant x\leqslant\alpha(t_x+0)\leqslant b
		\\
		& \:=\,\int_{t_x}^{t_{a}}f'(t)\,d\beta(t),\text{ \ \ \ \ \ \ }\,c\leqslant\beta(t_x+0)\leqslant x\leqslant \beta(t_x-0)\leqslant d
		\\
		& \:=\:1,\qquad\qquad\qquad\qquad\ b\leqslant x\leqslant c
		\\
		& \:=\:0,\qquad\qquad\qquad\qquad\ x<a\,\vee\,x>d
	\end{align*}
\end{enumerate}

\begin{proof}
	\ (i) \ Since $\alpha(t_{a})=a$ and $\alpha(t_{b})=b$, if $a\leqslant x\leqslant b$, there is a $t_x\in \left[t_{a},t_{b}\right]$ that $\alpha(t_x-0) \leqslant x\leqslant\alpha(t_x+0)$ (see Figure \ref{Fig4}). So for any $t>t_x$, $x\in I_{t}$ since $\alpha$ is non-decreasing and right continuous. Thus $\left[x\sbin \mathcal{G}\right] \,=\,\{I_{t}\colon t_x\leqslant t\leqslant t_{b}\}$ and by (\ref{WeightDensityFunction})
	\[
	T(x\sbin\mathcal{G})\,=\int_{t_x}^{t_{b}}f(t)\,d\alpha(t)
	\]
	
	Obviously, for $c\leqslant x\leqslant a$, $\left[x\sbin\mathcal{G}\right] =\mathcal{G}$ and so $T(x\sbin\mathcal{G}) =1$. For $x<c\ $or $x>b$, $\left[x\sbin\mathcal{G}\right] =\varnothing$ and $T(x\sbin\mathcal{G}) =0$. \medskip
	
	\begin{figure}[h]
		\centering
		\includegraphics{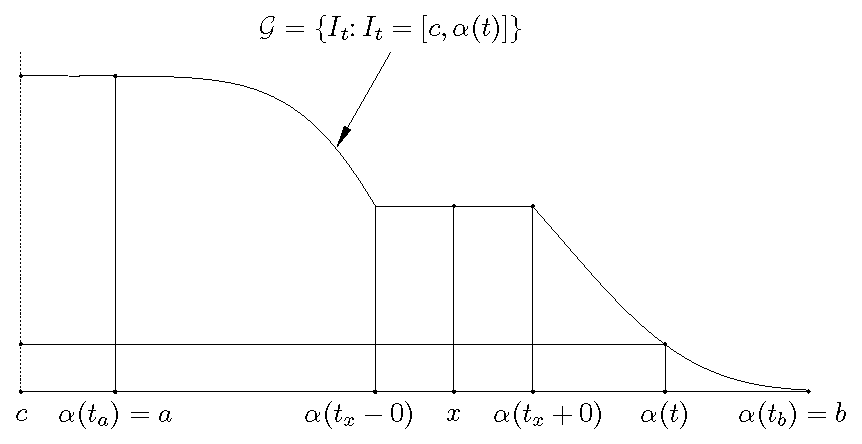}
		\caption{Diagram of the continuous belief distribution in theorem \ref{74}(i).} 
		\label{Fig4}
	\end{figure}	
	
	(ii) is a special case of (iii). \bigskip
	
	(iii) \ If $a\leqslant x\leqslant b$, let $t_x\in \left[t_{a},t_{b}\right]$ be $\alpha(t_x-0)\leqslant x\leqslant\alpha(t_x+0)$. So for any $t<t_x$, $x\in I_{t}$ since $\alpha$ is non-decreasing and right continuous (see  Figure \ref{Fig5}). Thus $\left[x\sbin\mathcal{G}\right] \,=\,\{I_{t}\colon t_{a}\leqslant t\leqslant t_x\}$ and by (\ref{WeightDensityFunction}
	\[
	T(x\sbin\mathcal{G})\,=\int_{t_{a}}^{t_x}f(t)\,d\alpha(t)
	\]
	
	If $c\leqslant x\leqslant d$, there is a $t_x\in \left[t_{a},t_{b}\right]$ that $\beta(t_x+0)\leqslant x\leqslant \beta(t_x-0)$ for $\beta(t_{b})=c$ and $\beta(t_{a})=d$.\footnote{Note that in Figure \ref{Fig5}, $\beta(t)$ is continuous at $t_x$.} So for any $t<t_x$, $x\in I_{t}$ since $\beta$ is non-increasing and left continuous. Thus $\left[x\sbin \mathcal{G} \right] \,=\,\{I_{t}\colon t_{a}\leqslant t\leqslant t_x\}$. Clearly, $-\beta$ is non-decreasing and right continuous. So by (\ref{WeightDensityFunction}) 
	\[
	T(x\sbin\mathcal{G})\,=\int_{t_{a}}^{t_x}f'(t)\,d(-\beta(t))\,=\int_{t_x}^{t_{a}}f'(t)\,d\beta(t)
	\]
	Obviously, for $b\leqslant x\leqslant c$, $\left[x\sbin\mathcal{G}\right] =\mathcal{G}$ and $T(x\sbin\mathcal{G}) =1$. For $x<a$ or $x>d$, $\left[x\sbin\mathcal{G}\right] =\varnothing$ and $T(x\sbin\mathcal{G}) =0$.\bigskip	
\end{proof}

\begin{figure}[h]
	\centering
	\includegraphics{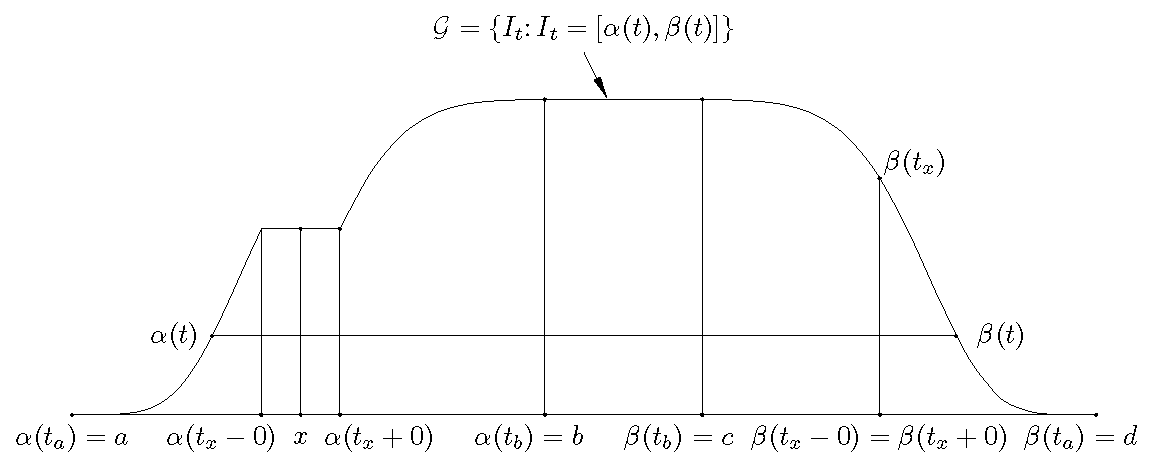}
	\caption{Diagram of the continuous belief distribution in theorem \ref{74}(iii).} 
	\label{Fig5}
\end{figure}

The belief distributions $T(x\sbin\mathcal{G})$ $(\alpha(t)\geqslant0\, \wedge\, \alpha(0)=0)$ are shown in Figure \ref{Fig6}.

\begin{remark}
	\ A belief distribution (discrete or continuous) is similar to a cumulative	probability distribution despite a notable difference between the two --- a cumulative probability distribution is always non-decreasing, whereas a belief distribution can be either increasing or decreasing (see Figure \ref{Fig6}).
\end{remark}

\begin{figure}[h]
	\centering
	\includegraphics{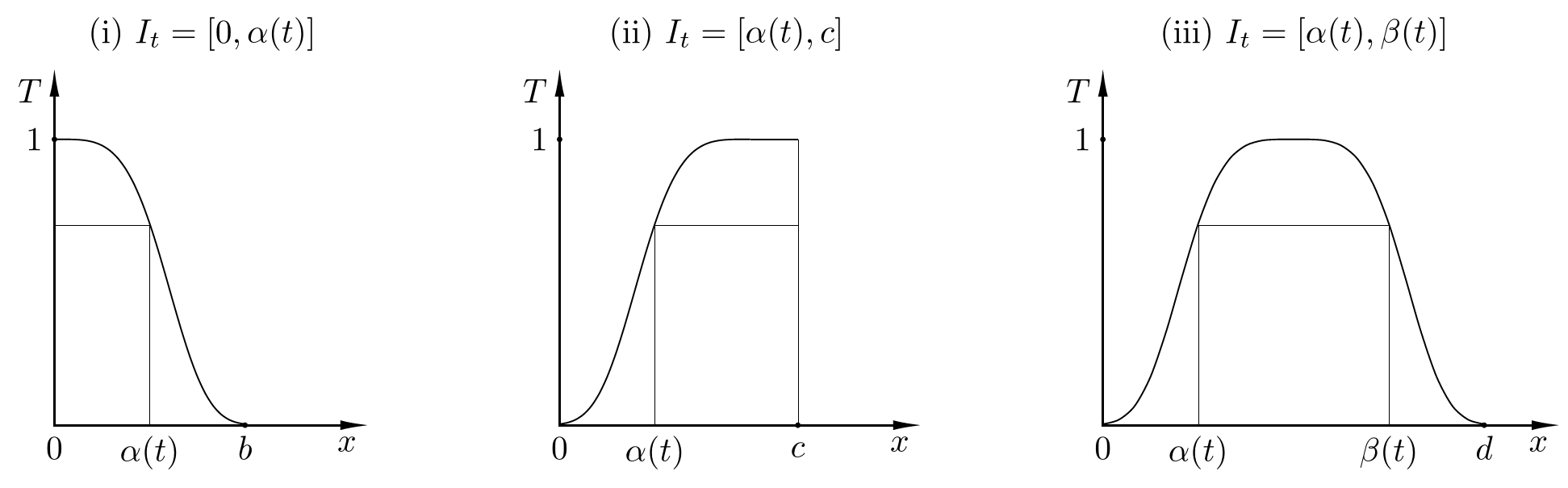}
	\caption{Diagrams of continuous belief distributions in theorem \ref{74}.} 
	\label{Fig6}
	\centering
\end{figure}

\begin{corollary} \label{58}
	\ Suppose everything is the same as in theorem \ref{74}. If $f(t)=f'(t)=M$ on $[t_a,t_b]$, then $M=1/(\alpha(t_b)-\alpha(t_a))$ and the non-constant parts of $T(x\sbin\mathcal{G})$ in theorem \ref{74} are:
\end{corollary}

\begin{enumerate}
	\item $T(x\sbin\mathcal{G})\,=\,M(\alpha(t_b)-\alpha(t_x))$.
	
	\item $T(x\sbin\mathcal{G})\,=\,M(\alpha(t_x)-\alpha(t_a))$.
	
	\item $T(x\sbin\mathcal{G})\,=\,M(\alpha(t_x)-\alpha(t_a))$ \textit{or} $M(\beta(t_a)-\beta(t_x))$.
\end{enumerate}

\begin{proof}
	\ (i) \ By definition \ref{DefWeightDensityFunction}
	\[
	\int_{t_{a}}^{t_{b}}M d\alpha(t)\,=\,1\quad\text{ and so }\quad M\,=\,1/(\alpha(t_b)-\alpha(t_a))
	\]
	Thus by theorem \ref{74}(i)
	\[
	T(x\sbin\mathcal{G}) \,=\int_{t_x}^{t_{b}}M d\alpha(t)\,=\,M(\alpha(t_b)-\alpha(t_x))
	\]
	(ii) and (iii) are proved similarly. 
\end{proof}

\begin{corollary}\label{59}
	\ Suppose $\left.\mathcal{G}\right\vert_{r}=\{I_{t}\colon t\in[a,\infty)\}, \:I_{t}=[c,\alpha(t)]\:(a,c\in\Bbb{R}\, \wedge \,c\leqslant a)$ and $\alpha(t)=t$. If $f(t)\,=\,2nb\,(t-a)^{2n-1} e^{-b(t-a)^{2n}}$ on $t\in[a,\infty)$ and $f(t)=0$ elsewhere $(n\in \Bbb{N}\,\wedge\, b\in\Bbb{R}^+)$, then $f(t)$ is a weight density function and $\,T(x\sbin\mathcal{G})\,=\, e^{-b(x-a)^{2n}}$.
\end{corollary}

\begin{proof} \ By definition \ref{DefWeightDensityFunction} 
	\[
	\int_{a}^{\infty} f(t)\,dt\,=\left.\left(-e^{-b(t-a)^{2n}}\right)\right\vert_{a}^{\infty}\,=\,1
	\]
	So $f(t)$ is a weight density function. For any $x\in[a,\infty)$, let $x=t_x$. So by theorem \ref{74}(i)
	\[
	T(x\sbin\mathcal{G}) \,=\int_{x}^{\infty} f(t)\,dt\,=\left.\left(-e^{-b(t-a)^{2n}}\right)\right\vert_{x}^{\infty} \,=\,e^{-b(x-a)^{2n}}
	\]	
	By definition \ref{DefWeightDensityFunction}, $f(t)=0$ on $[c,a]$. So for any $x\in[c,a]$,
	\[
	T(x\sbin\mathcal{G}) \,=\int_{x}^{\infty} f(t)\,dt\,=\int_{a}^{\infty} f(t)\,dt\,=\,1
	\] 
	For $x<c,\:T(x\sbin\mathcal{G})=0\,$ for $\left[x\sbin\mathcal{G}\right]=\varnothing$.\smallskip
\end{proof}

\begin{remark}
	\ $T(x\sbin\mathcal{G}) \,=\,e^{-b(x-a)^{2}}$ is known as the \textbf{normal belief distribution}.
\end{remark}

The following is a variant of theorem \ref{74} for a non-singleton set.

\begin{theorem}
	\label{75} \ Suppose everything is the same as in theorem \ref{74}, $\widetilde{\mathcal{S}}(\phi(A)) = \left(A,\left.\mathcal{G}\right\vert_{r},\subset\right)$, and $A$ is a non-singleton subset of $\Bbb{R}$.
\end{theorem}

\begin{enumerate}
	\item \textit{If $I_{t}\,=\,[c,\alpha(t)]$ and $\,c\leqslant a\leqslant\alpha(t)\leqslant b$, then}
	\begin{align*}
		T(\phi) & \:=\,\int_{t_{A}}^{t_{b}}f(t)\,d\alpha(t),\text{\ \ \ \ }\, a\leqslant\alpha(t_A-0)\leqslant \max\{A\}\leqslant\alpha(t_A+0)\leqslant b \,\wedge\, c\leqslant \min\{A\}
		\\ 
		& \:=\:1,\qquad\qquad\qquad\quad c\leqslant\min\{A\}<\max\{A\}\leqslant a
		\\
		& \:=\:0,\qquad\qquad\qquad\quad \min\{A\}<c\,\vee\,\max\{A\}>b
	\end{align*}
	
	\item \textit{If $I_{t}\,=\,[\alpha(t),c]$ and $\,a\leqslant\alpha(t)\leqslant b\leqslant c$, then}
	\begin{align*}
		T(\phi) & \:=\,\int_{t_{a}}^{t_A}f(t)\,d\alpha(t),\text{\ \ \ \ }a\leqslant\alpha(t_A-0)\leqslant \min\{A\}\leqslant\alpha(t_A+0)\leqslant b\,\wedge\,\max\{A\}\leqslant c
		\\
		& \:=\:1,\qquad\qquad\qquad\quad b\leqslant \min\{A\}<\max\{A\}\leqslant c
		\\
		& \:=\:0,\qquad\qquad\qquad\quad \min\{A\}<a\,\vee\,\max\{A\}>c
	\end{align*}
	
	\item \textit{If $I_{t}\,=\left[\alpha(t),\,\beta(t)\right]$ and $\,a\leqslant\alpha(t)\leqslant b\leqslant c\leqslant\beta(t)\leqslant d$, let $a\leqslant\alpha(t_A-0)\leqslant\min\{A\}\leqslant\alpha(t_A+0)\leqslant b$, $  c\leqslant\beta(t_A'+0)\leqslant\max\{A\}\leqslant\beta(t_A'-0)\leqslant d$, and $\,t_0=\min\{t_{A},t_{A'}\}$. Then}
	\begin{align*}
		T(\phi) & \:=\,\int_{t_{a}}^{t_0}f(t)\,d\alpha(t),\text{\ \ \ \ \ \ \ \ }\,\alpha(t_A-0)\leqslant \min\{A\}< \max\{A\}\leqslant\beta(t_A'-0)
		\\
		& \:=\:1,\qquad\qquad\qquad\qquad\ b\leqslant \min\{A\}<\max\{A\}\leqslant c
		\\
		& \:=\:0,\qquad\qquad\qquad\qquad\ \min\{A\}<a\,\vee\,\max\{A\}>d
	\end{align*}
\end{enumerate}

\begin{proof}
	\ (i) \ Let $t_A\in \left[t_{a},t_{b}\right]$ that $\alpha(t_A-0) \leqslant \max\{A\}\leqslant\alpha(t_A+0)$. Then for any $t\geqslant t_A,\, A\subset I_{t}$ for $\alpha(t)\geqslant\alpha(t_A)$ and $c\leqslant\min\{A\}$. So $[\phi] \,=\, \{I_{t}\colon t_A\leqslant t\leqslant t_{b}\}$ and by (\ref{WeightDensityFunction})
	\[
	T(\phi)\,=\int_{t_A}^{t_{b}}f(t)\,d\alpha(t)
	\]
	
	Obviously, if $A\subset[c,a]$, $[\phi] =\mathcal{G}$ and so $T(\phi) =1$. If $\min\{A\}<a$ or $\max\{A\}>d$, $[\phi] =\varnothing$ and $T(\phi) =0$. \medskip
	
	(ii) \ Let $t_A\in \left[t_{a},t_{b}\right]$ that $\alpha(t_A-0) \leqslant \min\{A\}\leqslant\alpha(t_A+0)$. Then for any $t\leqslant t_A,\, A\subset I_{t}$ since $\alpha(t)\leqslant\alpha(t_A)$ and $\max\{A\}\leqslant c$. So $[\phi] \,=\, \{I_{t}\colon t_a\leqslant t\leqslant t_{A}\}$ and it follows by (\ref{WeightDensityFunction}). The other two cases are obvious. \medskip
	
	(iii) \ Clearly, $\alpha(t_0)\leqslant \alpha(t_A)$ and $\beta(t_0)\geqslant\beta(t_{A'})$ for $t_0\leqslant t_{A}$ and $t_0\leqslant t_{A'}$. So for any $t\leqslant t_0,\,A\subset I_{t}$. Thus $[\phi] \,=\, \{I_{t}\colon t_{a}\leqslant t\leqslant t_0\}$ and it follows by (\ref{WeightDensityFunction}). The other two cases are obvious. \smallskip	
\end{proof}

\begin{problem}
	\label{76} \ Suppose $\phi_1(x)\rightleftharpoons $\textquotedblleft $\,x$ is young\textquotedblright\ and $\,\widetilde{\mathcal{S}}(\phi_1(x)) = \left(\{x\},\left.\mathcal{G}_{1}\right\vert_{r_1},\subset \right)$ ($x$ is an age for the human) with 
	\[
	\left.\mathcal{G}_1\right\vert_{r_1}=\,\{I_t\colon I_t=[0,\alpha(t)]\,\wedge \,\alpha(t)=20t+20\,\wedge \,0\leqslant t\leqslant 1\,\wedge \,f_1(t)=0.05\,\}
	\]
	and $\phi_2(x)\rightleftharpoons$\textquotedblleft $\,x$ is middle-aged\textquotedblright\ and $\,\widetilde{\mathcal{S}}(\phi_{2}(x)) = \left(\{x\},\left.\mathcal{G}_2\right\vert_{r_2},\subset \right) $ with 
	\[
	\left.\mathcal{G}_2\right\vert_{r_2}=\,\{J_{t}\colon J_{t}=[\alpha(t)\text{, }\beta(t)]\,\wedge \,\alpha(t) = 20t+20\,\wedge \,\beta(t) =70-20t\,\wedge \,0\leqslant t\leqslant 1\,\wedge \,f_2(t)=0.05\,\}
	\]
	and $\phi_3(x)\rightleftharpoons$\textquotedblleft $\,x$ is old\textquotedblright\ and $\,\widetilde{\mathcal{S}}(\phi_3(x)) = \left(\{x\},\left.\mathcal{G}_3\right\vert_{r_3},\subset \right)$ with
	\[
	\left.\mathcal{G}_3\right\vert_{r_3}=\,\{K_{t}\colon K_{t}=[\gamma(t)\text{,}\,150]\,\wedge \,\gamma(t) = 20t+50\,\wedge \,0\leqslant t\leqslant 1\,\wedge \,f_3(t)=0.05\,\}\footnote{We assume that human age is no more than $150$ years old.}
	\]
	where $f_i(t)$ $\left(1\leqslant i\leqslant 3\right)$ are weight density functions.
\end{problem}

\begin{enumerate}
	\item \textit{Find} $T(\phi_{i}(x))$ \textit{and prove} $\sum\limits_{1\leqslant i\leqslant3} T(\phi_i(x))=1$ \textit{for any} $x\in [0,150]$.
	
	\item \textit{Calculate} $T(\phi_i(x))$ \textit{for $x\,=\,10,\,20,\,32,\,57,\,75,\,90\:(\text{years old})$}. 
\end{enumerate}

\begin{proof}
	[Solution]\ (i) \ For $0\leqslant x\leqslant 20$, let $x=\alpha(t_x)$. Then $t_x\leqslant0$. So by theorem \ref{74}(i)
	\[
	T(\phi_{1}(x))\,=\int_{t_x}^{1}f_1(t)\,d\alpha(t)\,=\int_0^{1}f_1(t)\,d\alpha(t)\,=\,1
	\]
	\ Obviously, $T(\phi_{2}(x))\,=\,T(\phi_{3}(x))\,=\,0$.\medskip
	
	For $20\leqslant x\leqslant40,\:T(\phi_{3}(x))=0$. Suppose $x=\alpha(t_x)$. By corollary \ref{58}(i) and (ii)
	\[
	T(\phi_{1}(x))\,=\,0.05\,(\alpha(1)-\alpha(t_x))\,=\,1-t_x\quad\text{and}\quad T(\phi_{2}(x))\,=\,0.05\, (\alpha(t_x)-\alpha(0))\,=\,t_x
	\]
	
	For $40\leqslant x\leqslant 50,\:T(\phi_{1}(x))=T(\phi_{3}(x))=0$, and $T(\phi_{2}(x))=1$. \medskip
	
	For $50\leqslant x\leqslant 70,\:T(\phi_{1}(x))=0$. Suppose $x=\gamma(t_x)=\beta(t_x')$. Then $t_x'=1-t_x$. So by corollary \ref{58}(ii) and (iii)
	\[
	T(\phi_{2}(x))\,=\,0.05\,(\beta(0)-\beta(t_x'))\,=\,t_x'\,=\,1-t_x\quad\text{and}\quad T(\phi_{3}(x))\,=\,0.05\,(\gamma(t_x)-\gamma(0)) \,=\,t_x
	\]
	
	For $x\geqslant 70,\,T(\phi_{1}(x))=T(\phi_{2}(x))=0$, and $T(\phi_{3}(x))=1$. \bigskip
	
	From the above results, we reach the conclusion that for any $x\in [0,150]$, $\sum\limits_{1\leqslant i\leqslant 3} T(\phi_i(x))=1$. (See Figure \ref{Fig7} for more detail.)\bigskip
	
	(ii) \ From (i), we have \medskip
	
	$T(\phi_{1}(10))\,=\,T(\phi_{1}(20))=1$, \ and $\,T(\phi_{1}(57))\,=\,T(\phi_{1}(75))\,=\,T(\phi_{1}(90))=0$. From $\alpha(t_x)=32,\:t_x=0.6$. So $T(\phi_{1}(32))=1-t_x=0.4$.\medskip
	
	$T(\phi_{2}(10))\,=\,T(\phi_{2}(20))\,=\,T(\phi_{2}(75))\,=\,T(\phi_{2}(90))=0$. Also $T(\phi_{2}(32))=t_x=0.6$. From $\beta(t_x')=57,\:t_x'=0.65$. So $T(\phi_{2}(57))=0.65$. \medskip
	
	$T(\phi_{3}(10))\,=\,T(\phi_{3}(20))\,=\,T(\phi_{3}(32))=0$, \ and $\,T(\phi_{3}(75))\,=\,T(\phi_{3}(90))=1$. From $\gamma(t_x)=57,\:t_x=0.35$. So $T(\phi_{3}(57))=0.35$.\medskip
\end{proof}

\begin{figure}[h]
	\centering
	\includegraphics{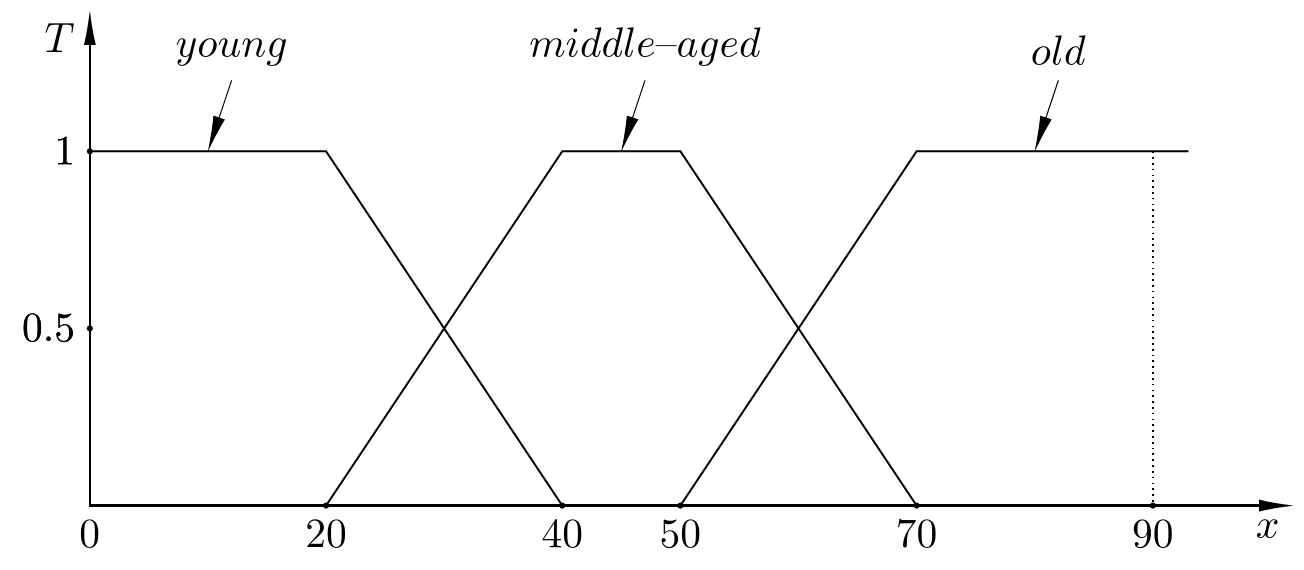}
	\caption{Diagram of belief distributions of three adjectives for \textit{age}.}
	\label{Fig7}
	\centering
\end{figure}

\begin{problem}
	\label{8} \ Suppose $\varphi_{1}(y)\rightleftharpoons\, $\textquotedblleft$\,y$ is short\textquotedblright \ ($y$ is a height for the human) and $\widetilde{\mathcal{S}}(\varphi_{1}(y))=\left(\{y\},\left.\mathcal{C}_{1}\right\vert_{r_{1}},\subset\right)$ with
	\[
	\left.\mathcal{C}_{1}\right\vert_{r_{1}}\,=\,\{I_{t}\colon I_{t}\,=\, [0,\alpha(t)] \,\wedge\, \alpha(t)=t/10+1.6 \,\wedge\,0 \leqslant t\leqslant1\,\wedge\,f_{1}(t)=10\,\}
	\]
	and $\varphi_{2}(y)\rightleftharpoons$\textquotedblleft$\,y$ is \textit{medium}\textquotedblright\ and $\widetilde{\mathcal{S}}(\varphi_{2}(y)) =\left(\{y\},\left.\mathcal{C}_{2}\right\vert_{r_{2}},\subset\right)$ with
	\[
	\left.\mathcal{C}_{2}\right\vert_{r_{2}}\,=\,\{J_{t}\colon J_{t}=[\alpha(t),\,\beta(t)] \,\wedge\, \alpha(t)=t/10+1.6 \,\wedge\,\beta(t)=1.85-t/10\,\wedge\,0\leqslant t\leqslant1\,\wedge\,f_{2}(t)=10\,\}
	\]
	and $\varphi_{3}(y)\rightleftharpoons$\textquotedblleft$\,y$ is tall\textquotedblright\ and $\widetilde{\mathcal{S}}(\varphi_{3}(y))=\left(\{y\}, \left. \mathcal{C}_{3}\right\vert_{r_{3}},\subset\right)$ with
	\[
	\left.\mathcal{C}_{3}\right\vert_{r_{3}}\,=\,\{K_{t}\colon K_{t} \,=\, [\gamma(t)\text{,}\,3] \,\wedge\, \gamma(t)=t/10+1.75 \,\wedge\,0\leqslant t\leqslant1\,\wedge\,f_{3}(t)=10\,\}
	\]
	where $f_{i}(t) \left(1\leqslant i\leqslant3\right)$ are weight density functions.\footnote{Here the belief distributions of \textit{height} are for men only.} 
\end{problem}

\begin{enumerate}
	\item \textit{Find} $T(\varphi_{i}(y))$ \textit{and prove that} $\sum\limits_{1\leqslant i\leqslant3} T(\varphi_i(y))=1$ \textit{for any} $y\in [0,3]$.
	
	\item \textit{Calculate} $T(\varphi_i(y))$ \textit{for} $y\,=\,1.6,\,1.71,\,1.82,\,1.9\left(m\right)$.
	
	\item \textit{Calculate} $T(\varphi_i(A))$ \textit{for} $A\,=\,\{1.71,1.77, 1.82\}$. 
\end{enumerate}

\begin{proof}
	[Solution]\ (i) \ It is similar to the solution of problem \ref{76}. \medskip
	
	For $0\leqslant y\leqslant 1.6,\: T(\varphi_{2}(y))=T(\varphi_{3}(y))=0$, and $T(\varphi_{1}(y))=1$. \medskip
	
	For $1.6\leqslant y\leqslant1.7$,\, $T(\varphi_{3}(y))=0$. Suppose $y=\alpha(t_y)$. By corollary \ref{58}, $T(\varphi_{1}(y))=1-t_y$ and $T(\varphi_{2}(y))=t_y$. \medskip
	
	For $1.7\leqslant y\leqslant 1.75,\,T(\varphi_{1}(y))=T(\varphi_{3}(y))=0$, and $T(\varphi_{2}(y))=1$. \medskip
	
	For $1.75\leqslant y\leqslant 1.85$, $T(\varphi_{1}(y))=0$. Suppose $y=\gamma(t_y)=\beta(t_y')$. Then $t_y'=1-t_y$. By corollary \ref{58}, $T(\varphi_{2}(y))=t_y'=1-t_y$ and $T(\varphi_{3}(y))=t_y$. \medskip
	
	For $y\geqslant 1.85,\,T(\varphi_{1}(y))=T(\varphi_{2}(y))=0$, and $T(\varphi_{3}(y))=1$. \medskip
	
	From above results, we have for any $y\in [0,3]$, $\sum\limits_{1\leqslant i\leqslant 3} T(\varphi_i(y))=1$. (See Figure \ref{Fig8} for more detail.) \smallskip
	
	(ii) \ From (i), we have \medskip
	
	$T(\varphi_{1}(1.6))=1$, \ and $\,T(\varphi_{1}(1.71))\,=\,T(\varphi_{1}(1.82))\,=\,T(\varphi_{1}(1.9))\,=\,0$. \medskip
	
	$T(\varphi_{2}(1.6))=T(\varphi_{2}(1.9))=0$, \ and $\,T(\varphi_{2}(1.71))=1$. From $1.85-t_y'/10=1.82, \,t_y'=0.3$, and so $T(\varphi_{2}(1.82))=0.3$. \medskip
	
	$T(\varphi_{3}(1.6))=T(\varphi_{3}(1.71))=0$, \ and $\,T(\varphi_{3}(1.9))=1$. From $t_y/10+1.75=1.82,\,t_y=0.7$, and so $T(\varphi_{3}(1.82))=0.7$. \medskip
	
	(iii) \ Obviously, $\min\{A\}=1.71$ and $\max\{A\}=1.82$. So by theorem \ref{75} and (ii), $T(\varphi_{1}(A))= T(\varphi_{3}(A))=0$. Since $t_0=\min\{1,0.3\}=0.3$, $\,T(\varphi_{2}(A))=0.3$. \smallskip
\end{proof}

\begin{figure}[h]
	\centering
	\includegraphics{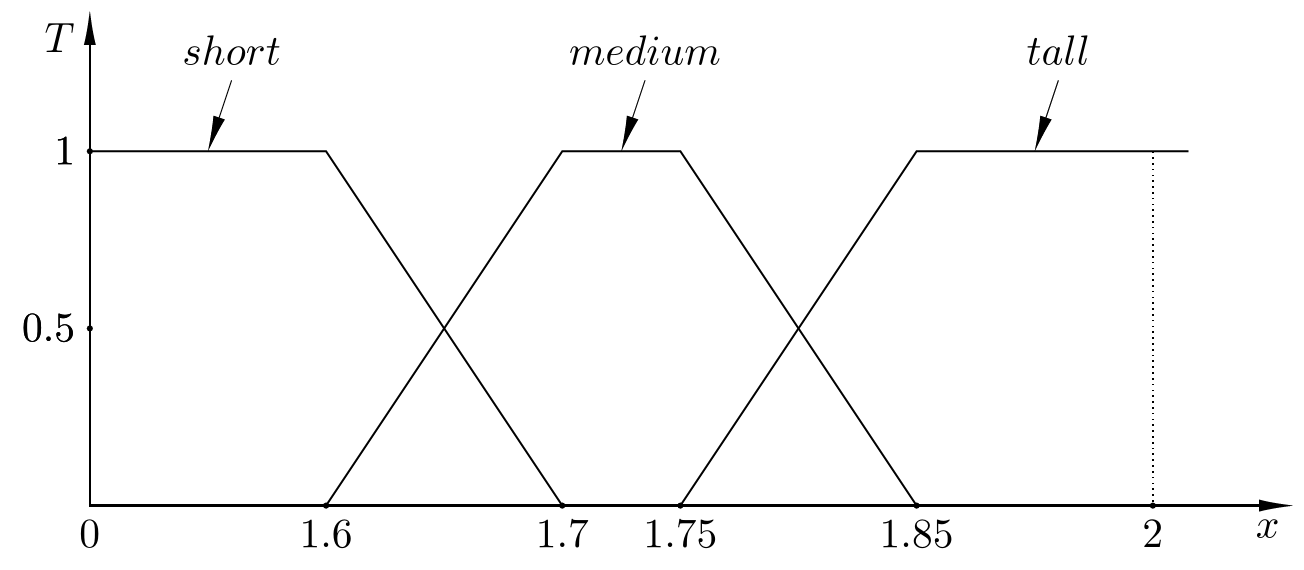}
	\caption{Diagram of belief distributions of three adjectives for \textit{height}.}
	\label{Fig8}
	\centering
\end{figure}

\subsection{Joint Bundles and Belief Distributions}\label{JointBundle}

Joint bundles are important because the conjunction, disjunction and implication of predicates in the multi-valued logic $\mathfrak{M}$ are based upon them. There are three types of joint bundles and belief distributions, discrete--discrete, continuous--continuous, and continuous--discrete. 

\begin{definition}
	\label{DefJointDiscreteDistribution} \ Suppose $\left.\mathcal{G}_{1}\right\vert_{r_{1}} =\,\{X\colon r_{1}(X)\}$ and $\left.\mathcal{G}_{2}\right\vert_{r_{2}} =\,\{Y\colon r_{2}(Y)\}$ are discrete set bundles, $\widetilde{\mathcal{S}}(\phi_{1}) =\left(\{x\}, \left.\mathcal{G}_{1}\right\vert_{r_{1}},\bigodot \right)$ and $\,\widetilde{\mathcal{S}}(\phi_{2}) =\left(\{y\}, \left.\mathcal{G}_{2}\right\vert_{r_{2}},\bigodot\right)$, $T(x,y)\,=\,T(x\sbin \mathcal{G}_{1}\wedge y\sbin \mathcal{G}_{2}).$\footnote{In the rest discussion of this section, we assume $\,\widetilde{\mathcal{S}}(\phi_{1}) =(\{x\}, \left.\mathcal{G}_{1}\right\vert_{r_{1}},\bigodot), \,\widetilde{\mathcal{S}}(\phi_{2}) =(\{y\}, \left.\mathcal{G}_{2}\right\vert_{r_{2}},\bigodot),\,T(x,y)=T(x\sbin \mathcal{G}_{1}\wedge y\sbin \mathcal{G}_{2})$.} Then the \textbf{joint discrete distribution} of $\phi_{1}$ and $\phi_{2}$ is defined as:
	\begin{equation}
		T(\phi_{1}\barwedge\phi_{2})\,=\,T(x,y)\:=\sum_{X\in\left[\phi_{1}\right]}\sum_{Y\in\left[\phi_{2}\right]} r(\left(X,Y\right)) \label{JointDiscreteDistribution}
	\end{equation}
	Where $r$ is a weight measure on $\sigma(\mathcal{G}_{1}\times\mathcal{G}_{2})$ and  $\bigodot$ denotes $\subset$ or $\not \subset$.
\end{definition}

\begin{definition}
	\label{DefJointContinuousDistribution} \ Suppose $\left.\mathcal{G}_{1}\right\vert_{r_{1}} = \{I_{t}\colon t\in\left[t_{a},t_{b}\right]\subset\Bbb{R}\}$ and $\left.\mathcal{G}_{2} \right\vert_{r_{2}} \,=\, \{I_{s}\colon s\in\left[s_{c},s_{d}\right] \subset \Bbb{R}\}$ where $I_{t}=[A,\alpha(t)]$ or $[\alpha(t),A]$ or $[\alpha(t),\,\gamma(t)]$ and  $I_{s}=[B,\beta(s)]$ or $[\beta(s),B]$ or $[\beta(s),\,\gamma(s)]$, $\alpha$ and $\beta$ are non-decreasing right continuous functions and $\gamma$ a non-increasing left continuous function in $\left[a,b\right]^{\Bbb{R}}$ $\left(\left[a,b\right] \subset \Bbb{R} \wedge A,B\in\Bbb{R}\right)$, $f(t,s)$ is a joint weight density function on $\left.\left(\mathcal{G}_{1}\times\mathcal{G}_{2}\right)\right\vert_{r}$ that is non-negative on $[t_a,t_b]\times [s_{c},s_{d}]$ and zero elsewhere. Then the \textbf{joint continuous distribution} of $\phi_{1}$ and $\phi_{2}$ is defined as:
	\begin{equation}
		T(\phi_{1}\barwedge\phi_{2})\,=\,T(x,y)\,=\,r\left(\{(I_{t},I_{s})\colon(t,s) \in\left[t_{1},t_{2}\right]\times\left[s_{1},s_{2} \right]\}\right) \,= \int_{t_{1}}^{t_{2}}\int_{s_{1}}^{s_{2}}f(t,s)\,d\beta(s)\,d\alpha(t) \label{JointContinuousDistribution}
	\end{equation}
	Where $\left[t_{1},t_{2}\right] \times\left[s_{1},s_{2}\right]\subset\Bbb{R}^2$. In particular, $r(\{(I_{t},I_{s})\colon (t,s) \in\left[t_{a},t_{b}\right] \times \left[s_{c},s_{d}\right]\}) \,=\,1$.
\end{definition}

\begin{definition}
	\label{DefJointMixDistribution} \ Suppose $\left.\mathcal{G}_{1}\right\vert_{r_{1}}$ is given in definition \ref{DefJointContinuousDistribution} and $\left.\mathcal{G}_{2}\right\vert_{r_{2}}$ in definition \ref{DefJointDiscreteDistribution}, $f(t,Y)$ is a weight density function of $\mathcal{G}_{1}$ for each $Y\in\mathcal{G}_{2}$. Then the \textbf{joint continuous--discrete (mixed) distribution} of $\phi_{1}$ and $\phi_{2}$ is defined as:
	\begin{equation}
		T(\phi_{1}\barwedge\phi_{2})\,=\,T(x,y)\:=\underset{Y\in\left[\phi_{2}\right]}{\sum}\int_{t_{1}}^{t_{2}}f(t,Y)\, d\alpha(t) \label{JointMixDistribution}
	\end{equation}
\end{definition}

\begin{remark}
	In the above three definitions, $T(\phi_{1}\barwedge\phi_{2})$ can also be defined upon theorem \ref{73} and \ref{75} for non-singleton subject sets, i.e. $\widetilde{\mathcal{S}}(\phi_{1}) =\left(A_1, \left.\mathcal{G}_{1}\right \vert_{r_{1}},\bigodot \right)$ and $\,\widetilde{\mathcal{S}}(\phi_{2}) =\left(A_2, \left.\mathcal{G}_{2}\right \vert_{r_{2}},\bigodot\right)$ where $A_1$ and $A_2$ are non-singleton subsets of $\Bbb{R}$.
\end{remark}

A marginal weight density function that is similar to a marginal probability density function can be defined as well. 

\begin{definition}\label{DefMarginalWeightDensityFunction}
	\ Suppose everything is the same as in definition \ref{DefJointContinuousDistribution}. Then the \textbf{marginal weight density functions} for $\phi_1$ and $\phi_2$ are defined as:
	\[
	f_T(t)\,=\int_{s_{c}}^{s_{d}}f(t,s)\,d\beta(s)\quad\text{and}\quad f_S(s)\,=\int_{t_{a}}^{t_{b}}f(t,s)\,d\alpha(t)
	\]
\end{definition}

For the joint continuous distribution, we can define irrelevancy as well.

\begin{definition}
	\label{DefIrrelevantContinuousBundles} \ Suppose everything is the same as in definition \ref{DefJointContinuousDistribution} and \ref{DefMarginalWeightDensityFunction}. If $f(t,s)\,=\,f_T(t)\,f_S(s)$, then $\left.\mathcal{G}_{1} \right\vert_{r_{1}}$ and $\left.\mathcal{G}_{2}\right\vert_{r_{2}}$ are \textbf{irrelevant}. Otherwise, they are \textbf{relevant}.
\end{definition}

\begin{corollary}\label{45}
	\ If $\left.\mathcal{G}_{1}\right\vert_{r_{1}}$ and $\left.\mathcal{G}_{2}\right\vert_{r_{2}}$ are irrelevant, then $\phi_{1}$ and $\phi_{2}$ are irrelevant.
\end{corollary}

\begin{proof}
	\ By (\ref{JointContinuousDistribution}) and definition \ref{DefIrrelevantContinuousBundles}
	\begin{spreadlines}{1.5ex}
		\begin{align*}
			T(\phi_{1}\barwedge\phi_{2})  & \,=\, \int_{t_{1}}^{t_{2}}\int_{s_{1}}^{s_{2}}f(t,s)\,d\beta(s)\,d\alpha(t)
			\\
			& \,=\, \int_{t_{1}}^{t_{2}}f_T(t)\,d\alpha(t)\int_{s_{1}}^{s_{2}}f_S(s)\,d\beta(s)
			\\
			& \,=\, T(\phi_{1})\,T(\phi_{2})
		\end{align*}	
	\end{spreadlines}
	
	So it follows by lemma \ref{18}. \smallskip
\end{proof}

\begin{problem}
	\ Suppose $\phi_3(x)\rightleftharpoons$\textquotedblleft $\,x$ is old\textquotedblright\ and $\,\widetilde{\mathcal{S}}(\phi_3(x)) = \left(\{x\},\left.\mathcal{G}_3\right\vert_{r_3},\subset \right)$ are given in problem \ref{76}, $\varphi_{3}(y)\rightleftharpoons\, $\textquotedblleft$\,y$ is tall\textquotedblright\ and $\widetilde{\mathcal{S}}(\varphi_{3}(y))=\left(\{y\},\left.\mathcal{C}_{3}\right\vert_{r_{3}},\subset\right)$ given in problem \ref{8}. Find $\,T(\phi_3(x)\barwedge\varphi_3(y))$. Verify that $\phi_3$ and $\varphi_3$ are irrelevant.
\end{problem}

\begin{proof} 
	\ Let a joint density function of $\phi_3$ and $\varphi_3$ be
	\begin{align*}
		f(t,s) & \,=\,0.5,  \qquad t\in [0,1]\,\wedge\,s\in[0,1]
		\\
		& \,=\, 0, \,\qquad\text{\ \ otherwise}
	\end{align*}
	
	\ For $x\geqslant50\,$ and $\,y\geqslant 1.75$, let $x=\gamma(t_x)\,$ and $\,y=\gamma'(s_y)$. So by (\ref{JointContinuousDistribution})
	\[
	T(\phi_3(x)\barwedge\varphi_3(y))\,=\int_{0}^{t_{x}}\int_{0}^{s_{y}}f(t,s)\,d\gamma^{\prime}(s)\,d\gamma(t)\,=\,T(\phi_3(x))\, T(\varphi_3(y))
	\]
	
	\ So by lemma \ref{18}, $\phi_3$ and $\varphi_3$ are irrelevant. \bigskip
\end{proof}

The above result is consistent with the fact that adult age and height are irrelevant because human height generally stops growing after around 20 years old.

\begin{problem}
	\ Suppose $\,\theta(z)\rightleftharpoons\,$\textquotedblleft $\,z$ is heavy\textquotedblright\ ($z$ is a weight of the human) and $\,\widetilde{\mathcal{S}}(\theta(z)) = (\{z\},\left.\mathcal{W}\right\vert_{r},\subset)$ where
	\[
	\left.\mathcal{W}\right\vert_{r}\,=\,\{I_{s}\colon I_{s}=[\alpha(s),\,1000]\,\wedge \,\alpha(s) = 40s+80\,\wedge \,0\leqslant s\leqslant 1\,\wedge \,f(s)=0.025\,\}\footnote{We assume that human weight is no more than $1000kg$.}
	\]
\end{problem}

\begin{enumerate}
	\item \textit{Find $\,T(\theta(z))$ and $\,T(\varphi_3(y)\barwedge \theta(z))$}.
	
	\item \textit{Show that $\varphi_3$ and $\theta$ are relevant.}
\end{enumerate}

\begin{proof}
	\ (i) \ Obviously, for $0\leqslant z\leqslant 80,\:T(\theta(z)=0$. For $z\geqslant 120,\:T(\theta(z))=1$. \medskip
	
	For $80\leqslant z\leqslant 120$, suppose $z=\alpha(s_z)$. Then by corollary \ref{58}(ii)
	\[
	T(\theta(z))\,=\,0.025 \left(\alpha(s_z)-\alpha(0)\right) \,=\,s_z
	\]
	
	(ii) \ Let a joint density function of $\varphi_3$ and $\theta$ be
	\begin{align*}
		f(t,s) & \,=\,0.25,  \qquad0.5t\leqslant s\leqslant 1+0.5t\,\wedge\,t\in [0,1]
		\\
		& \,=\, 0, \,\quad\qquad\text{otherwise}
	\end{align*}
	
	For $\,y\geqslant 1.75$ and $z\geqslant80$, let $y=\gamma(t_y)$ and $z=\alpha(s_z)$. So by (\ref{JointContinuousDistribution})
	\begin{spreadlines}{1.5ex}
		\begin{align*}
			T(\varphi_3(y)\barwedge\theta(z)) & \,=\int_{0}^{t_{y}}\int_{0.5t}^{s_z}f(t,s)\,d\alpha(s)\,d\gamma(t)
			\\
			& \,=\, 10 \int_{0}^{t_{y}}(s_z-0.5t)\,d\gamma(t)
			\\
			& \,=\, t_y\left(4s_z-t_y\right)/4
		\end{align*}
	\end{spreadlines}
	
	From problem \ref{8}, for $1.75\leqslant y\leqslant 1.85$, $T(\varphi_{3}(y))=t_y$. If $t_y\neq0$, then
	\[
	T(\varphi_{3}(y)\rightharpoonup\theta(z))\,=\,\dfrac{T(\varphi_{3}(y)\barwedge\theta(z))}{T(\varphi_{3}(y)}\,=\, \dfrac{4s_z-t_y}{4}\,\neq\,T(\theta(z))
	\]
	\ So by definition \ref{DefIrrelevantPredicates}, $\varphi_{3}$ and $\theta$ are relevant. \medskip
	
	For example, we have
	\[
	T(\varphi_{3}(1.82)\rightharpoonup\theta(100))=0.5-0.7/4=0.33\quad\text{and}\quad T(\varphi_{3}(1.9)\rightharpoonup\theta(100))=0.5-1/4=0.25
	\]
	\ In general, fix $z$ and if $y_1<y_2$, then $t_{y_1}<t_{y_2}$ and
	\[
	T(\varphi_{3}(y_1)\rightharpoonup\theta(z))\,=\,(4s_z-t_{y_1})/4\,>\,(4s_z-t_{y_2})/4 \,=\, T(\varphi_{3}(y_2)\rightharpoonup\theta(z))
	\]
	This is consistent with the fact that human height and weight are relevant because a tall man tends to be heavier than a short one. Consequently, a man with a lower height ($1.82m$) has a higher belief of \textit{heavy} for the same weight ($100kg$) than the one with a higher height ($1.9m$).
\end{proof}

\subsection{Contraction and Partition}
In this section, we will discuss additional operations for set bundles. 

\begin{definition}
	\label{DefContraction} \ Suppose $f\colon\mathcal{G}\to\mathcal{C}$ and $f(\mathcal{G})\subset\mathcal{C}$ is a set bundle. If for any $Y\in f(\mathcal{G})$, there is a $X\in\mathcal{G}$ that $Y\subset X$, then $f$ is known as a \textbf{contraction mapping} from $\mathcal{G}$ to $\mathcal{C}$, and $f(\mathcal{G})$ is a \textbf{contraction bundle} of $\mathcal{G}$.
\end{definition}

\begin{figure}[h]
	\center
	\includegraphics[scale=1]{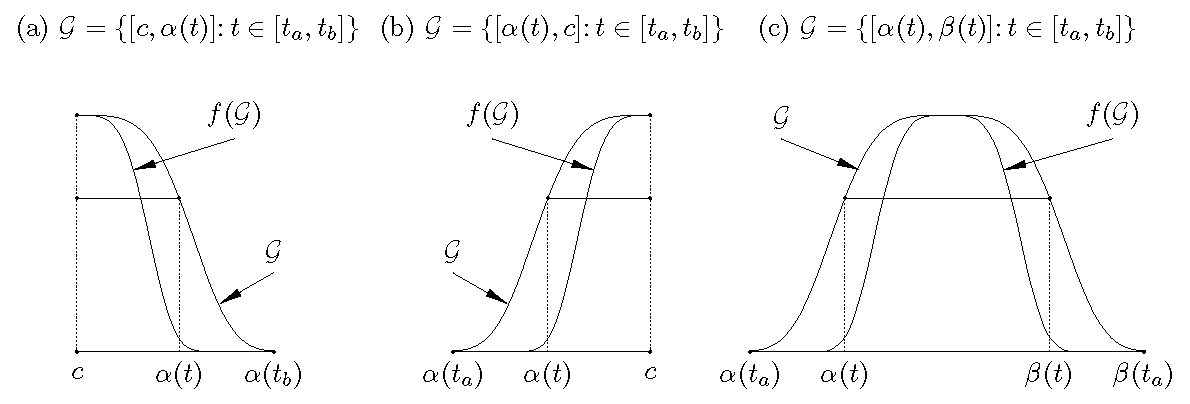}
	\caption{Diagrams of contraction of set bundles.}
	\label{Fig9}
	\center
\end{figure}

\begin{remark}
	\ In definition \ref{DefContraction}, note that $f(\mathcal{G}) \subset \mathcal{G}$ may not be true.
\end{remark}

\begin{lemma}
	\label{54} \ Suppose $\left.\mathcal{G}\right\vert_{r}\,=\,\{I_{t}\colon t\in\left[t_{a},t_{b}\right] \subset \Bbb{R}\}$ and $\,I_t=[c,\alpha(t)]$ or $[\alpha(t),c]$ or $[\alpha(t),\beta(t)]$ where $\,\alpha(t)$ is a non-decreasing right continuous function and $\beta(t)$ a non-increasing left continuous function in $[a,b]^{\Bbb{R}}$, $a,b,c\in\Bbb{R}$ and $f$ is a mapping on intervals in $\Bbb{R}$. If for any $t\in\left[t_{a},t_{b}\right]$, $f(I_t)\subset I_t$, then $f$ is a contraction mapping of $\mathcal{G}$.
\end{lemma}

\begin{proof}
	\ For any $Y\in f(\mathcal{G})$, $Y= f(I_t)\subset I_t\in\mathcal{G}$. So it follows by definition \ref{DefContraction}.
	(A diagram of $f(\mathcal{G})$ is shown in Figure \ref{Fig9}.) \bigskip
\end{proof}

\begin{remark}
\ Contraction of set bundles is often used in modeling adverbs such as \textit{very} (definition \ref{DefVery}).\end{remark}

\begin{definition}
	\label{DefPartition} \ Suppose $\left.\mathcal{G}_i\right\vert_{r_i} (1\leqslant i \leqslant n)$ are $n$ set bundles and A is a set. If for any $x\in A$, $\,\sum\limits_{1\leqslant i \leqslant n}T(x\sbin\mathcal{G}_i)\,=\,1$, then $\bigcup\limits_{1\leqslant i \leqslant n}\mathcal{G}_i$ is known as a \textbf{partition} of $A$.
\end{definition}

\begin{figure}[h]
	\centering
	\includegraphics{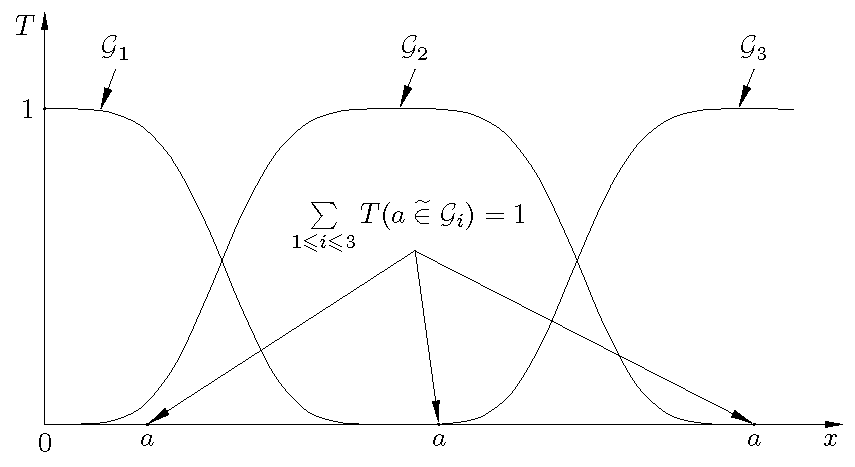}
	\caption{Diagram of partition of set bundles.}
	\label{Fig10}
	\centering
\end{figure}

\begin{remark}
	\ In definition \ref{DefPartition}, $\bigcup\limits_{1\leqslant i \leqslant n}\mathcal{G}_i$ is generally not a set bundle because $r(\mathcal{G}_i)=1$ for each $i$. If $\bigcup\limits_{1\leqslant i \leqslant n}\mathcal{G}_i$ is specified as a set bundle, its weight measure must be renormalized to $1$.
\end{remark}

\begin{example}
	\ Suppose $\mathcal{G}_i \ (1 \leqslant i \leqslant 3)$ are set bundles on $\Bbb{R}$. If for any $a\in \Bbb{R}$, $\sum\limits_{1\leqslant i \leqslant 3}T(a\sbin\mathcal{G}_i)=1$, then $\mathcal{G}_i$ is a partition on $\Bbb{R}$. A diagram is shown in Figure \ref{Fig10}.\medskip
\end{example}

Partition of set bundles is important in describing nouns by adjectives. For example, \textit{age} can be described by \textit{young}, \textit{middle-aged} and \textit{old}. See problem \ref{76} and section \ref{SectionAdjectives} for more detail.

\section{Set Bundles in Linguistics}\label{SectionLinguistics}

In this section, we will investigate further the modeling of adjectives and adverbs in linguistics based on set bundles and degree of truth.

\subsection{Adjectives} \label{SectionAdjectives}
An adjective is a modifier of a noun. Due to the diversity of adjectives, we will only discuss the modeling of adjectives based on set bundles for a few (important) examples in this paper.

\begin{definition}
	Suppose $A_k\:(1\leqslant k \leqslant n)$ are $n$ adjectives, each of which is represented by the set bundle $\mathcal{G}_i$, and $\mathscr{R}_N$ is the range of a noun $N.$ If $\bigcup\limits_{1\leqslant i \leqslant n}\mathcal{G}_i$ is a partition of $\mathscr{R}_N$ (definition \ref{DefPartition}), then the adjectives $A_k$ are known as a \textbf{partition} of the noun $N$.
\end{definition}

By problem \ref{76}, \textit{young}, \textit{middle-aged} and \textit{old} form a partition of \textit{age} (for the human). The following examples are similar.

\begin{example}	\label{56}
\end{example}
\begin{enumerate}	
	\item \textit{Length} can be partitioned by \textit{short}, \textit{medium} and \textit{long}.
	
	
	\item \textit{Height} (for objects) can be partitioned by \textit{low}, \textit{average} and \textit{high}.
	
	\item \textit{Weight} can be partitioned by \textit{light}, \textit{medium} and \textit{heavy}.
	
	\item \textit{Size} can be partitioned by \textit{small}, \textit{medium} and \textit{big}.
	
	\item \textit{Speed} can be partitioned by \textit{slow}, \textit{average} and \textit{fast}.
	
	\item \textit{Sound} can be partitioned by \textit{quiet}, \textit{average} and \textit{loud}.	
	
	\item \textit{Luminance} can be partitioned by \textit{dark}, \textit{medium} and \textit{bright}.
	
	\item \textit{Amount} can be partitioned by \textit{few}, \textit{several}, \textit{average}, \textit{many} and \textit{most}.
	
	\item \textit{Temperature} can be partitioned by \textit{cold}, \textit{cool}, \textit{moderate}, \textit{warm} and
	\textit{hot}.
\end{enumerate}

Next, we will study similarity relation that is described by the adjective \textit{similar} (or \textit{similar to}). It also includes relations depicted by \textit{close}, \textit{approximate}, \textit{analogous} and so on. \medskip

Obviously, similarity relation generally can not be described by a formal predicate because there is no clear-cut value in a similarity (belief) measure. For example, let $x,y\in\Bbb{R}$ and $\psi(x,y) \rightleftharpoons$ \textquotedblleft$x$ is close to $y$\textquotedblright. Clearly, $T(\psi)=1$ if $x=y$, and $T(\psi)$ is inversely proportional to $|x-y|$, i.e. the larger the distance between $x$ and $y$, the less likely they are close to each other. This example shows that the notion of a similarity distance is essential for a similarity measure.

\begin{definition}
	\label{DefSimilarMeasure}\ Suppose $\mathcal{E}$ is a universe of discourse and $A,B\in \mathcal{E}$. Then a function $d\colon \mathcal{E}\times \mathcal{E}\to\Bbb{R}$ is known as a \textbf{similarity distance} if it satisfies the axioms for a metric space. Let $\psi(A, B)\rightleftharpoons$ \textquotedblleft$A$ is similar to $B$\textquotedblright. Then $\psi$ is a predicate in $\mathfrak{M}$ with $\widetilde{\mathcal{S}}(\psi) = \left(\{d\}, \{\left.\mathcal{G}\right\vert_{r}\}, \subset\right)$ ($\left.\mathcal{G}\right\vert_{r}$ is given in theorem \ref{70}(i)). A \textbf{similarity belief measure} for $A$ and $B$ is defined to be $T(\psi(A,B))$ that is inversely proportional to $d(A,B)$. 
\end{definition}

\begin{definition}
	\label{DefCloseness}\ Suppose $x,y\in \Bbb{R}$. Then the \textbf{closeness} of $x$ and $y$ is defined to be the similarity measure $T(\psi(x,y))$ with the similarity distance being $|x-y|$ as in definition \ref{DefSimilarMeasure}. This also applies to linguistic expressions such as \textquotedblleft$\,x$ is close to $y$\textquotedblright, \textquotedblleft$\,x$ is an approximation of $y$\textquotedblright, \textquotedblleft$\,x$ is approaching $y$\textquotedblright, \textquotedblleft the difference between $x$ and $y$ is small\textquotedblright, and so on.
\end{definition}

\begin{problem}
	\ Suppose $\psi(x,y)\rightleftharpoons$\textquotedblleft$\,x$ is close to $y$\textquotedblright \ $(x,y\in\Bbb{R})$ and $\,\widetilde{\mathcal{S}}(\psi(x,y))=\left(\{|x-y|\},\left.\mathcal{G}\right\vert_r,\subset \right)$ where $\left.\mathcal{G}\right\vert_r=\,\{I_{t}\colon I_{t}=[0,\alpha(t)]\,\wedge \,\alpha(t)=t\,\wedge \,t\in[0, \infty)\,\wedge \,f(t)=2t\,e^{-t^2}\,\}$. Find $T(\psi(x,y))$ and calculate for $|x-y|\,=\,0,\,0.1,\,0.5$.
\end{problem}

\begin{proof}
	[Solution] \ Let $a=0,\:b=1$ and $n=1$ in corollary \ref{59}. Then
	\[
	T(\psi(x,y))\,=\,T(|x-y|\sbin\mathcal{G}) \,=\,e^{-|x-y|^2}
	\]
	If $|x-y|\,=\,0,\:T(\psi(x,y))\,=\,1$, which means that $x$ is \textit{identical to} $y$. For $|x-y|\,=\,0.1$ and $0.5$, $T(\psi(x,y))\,=\,e^{-0.01}\,=\,0.99\,$ and $\,T(\psi(x,y))\,=\,e^{-0.25} \,=\,0.78$. \bigskip
\end{proof}

Next, we will investigate inexact circles by similarity measure. In \cite[p214]{Kosko}, Kosko presents an inexact oval as a fuzzy ellipse to dispute a Lindley's claim \cite{Lindley}: \textit{Probability is the only sensible description of uncertainty and is adequate for all problems involving uncertainty. All other methods are inadequate}.\smallskip

Lindley's claim is wrong because not all uncertainty issues can be handled by probability. However, the fuzzy ellipse by Kosko is inadequate because it lacks mathematical rigor to define the exact degree of elliptical roundness.\smallskip

Since \textquotedblleft$x$ is an inexact circle\textquotedblright \ is equivalent to \textquotedblleft$x$ is similar to an (exact) circle\textquotedblright,\footnote{For simplicity, we only consider inexact circles in this paper.} we can give a precise definition for inexact circles based on similarity distance and measure.

\begin{definition}
	\label{DefInexacCircle}\ Suppose $C$ is an inexact circle (Figure \ref{Fig11}(a)) with $\left\vert K_{c}\right\vert \leqslant \varepsilon$ ($K_{c}$ is the maximum curvature of $C$)\footnote{This condition is to remove inexact circles with large curvature as shown in Figure \ref{Fig11}(b).} and $\mathbf{o}$ is the center of $C$. Then the similarity distance is the \textbf{relative maximum deviation} of $C$ as:
	\[
	d_c\:=\,\dfrac{\,\sup\limits_{\mathbf{z}\in C}\left\vert\mathbf{z} - \mathbf{o}\right\vert -\inf\limits_{\mathbf{z}\in C} \left\vert\mathbf{z} -\mathbf{o}\right\vert}{\sup\limits_{\mathbf{z}\in C}\left\vert\mathbf{z} - \mathbf{o}\right\vert}
	\]
	Let $\phi(d_c)\rightleftharpoons\,$\textquotedblleft $\,C$ is an inexact circle\textquotedblright\ (or \textquotedblleft $\,C$ is round\textquotedblright)\ and $\,\widetilde{\mathcal{S}}(\phi(d_c))= \left(\{d_c\},\left.\mathcal{G}\right\vert_{r},\subset \right)$ where $\left.\mathcal{G}\right\vert_{r}$ is given in theorem \ref{70}(i) with $c=0$. Then the \textbf{degree of roundness} (similarity measure) of $C$ is defined to be $T(\phi\left(d_{c}\right))$.
\end{definition}

\begin{figure}[h]
	\centering
	\includegraphics{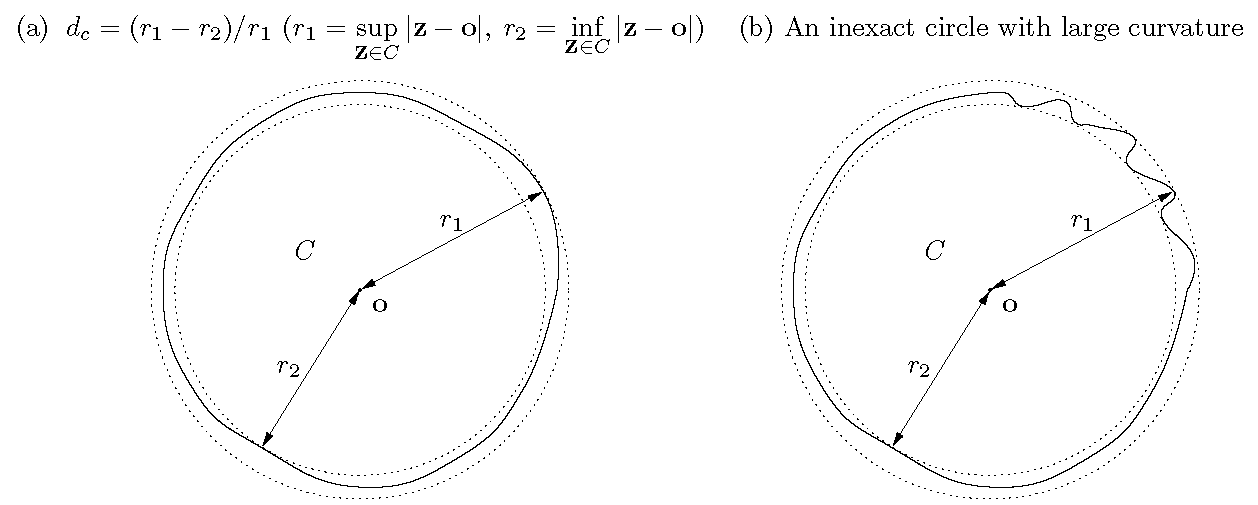}
	\caption{Diagrams of inexact circles.}
	\label{Fig11}
	\centering
\end{figure}

Obviously, if $d_{c}=0$, $C$ is an exact circle (one hundred percent round). If $d_{c}\neq 0$, $C$ has a degree of roundness between zero and one with a bigger $d_{c}$ having a less degree of roundness. Since the degree of roundness is distinctively different from the probability measure, probability is not the only description of uncertainty. \medskip

The adjective \textit{possible} and \textit{necessary} are of particular importance and studied by a separate field known as modal logic. We will devote a separate section (section \ref{SectionModalLogic}) to investigating it. 

\subsection{Adverbs} \label{SectionAdverbs}
Adverbs can be modifiers of verbs or adjectives. Due to the diversity of adverbs, we will only investigate the modeling of adverbs for a few (important) examples in this section. \smallskip

An adverb modifying a verb can often be reduced to an adjective modifying a noun. For instance, \textquotedblleft he is running quickly\textquotedblright\ is equivalent to \textquotedblleft his running speed is fast\textquotedblright. Here are more examples.

\begin{example}
\end{example}
\begin{enumerate}
	\item \textquotedblleft He is speaking loudly\textquotedblright\ is equivalent to \textquotedblleft his sound level is high\textquotedblright.
	
	\item \textquotedblleft He finished the job quickly\textquotedblright\ is equivalent to \textquotedblleft his time of completing the job was short\textquotedblright.
	
	\item \textquotedblleft He drove the car slowly\textquotedblright\ is equivalent to \textquotedblleft his driving speed was slow\textquotedblright.
	
	\item \textquotedblleft He often visits us\textquotedblright\ is equivalent to \textquotedblleft the number of times he visits us is high\textquotedblright.
	
\end{enumerate}

Therefore, this type of adverb can be modeled in the same way as adjectives discussed in the previous section (example \ref{56}). Next, we will study adverbs that are modifiers of adjectives such as \textit{very} and \textit{much}. \medskip

The adverb \textit{very} can strengthen the adjective it modifies. For examples, an age of $20$ has one hundred percent belief of \textit{young} but can have only fifty percent belief of \textit{very young}; a height of six feet is one hundred percent of \textit{tall} (for men) but may be only seventy percent of \textit{very tall}, and so on. The contraction operation on set bundles provides a rigorous model for \textit{very}. 

\begin{definition}
	\label{DefVery} \ Suppose A is an adjective, $\phi(x)$ is a predicate in $\mathfrak{M}$ involving A and $\,\widetilde{\mathcal{S}}(\phi(x)) = \left(\{x\}, \left.\mathcal{G}\right\vert_{r},\subset \right)$ ($\left.\mathcal{G}\right \vert_{r}$ is given in theorem \ref{70}(i) or (ii)). Then the adverb \textbf{very} in a predicate $\eta(x)$ involving A can be modeled by a contraction bundle of $\mathcal{G}$ (definition \ref{DefContraction}), i.e. $\widetilde{\mathcal{S}}(\eta(x)) = \left(\{x\}, f(\mathcal{G}), \subset\right)$.
\end{definition} 

\begin{problem}
	\ Suppose $\eta_1(x)\rightleftharpoons $\textquotedblleft $\,x$ is very young\textquotedblright\ and $\,\widetilde{\mathcal{S}}(\eta_1(x)) = \left(\{x\},\left.\mathcal{H}_1\right\vert_{r_1},\subset \right)$, $\eta_2(x)\rightleftharpoons $\textquotedblleft $\,x$ is very old\textquotedblright\ and $\,\widetilde{\mathcal{S}}(\eta_2(x)) = \left(\{x\},\left.\mathcal{H}_2\right \vert_{r_2},\subset \right)$, where
	\[
	\left.\mathcal{H}_1\right\vert_{r_1}=\,\{I_t\colon I_t=[0,\delta(t)]\,\wedge \,\delta(t)=20t+10\,\wedge \,0\leqslant 	t\leqslant 1\,\wedge \,f_1(t)=0.05\,\}
	\]
	and
	\[
	\left.\mathcal{H}_2\right\vert_{r_2}=\,\{K_{t}\colon K_{t}=[\nu(t),\,150]\,\wedge \,\nu(t) = 20t+60\,\wedge \,0\leqslant t\leqslant 1\,\wedge \,f_2(t)=0.05\,\}	
	\]
	Also, $\mathcal{G}_1$ and $\mathcal{G}_3$ $(\alpha(t)\operatorname{ \textit{and} }\gamma(t))$ are given in problem \ref{76}.
\end{problem} 

\begin{enumerate}
	\item \textit{Show that $\mathcal{H}_1$ is a contraction bundle of $\mathcal{G}_1$ and $\mathcal{H}_2$ a contraction bundle of $\mathcal{G}_3$}. 
	
	\item \textit{Find $\,T(\eta_1(x))$ and $\,T(\eta_2(x))$. Compare the results with problem \ref{76}.}
\end{enumerate}

\begin{proof}[Solution] (i) \ Let $f([0,\alpha(t)])=[0,\delta(t)]\in \mathcal{H}_1$. Since for any $t\in[0,1], \:[0,\delta(t)] \subset [0,\alpha(t)]$, by lemma \ref{54} and definition \ref{DefContraction}, $f$ is a contraction mapping and $\mathcal{H}_1=f(\mathcal{G}_1)$. Likewise, $\mathcal{H}_2$ is a contraction bundle of $\mathcal{G}_3$. (See Figure \ref{Fig12} for more detail.) 
	
	\begin{figure}[h]
		\centering
		\includegraphics{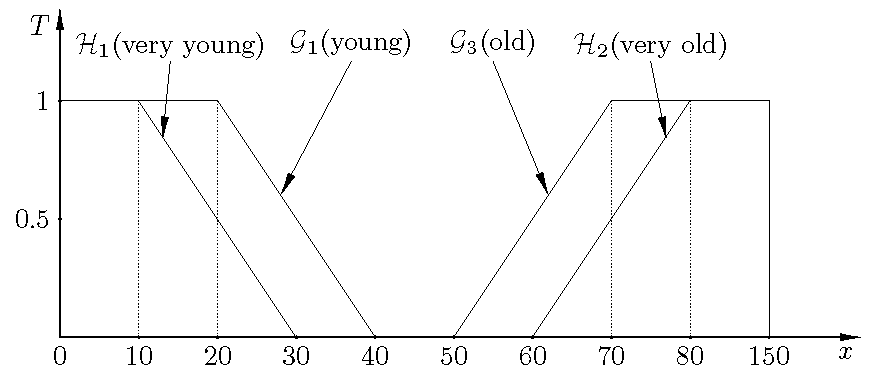}
		\caption{Diagram of belief distributions for adverb \textit{very} on adjectives.}
		\label{Fig12}
		\centering
	\end{figure}
	
	(ii) \ If $x\leqslant10,\:T(\eta_1(x))=1$ for $\left[\eta_1(x)\right]=\mathcal{H}_1$. If $x>30,\:T(\eta_1(x))=0$ for $\left[\eta_1(x)\right]=\varnothing$.\medskip
	
	If $10\leqslant x\leqslant30$, let $\delta(t_x)=x$. Then by corollary \ref{58}(i), $T(\eta_1(x))\,=\,1-t_x$. \medskip
	
	Likewise, if $x\geqslant80,\:T(\eta_2(x))=1$ for $\left[\eta_2(x)\right]=\mathcal{H}_2$. If $x<60,\:T(\eta_2(x))=0$ for $\left[\eta_2(x)\right]=\varnothing$. If $60\leqslant x\leqslant80$, let $\nu(s_x)=x$. Then by corollary \ref{58}(ii), $T(\eta_2(x))\,=\,s_x$. \medskip
	
	If $10\leqslant x\leqslant40$, let $\alpha(t_x')=x$ in problem \ref{76}. Then $t_x>t_x'\,$ and $\,T(\eta_1(x))\,=\,1-t_x <1-t_x'\,=\,T(\phi_1(x))$. \medskip
	
	Likewise, for $50\leqslant x\leqslant80$, let $\gamma(s_x')=x$. Then $s_x<s_x'\,$ and $\,T(\eta_2(x)) < T(\phi_3(x))$.\medskip
	
	By comparing them with results in problem \ref{76}, we can see that \textit{very} can strengthen \textit{young} and \textit{old} by lowering the beliefs. For example, $20$ has 100\% belief of \textit{young} but only 50\% percent of \textit{very young}, while $75$ has 100\% belief of \textit{old} but only 75\% percent of \textit{very old}. This strengthening effect of \textit{very} is consistent with our intuition. \bigskip
\end{proof}

Sentences involving comparative adjectives are generally predicates in formal logic. For example, let $\psi\Leftrightarrow$ \textquotedblleft John is older than Jason\textquotedblright.\ Then $\psi$ has a set representation of $(\{x-y\},(0,\infty),\subset)$ where $x$ and $y$ are ages of John and Jason respectively. So $\psi$ is true iff $x>y$. However, comparative adjectives involving the adverb \textit{much} are generally predicates in the multi-valued logic $\mathfrak{M}$. For instance, \textquotedblleft John is much older than Jason\textquotedblright\ must be a predicate in $\mathfrak{M}$ because there is no clear-cut value of $x-y$ to define \textit{much}. 

\begin{definition}
	\label{DefMuch}\ Suppose A is a comparative adjective and $\psi(x,y)$ is a formal predicate involving A ($x$ and $y$ are in the range of the subject and object in $\psi$). Then the addition of the adverb \textbf{much} to $\psi$ as a modifier of A forms a predicate $\xi(z)$ in $\mathfrak{M}$ that can be modeled by $\,\widetilde{\mathcal{S}}(\xi(z)) = \left(\{z\}, \left.\mathcal{G}\right\vert_{r}, \subset\right)$ where $z=x-y$ or $z=y-x$.
\end{definition}

\begin{problem}
	\label{55}\ Suppose $\xi_1(z)\rightleftharpoons $\textquotedblleft $\,x$ is much older than $y$\textquotedblright\ and $\,\widetilde{\mathcal{S}}(\xi_1(z)) = \left(\{z\},\left.\mathcal{G}\right \vert_{r},\subset \right)\:(z=x-y)$, $\xi_2(z)\rightleftharpoons $\textquotedblleft $\,x$ is much younger than $y$\textquotedblright\ and $\,\widetilde{\mathcal{S}}(\xi_2(z)) = \left(\{z\},\left.\mathcal{G}\right \vert_{r},\subset \right)\:(z=y-x)$, where
	\[
	\left.\mathcal{G}\right\vert_{r}=\,\{I_t\colon I_t=[t,150] \,\wedge \,t\in[0,10]\,\wedge \,f(t)=4b\left(10-t\right)^3 e^{-b(10-t)^4}+e^{-10}/{10}\,\wedge\,b=10^{-3}\,\}
	\]
	Find $\,T(\xi_1(z))$ and calculate for $z=3,\,6,\,10$, and $\,T(\xi_2(z))$ for $z=-2,\,8,\,15\:(\text{years})$.
\end{problem}

\begin{proof}
	[Solution] \ For $z> 10$, by definition \ref{DefWeightDensityFunction}
	\[
	T(\xi_1(z))\,=\int_0^{z}f(t)\,dt\,=\int_0^{10}f(t)\,dt\,=\left.\left(e^{-b(t-10)^4}+\dfrac{e^{-10}t}{10}\right) \right\vert_0^{10}\,=\,1
	\]
	
	For $0\leqslant z\leqslant 10$, let $z=t_z$. By theorem \ref{70}(ii)
	\[
	T(\xi_1(z))\,=\int_0^{z}f(t)\,dt\,=\left.\left(e^{-b(t-10)^4}+\dfrac{e^{-10}t}{10}\right)\right\vert_0^{z}\,=\, e^{-b(z-10)^4}+e^{-10}(0.1z-1)
	\]
	
	Therefore, $T(\xi_1(10))\,=\,1,\:T(\xi_1(6))\,=\,0.77,\:T(\xi_1(3))\,=\,0.09$.\medskip
	
	Likewise, $T(\xi_2(z))\,=\,T(\xi_1(z))$. So $T(\xi_2(8))=0.98\,$  and  $\,T(\xi_2(15))=1$. \medskip
	
	For $y-x=-2,\, \left[\xi_2(-2)\right] = \varnothing$ and $T(\xi_2(-2))=0$. This means that \textquotedblleft $x$ is much young than $y$\textquotedblright\ is completely false because $x$ is actually older than $y$. \smallskip
\end{proof}

\begin{remark}
	\ From problem \ref{55}, we can see that two opposite comparative adjectives can have the same set bundle that is different from the one for each individual adjective. For example, $\mathcal{G}$ in problem \ref{55} is different from $\mathcal{G}_1$ (for \textquotedblleft young\textquotedblright) and $\mathcal{G}_3$ (for \textquotedblleft old\textquotedblright) in problem \ref{76}.
\end{remark}

\section{More Applications}

In this section, we will apply predicate bundles to the investigation of some important fields and problems in non-classical logic.

\subsection{Modal Logic} \label{SectionModalLogic}

Modal logic is a formal system that studies statements involving \textit{necessary} and \textit{possible} (and synonyms) \cite{BlackburnP} in a natural language. Since necessity and possibility belong to non-classical logic, multi-valued logic and truth degree are essential for them. We will see that the theory of predicate bundles provides an ideal framework on possibility and necessity that allows us to gain more perspective on the nature of the pair. First, let's look at the following example. \smallskip

Let $p\Leftrightarrow\,$\textquotedblleft he is in the office\textquotedblright. Obviously, $p$\ is a formal predicate that is equivalent to $``x\in A_1"$, where $x$ is the position (a GPS coordinate) of \textquotedblleft he\textquotedblright\ and $A_1$ is the location of the office (a collection of GPS coordinates). Since \textquotedblleft he is possibly in the office\textquotedblright\ means there are other locations he could be in, $p$ needs to be expanded to a predicate $\overline{p}$ in $\mathfrak{M}$ where $\widetilde{\mathcal{P}}(\overline{p}) =(\left.\mathscr{P}\right\vert_{r}, \Bbb{R}^3)$ and $\left.\mathscr{P}\right\vert_{r}=\{p_n\colon p_n\Leftrightarrow ``u\in A_n"\wedge \,r(p_n)\,\wedge\, n\geqslant1\}$ ($A_2, A_3,\cdots$ are other locations). Consequently, we can define $\Diamond p$ (meaning $p$ is possibly true) to be a formal predicate that $T(\overline{p})>0$ ($T(\overline{p})=0$ implies $p$ is necessarily false). Furthermore, we can define $\square p$ (meaning $p$ is necessarily true) to be a formal predicate that $T(\overline{p})=1$. Since a person can only have one position at a time, there is only one $p_n$ holds in $\mathscr{P}$, say $p_k$ is true. So $r(p_k)=1$ and $r(p_n)=0$ for any other $n$. Obviously, if $p$ is true, then $T(\overline{p})=1$ or $p$ is necessarily true (proposition \ref{30}(i)). \smallskip

Let's take a look at another example. Suppose $\phi_1\rightleftharpoons\,$\textquotedblleft he is young\textquotedblright\ (problem \ref{16}). Obviously, \textquotedblleft he is possibly young\textquotedblright\ and \textquotedblleft he is necessarily young\textquotedblright\ both need $\phi_1$ to be expanded, say to a $\overline{\phi}_1$ where $\widetilde{\mathcal{P}}\left(\overline{\phi}_1\right)= \left(\left.\left( \mathscr{P}_{1}\cup \mathscr{P}'\right)\right\vert_{r}, [0,150]\right)$ and $\mathscr{P}'$ is a predicate bundle for adjectives other than \textit{young}. For instance, since he could be either young or old, $\mathscr{P}'$ should include that of \textit{old}. Thus $\Diamond \phi_1$ is true if $T(\overline{\phi}_1)>0$ for otherwise he is necessarily old, and $\square \phi_1$ is true if $T(\overline{\phi}_1)=1$ for otherwise he is possibly young. \smallskip

Since a formal predicate can be considered as a predicate in $\mathfrak{M}$ whose predicate bundle consists of only one predicate (itself) with the weight of one (corollary \ref{28}), we have the following definitions.

\begin{definition}
	\label{DefNecessity} \ Suppose $p\in\mathfrak{M}$ and $\widetilde{\mathcal{P}}(p) = \left(\left.\mathscr{P}\right\vert_{r}, \mathscr{D}\right)$. Let $\overline{p}$ be an expansion of $p$ in $\mathfrak{M}$, i.e. $\widetilde{\mathcal{P}}(\overline{p})=\left(\left.\overline{\mathscr{P}} \right\vert_{r}, \mathscr{D}\right)$ and $\mathscr{P}\subset\overline{\mathscr{P}}$.\footnote{In this section, we assume $p$ and $q$ as well as $\overline{p}$ and $\overline{q}$ are in $\mathfrak{M}$ unless further specified.} Then $\square p$ that means $p$ is \textbf{necessarily true} (or it is necessary that $p$ is true) is a formal predicate and $\square p\,\Leftrightarrow\left(T(\overline{p})=1\right)$.
\end{definition}

\begin{definition}
	\label{DefPossibility} \ Suppose everything is the same as in definition \ref{DefNecessity}. Then $\Diamond p$ that means $p$ is \textbf{possibly true} (or it is possible that $p$ is true) is a formal predicate  and $\Diamond p\, \Leftrightarrow \left(T(\overline{p})>0\right)$.
\end{definition}

\begin{lemma}
	\label{46}\qquad
\end{lemma}

\begin{enumerate}
	\item \textit{$\square\urcorner p${\ means }$p$\ is necessarily false (or $\urcorner p$ is necessarily true) and $\square\urcorner p\Leftrightarrow\left(T(\overline{p})=0\right)$.}
	
	\item \textit{$\Diamond\urcorner p$ means $p$ is possibly false (or $\urcorner p$ is possibly true) and $\Diamond\urcorner p\Leftrightarrow\left(T(\overline{p})<1\right)$.}
\end{enumerate}

\begin{proof}
	\ (i) By definition \ref{DefNecessity} and \ref{DefCentralBundle}, for $x\in\mathscr{D}$
	\[
	\left[\overline{\urcorner p}(x)\right]\,=\,\{u\colon u\in\mathscr{{\overline{P}}}\,\wedge\,\vdash\urcorner u(x)\}\,=\,\left[\overline{p}(x)\right]^c
	\] 
	
	So by corollary \ref{12}, $T(\overline{p})\,=\,1-T(\overline{\urcorner p}) \,=\,0$.\medskip
	
	(ii) follows by definition \ref{DefPossibility} and (i). 
\end{proof}

\begin{proposition}
	\label{48}\qquad
\end{proposition}

\begin{enumerate}
	\item $\square p\,\Leftrightarrow\,\urcorner\Diamond\urcorner p$
	
	\item $\Diamond p\,\Leftrightarrow\,\urcorner\square\urcorner p$
\end{enumerate}

\begin{proof}
	\ (i) \ By definition \ref{DefNecessity}, \ref{DefPossibility} and lemma \ref{46}
	\[
	\square p\,\Longleftrightarrow\,\urcorner\left(T(\overline{p})<1\right)\,\Longleftrightarrow\,\urcorner\Diamond\urcorner p
	\]
	
	(ii) is similar to (i).$\bigskip$
\end{proof}

Now we prove that definition \ref{DefNecessity} and \ref{DefPossibility} are consistent with axioms in modal logic.

\begin{proposition}
	\label{30}\ \ The following axioms in modal logic hold in $\mathfrak{M}$.
\end{proposition}

\begin{enumerate}
	\item $p\,\Rightarrow\square p$
	
	\item $\square\left(p\rightharpoonup q\right) \Rightarrow\left(\square p\Rightarrow\square q\right)$
	
	\item $\square p\,\Rightarrow\Diamond p$
	
	\item $p\,\Rightarrow\square\Diamond p$
	
	\item $\square p\,\Rightarrow\square\square p$
	
	\item $\Diamond p\,\Rightarrow\square\Diamond p$
\end{enumerate}

\begin{proof}
	\ (i) \ Since $p$ is formal, let $\widetilde{\mathcal{P}}(p)=\left(\left.\mathscr{P}\right\vert_{r}=\{p\wedge r(p)=1\}, \mathscr{D}\right)$. Then in any expansion of $\left.\mathscr{P}\right\vert_{r}$, $r(p)=1$ and $T(\overline{p})=1$. So (i) follows by definition \ref{DefNecessity}.\medskip
	
	(ii) \ By definition \ref{DefNecessity} and (\ref{DegreeTruthImplication})
	\[
	T(\overline{p}\rightharpoonup \overline{q})=1\:\Longrightarrow\:T(\overline{p}\barwedge\overline{q})= T(\overline{p}) =1\text{ \ \ \ and so \ \ \ } T(\overline{q}) \,\geqslant\, T(\overline{p}\barwedge\overline{q}) \,=\,1
	\]
	
	(iii) \ Since $T(\overline{p})=1$, $T(\overline{p})>0$. So (iii) follows by definition \ref{DefPossibility}.\medskip
	
	(iv) \ By (i) and (iii).\medskip
	
	(v) \ If $\square p$, then $T(\overline{p})=1$. Since $``T(\overline{p})=1\textquotedblright$ is formal, $T(``T(\overline{p})=1\textquotedblright)\,=\,1$.\medskip
	
	(vi) \ If $\Diamond p$, then $T(\overline{p})>0$. Since $``T(\overline{p})>0\textquotedblright$ is formal, $T(``T(\overline{p})>0\textquotedblright)\,=\,1$.\smallskip
\end{proof}

\begin{proposition}
	\qquad
\end{proposition}

\begin{enumerate}
	\item $\square\left(p\barwedge q\right) \Leftrightarrow\left(\square p\wedge\square q\right)$
	
	\item $\Diamond(p\veebar q) \Leftrightarrow\left(\Diamond p\vee\Diamond q\right)$
	
	\item $\left(\square p\vee\square q\right) \,\Rightarrow\,\square\left(p\veebar q\right)$
	
	\item $\Diamond(p\barwedge q) \,\Rightarrow\left(\Diamond p\wedge\Diamond q\right)$
\end{enumerate}

\begin{proof}
	\ (i) \ By definition \ref{DefNecessity}, if $T(\overline{p}\,\barwedge\,\overline{q})\,=\,1$, then
	\[
	T(\overline{p})\,\geqslant\,T(\overline{p}\,\barwedge\,\overline{q})\,=\,1\text{ \ \ \ and \ \ \ }T(\overline{q})\,\geqslant\,T(\overline{p}\, \barwedge\,\overline{q})\,=\,1
	\]
	
	Conversely, if $T(\overline{p})=1\,$ and $\,T(\overline{q})=1,\,T(\urcorner\,\overline{p})=0\,$ and $\,T(\urcorner\,\overline{q})=0$. So by theorem \ref{34} and \ref{44}(vi)
	\[
	T(\urcorner\,\overline{p}\,\,\veebar\,\urcorner\,\overline{q})\,\leqslant\,T(\urcorner\,\overline{p})+T(\urcorner\,\overline{q})\,=\,0\text{ \ \ \ and
		\ \ \ }T(\overline{p}\,\barwedge\,\overline{q})\,=\,1-T(\urcorner\,\overline{p}\,\,\veebar\,\urcorner\,\overline{q})\,=\,1
	\]
	
	(ii) \ By proposition \ref{48} and (i). \medskip
	
	(iii) \ Since $T(\overline{p})\leqslant T(\overline{p}\,\veebar\,\overline{q})\,$ and $\,T(\overline{q})\leqslant T(\overline{p}\,\veebar\,\overline{q})$, it follows by definition \ref{DefNecessity}. \medskip
	
	(iv) \ Since $T(\overline{p}\,\barwedge\,\overline{q})\leqslant T(\overline{p})\,$ and $\,T(\overline{p}\,\barwedge\, \overline{q}) \leqslant T(\overline{q})$, it follows by definition \ref{DefPossibility}. \smallskip
\end{proof}

\begin{proposition}
	\label{53}\qquad
\end{proposition}

\begin{enumerate}
	\item $\square p\,\not \Rightarrow p$
	
	\item $\square\left(p\veebar q\right) \not \Rightarrow \left(\square p\vee\square q\right)$
	
	\item $\left(\Diamond p\wedge\Diamond q\right) \not \Rightarrow \Diamond(p\barwedge q)$
\end{enumerate}

\begin{proof}
	\ (i) \ Since $\square p$ only means $T(\overline{p})=1$, it can not imply that $p$ is a formal predicate. So (i) fails in $\mathfrak{M}$. \medskip
	
	(ii) \ Obviously, $T(\overline{p}\,\veebar\,\overline{q})\,=\,1\,$ can not imply $\,T(\overline{p})=1\,$ or $\,T(\overline{q})=1$. So (ii) fails in $\mathfrak{M}$. \medskip
	
	(iii) \ Obviously, neither $\,T(\overline{p})>0\,$ nor $\,T(\overline{q})>0\,$ can imply $T(\overline{p}\,\barwedge\, \overline{q})>0$. So (iii) fails in $\mathfrak{M}$. \smallskip
\end{proof}

\subsection{Solution to Sorites Paradox}

The sorites paradox, also known as the heap paradox, originated in an ancient Greek puzzle that creates a contradiction in describing the vagueness of predicates. The paradox can be constructed upon an argument (with premises) as follows.

\begingroup\renewcommand{\labelenumi}{(\alph{enumi})}

\begin{enumerate}
	\item (A large integer) $n$ grains of sand is a heap of sand.
	
	\item $n-1$ grains of sand is still a heap.
	
	\item So by repeating the above step, $1$ grain is a heap.
\end{enumerate}

For the formal structure of this argument, let $\phi(n)$ be the predicate \textquotedblleft$n$ grains of sand is a heap\textquotedblright. Then the above argument can be written as follows. (Assume $m$ is a large integer.)

\begin{enumerate}
	\item $\phi(m)\ $is true.
	
	\item For any $2\leqslant n\leqslant m$, $\phi(n)\Rightarrow\phi(n-1)$ is true.
	
	\item Then $\phi(1)$ is true.
\end{enumerate}

\endgroup

Since $\phi(m)$ is clearly true and $\phi(1)$ clearly false, a key to solving the heap paradox is to find out when the truth value of $\phi(n)$ changes. But since there is no clear-cut value of $n$ that $\phi(n)$ turns from true to false, the situation is very similar to the adjective \textit{young} discussed in section \ref{SectionIntroduction} where a predicate bundle and degree of truth come into effect. More specifically, the paradox can be resolved if we consider $\phi(n)$ as a predicate in $\mathfrak{M}$ and its degree of truth decreases as $n$ becomes smaller.

\begin{conclusion}
	\label{38} \ The sorites paradox is solved because there is a (small) $l$ $\left(1\leqslant n\leqslant l\ll m\right)  $ that $\phi(n)$ becomes false.
\end{conclusion}

\begin{proof}
	\ Let a predicate bundle for $\phi(n)$ be
	\[
	\left.\mathscr{P}\right\vert_{r}\,=\,\{p_{k}(z)\colon p_{k}(z)\Leftrightarrow(k+a-1\leqslant z)\,\wedge\,1\leqslant k\leqslant b-a\,\wedge\,r(p_{k})=1/(b-a) \,\land\, a,b\in \Bbb{N}\,\land\,b\gg a\,\}
	\]
	
	Obviously, $a\leqslant k+a-1\leqslant b-1$. So $T(\phi(n))=1$ for $n\geqslant b-1$ because $n$ satisfies all $p_{k}$ in $\mathscr{P}$, and $T(\phi(n))=0$ for $n<a$ because $n$ satisfies none of $p_{k}$ in $\mathscr{P}$. For $a\leqslant n<b-1,$ since for any $k\leqslant n-a+1$, $\,\vdash p_k(n)$, $\,\left[\phi(n)\right]\,=\,\{p_{k}\colon1\leqslant k\leqslant n-a+1\}$. Thus
	\begin{align*}
		T(\phi(n))  & \:=\, \frac{n-a+1}{b-a},\text{ \ \ \ \ }a\leqslant n<b-1
		\\
		& \:=\: 1,\qquad\qquad\ \ \ \ n\geqslant b-1
		\\
		& \:=\: 0,\qquad\qquad\ \ \ \ n<a
	\end{align*}
	
	Since $\left[\phi(n-1)\right]=\{p_{k}\colon1\leqslant k\leqslant n-a\}$, $\left[\phi(n-1)\right] \subset\left[\phi(n)\right]$. So from remark \ref{3}, $T(\phi(n)\, \barwedge\,\phi(n-1))\,=\,T(\phi(n-1))$. Also, for $n\geqslant b$, $T(\phi(n-1))=1$ and $T(\phi(n)\,\barwedge\,\phi(n-1))=1$. Thus by (\ref{DegreeTruthImplication})
	\begin{align*}
		T(\phi(n)\rightharpoonup\phi(n-1))  \,=\, \dfrac{T(\phi(n-1))}{T(\phi(n))}& \:=\, \frac{n-a}{n-a+1},\text{\ \ \ \ \ }a\leqslant n\leqslant b-1
		\\
		& \:=\: 1,\qquad\qquad\ \ \ \ n\geqslant b
		\\
		& \:=\: 1,\qquad\qquad\ \ \ \ n<a
	\end{align*}
	
	For $n\geqslant b,$ since $T(\phi(n))=1$ and $T(\phi(n)\rightharpoonup \phi(n-1))=1$, it is reduced to propositional logic and by modus ponens, $\phi(n-1)$ follows from $\phi(n)$. If $n<b$ and $n\gg a$,
	\[
	T(\phi(n))\approx1\text{ \ and \ }T(\phi(n)\rightharpoonup\phi(n-1))\approx1\text{. \ So \ }T(\phi(n-1))\,=\,T(\phi(n)\,\barwedge \,\phi(n-1))\,T(\phi(n))\,\approx\,1
	\]
	
	This means that we can almost imply $\phi(n-1)$ from $\phi(n)$. But when $n$ is close to $a$, $T(\phi(n))\approx0$ and $T(\phi(n)\rightharpoonup\phi(n-1))<1$. So $T(\phi(n))\approx0$ implies $T(\phi(n-1))\approx0$. Also, $T(\phi(n))=0$ when $n<a$. Thus $\phi(1)$ is false and the sorites paradox is solved. In other words, it is possible for \textquotedblleft$1$ grain is a heap\textquotedblright\ to be false despite \textquotedblleft$1$ million grains is a heap\textquotedblright\ is true because the deduction process fails to be true at a certain stage.
\end{proof}

\section{Conclusion}

In this section, we will compare the multi-valued logic $\mathfrak{M}$ with other (multi-valued) logics.

\subsection{Predicate Bundles and Other Logics}

First, we confirm the consistency of $\mathfrak{M}$ with propositional logic by showing that all the logical operations defined in $\mathfrak{M}$ are reduced to those in propositional logic if predicate bundles contain only one predicate.

\begin{corollary}
	\label{28}\ Suppose $\,\widetilde{\mathcal{P}}(\phi_{1})=\left(\{p\wedge r_{1}(p)=1\},\mathscr{D}_1\right)$ and $\,\widetilde{\mathcal{P}}(\phi_{2})= \left(\{q\wedge r_{2}(q)=1\},\mathscr{D}_2\right)$. Then $\urcorner\phi_{1}$, $\phi_{1}\barwedge\phi_{2}$, $\phi_{1}\veebar\phi_{2}$, $\phi_{1}\rightharpoonup\phi_{2}$, $\phi_{1}\rightleftharpoons\phi_{2}$ are reduced to those of propositional logic.
\end{corollary}

\begin{proof}
	\ Modus ponens is proved in corollary \ref{41}. Since $\left[\phi_{1}\right]=\{p\}$ or $\varnothing$, $T(\phi_{1})=1$ or $0.$ By corollary \ref{12}, $T(\urcorner\phi_{1})=1-T(\phi_{1})$.\medskip
	
 	By corollary \ref{19}, $\left[\phi_{1}\barwedge\phi_{2}\right] \,=\,\left[\phi_{1}\right] \times\left[\phi_{2}\right]$. So we have
	\begin{align*}
		T(\phi_{1}\barwedge\phi_{2})\,=\,1  & \:\Longleftrightarrow\: \left(T(\phi_{1})=1\,\wedge\,T(\phi_{2})=1\right) 
		\\
		T(\phi_{1}\barwedge\phi_{2})\,=\,0  & \:\Longleftrightarrow\: \left(T(\phi_{1})=0\,\vee\,T(\phi_{2})=0\right)
	\end{align*}
	
	Also, $\left[\phi_{1}\veebar\phi_{2}\right] \,=\,\left[\phi_{1}\right] \times\{q\}\cup\{p\}\times\left[\phi_{2}\right]$. So we have
	\begin{align*}
		T(\phi_{1}\veebar\phi_{2})\,=\,1  & \:\Longleftrightarrow\: \left(T(\phi_{1})=1\,\vee\,T(\phi_{2})=1\right) 
		\\
		T(\phi_{1}\veebar\phi_{2})\,=\,0  & \:\Longleftrightarrow\: \left(T(\phi_{1})=0\,\wedge\,T(\phi_{2})=0\right)
	\end{align*}
	
	If $T(\phi_{1}\rightharpoonup\phi_{2})=1$ and $T(\phi_{1})\neq0$, then $T(\phi_{1}\barwedge\phi_{2})=1$ and so $T(\phi_{2})=1$. Conversely, if
	$T(\phi_{2})=1$ and $T(\phi_{1})\neq0$, then $T(\phi_{1}\barwedge\phi_{2})=1$ and so $T(\phi_{1}\rightharpoonup\phi_{2})=1$. Thus
	\[
	T(\phi_{1}\rightharpoonup\phi_{2})=1\:\Longleftrightarrow\:\left(T(\phi_{1})=0\,\vee\,T(\phi_{2})=1\right)
	\]
	
	Finally, by definition \ref{DefEquivalenceMultiValueLogic}
	\[
	T(\phi_{1}\rightleftharpoons\phi_{2})=1\,\,\Longleftrightarrow\,\,\left(T(\phi_{1})=T(\phi_{2})=1\,\vee\,T(\phi_{1})=T(\phi_{2})=0\right)
	\]
	
	Since the above results are consistent with rules of propositional calculus, these operations in $\mathfrak{M}$\ are reduced to those of propositional logic. \bigskip
\end{proof}

Corollary \ref{28} and \ref{23} suggest that predicate bundles can generalize both propositional logic and probability spaces. Consequently, the theory of predicate bundles has potential to unify both formal logic and probability theory, a topic that we will discuss in more detail in other papers.\medskip

In Lukasiewicz logic \cite{Gottwald}, the negation rule $\left(\urcorner u:=1-u\right)$ is the same as corollary \ref{12}, but the implication rule $\left(u\rightarrow v:=\,\min\{1,1-u+v\}\right)$ is inconsistent with our model ((\ref{DegreeTruthImplication}) and (\ref{DegreeTruthImplicationSentence})). From corollary \ref{15} (and \ref{20}), we can see that $\min\{T(\phi_{1}), T(\phi_{2})\}$ and $\max\{T(\phi_{1}), T(\phi_{2})\}$ are only bounds of $T(\phi_{1}\barwedge\phi_{2})$ and $T(\phi_{1}\veebar\phi_{2})$ rather than their exact values. (The relation between $T(\phi_{1}\barwedge\phi_{2})$ and $T(\phi_{1}\veebar\phi_{2})$ is given in corollary \ref{25}.) As the result, the min-max rule $(u\wedge v:=\,\min\{u,v\}$, $u\vee v:=\,\max\{u,v\})$ in the logics of Post and Godel is also inconsistent with $\mathfrak{M}$. The conjunction rule $\left(u\wedge v:=\,uv\right)$ in product logic corresponds to the irrelevancy case in $\mathfrak{M}$ (lemma \ref{18}). \medskip

The rigorous solution to the sorites paradox (conclusion \ref{38}) also provides a confirmation of the implication formula in $\mathfrak{M}$ (\ref{DegreeTruthImplication}). Unlike formal logic, there are well-formed formulas that do not have truth degrees in $\mathfrak{M}$ (corollary \ref{21}). In addition, even some of the basic tautologies in propositional logic can fail in $\mathfrak{M}$ (remark \ref{33}).

\subsection{Set Bundles and Fuzzy Sets}\label{SectionSetBundleFuzzySet}

The notion of fuzzy sets was introduced by Zadeh in 1965 \cite{Zadeh}. Since then, fuzzy set theory has blossomed into a popular field widely used by researchers in areas such as linguistic modeling, soft computing, approximate reasoning, imprecision and uncertainty in artificial intelligence, and so on. Initially, a fuzzy set is introduced to extend the classical notion of a set by assigning each of its elements with a membership degree between 0 and 1. Later on, axiomatization and $t-$norm based systems are proposed in attempts to put fuzzy logic on a solid footing (\cite{Dubois-Prade} and \cite{Hajek}). \medskip

The problem with a fuzzy set is that its fuzzy membership lacks a rigorous definition and is often under subjectivity and speculation. However, since uncertainty problems tackled by fuzzy set theory are often important, a rigorous foundation for fuzzy sets is desired upon which these problems can be reinvestigated. We can see notions of predicate and set bundles provide the rigorous foundation for fuzzy sets because membership of a set bundle provides a precise definition for the fuzzy membership of a fuzzy set.\medskip

First, let's compare membership of a set bundle with the fuzzy membership of a fuzzy set. In \cite{Zadeh}, a fuzzy set is defined as an extension of a classical set by assigning each of its elements with a membership degree between 0 and 1 (fuzzy membership). The union and intersection of two fuzzy sets are defined as follows. Given two fuzzy sets $A$ and $B$ whose membership functions are $\mu_{A}(x)$ and $\mu_{B}(x)$. Then the membership functions for $A\cap B$ and $A\cup B$ are $\mu_{A\cap B}(x)=\min\{\mu_{A}(x),\,\mu_{B}(x)\}$ and $\mu_{A\cup B}(x)=\max\{\mu_{A}(x),\, \mu_{B}(x)\}$, which is known as the min-max rule for fuzzy sets. Later on, they are extended to $t-$norms as $\mu_{A\cap B}(x)=t_{1}(\mu_{A}(x),\,\mu_{B}(x))$ and $\mu_{A\cup B}(x)=t_{2}(\mu_{A}(x),\,\mu_{B}(x))$, where a $t-$norm is a function $T\colon\left[0,1\right] \times\left[0,1\right]  \rightarrow\left[0,1\right]$ that satisfies such properties as commutativity, monotonicity, associativity and identity. Common $t-$norms include Godel norm, drastic norm, product norm, and so on (\cite{Dubois-Prade} and \cite{Hajek}). The problem of a fuzzy set is that its fuzzy membership lacks a rigorous definition and is often assigned arbitrarily and speculatively. The min-max rule for fuzzy sets is also inaccurate and inconclusive. \medskip

On the other hand, the membership of $x$ in a set bundle $\left.\mathcal{G}\right\vert_{r}$ is rigorously defined (in definition \ref{DefMembershipSetbundle} and lemma \ref{29}) as all sets in the bundle containing $x$. As the result, the membership degree of $x$ can be precisely determined as the truth degree of \textquotedblleft$ x\sbin\mathcal{G}$\textquotedblright. This indicates that $\mathcal{G}$ can be viewed as a rigorous model of a fuzzy set where $T(x\sbin\mathcal{G})$ is the precise definition of fuzzy membership. Furthermore, we can see that the min-max rule for fuzzy sets is inaccurate because corollary \ref{31} shows that $\min\{T(x\sbin \mathcal{C}_{1})$, $T(x\sbin\mathcal{C}_{2})\}$ and $\max\{T(x\sbin\mathcal{C}_{1})$, $T(x\sbin\mathcal{C}_{2})\}$ are only bounds of $T(x\sbin\mathcal{C}_{1}\cap\mathcal{C}_{2})$ and $T(x\sbin\mathcal{C}_{1}\cup\mathcal{C}_{2})$ rather than their exact values. In addition, corollary \ref{49} shows that there are no $t-$norms in general to express $T(x\sbin\mathcal{C}_{1}\cap\mathcal{C}_{2})$ and $T(x\sbin\mathcal{C}_{1} \cup \mathcal{C}_{2})$ based on $T(x\sbin\mathcal{C}_{1})$ and $T(x\sbin\mathcal{C}_{2})$. This fact holds for any membership degree as long as it is defined in measure theory. The only $t-$norm known so far is the product $t-$norm as in corollary \ref{42}. \medskip

Since fuzzy set theory lacks a rigorous foundation, it overproduces many unnecessary, redundant, or even incorrect works. However, uncertainty issues tackled by fuzzy set theory are often important real-world problems that can be reinvestigated rigorously by predicate and set bundles. For example, fuzzy measures are unnecessary since the truth degree of a predicate in $\mathfrak{M}$ can be defined precisely as the weight measure of its central bundle. Fuzzy possibility theory is incorrect and can be replaced by the predicate bundle-based modal logic that provides a rigorous foundation for necessity and possibility (see section \ref{SectionModalLogic}). \medskip

Since this paper is the first paper on predicate and set bundles, we omit many topics of reinvestigation on fuzzy set theory such as fuzzy integration, fuzzy functions, and so on. We will leave these studies in other papers. 


\subsection{Predicate Bundles and Modal Logic}

The predicate bundle-based model of necessity and possibility is consistent with the Kripke system \textbf{K} (also known as the weakest normal modal logic) \cite{BlackburnP}. In fact, all results in \textbf{K}, including the necessitation rule and distribution axiom, can be proved based on predicate bundles and degree of truth (proposition \ref{30}). Our model can also prove the axiom $4$ (proposition \ref{30}(v)) which \textbf{K} can not. However, our model can not prove the reflexivity axiom (proposition \ref{53}(i)) which is consistent with \textbf{K}. In general, our model is a (significant) improvement over \textbf{K} in that it gives a precise possibility measure based on the notion of degree of truth and abandons the plausible many-world theory.\medskip

This paper is the first one on the theory of predicate bundles in which we only discuss the most important topics and omit many others. We will continue to investigate it in other papers.

\textit{E-mail}: e353zhan@uwaterloo.ca

\end{document}